\documentclass{amsart}
\usepackage{amssymb}
\usepackage{amsfonts}

\newtheorem{theorem}{Theorem}
\theoremstyle{plain}

\newtheorem{corollary}{Corollary}

\newtheorem{definition}{Definition}

\newtheorem{lemma}{Lemma}

\newtheorem{proposition}{Proposition}
\newtheorem{remark}{Remark}

\numberwithin{equation}{section}
\input{tcilatex}

\begin{document}
\title[Integro-differential problems]{On the Cauchy problem for
integro-differential equations in the scale of spaces of generalized
smoothness}
\author{R. Mikulevi\v{c}ius and C. Phonsom}
\address{University of Southern California, Los Angeles}
\date{May 20, 2017}
\subjclass{35R09, 60J75, 35B65}
\keywords{non-local parabolic integro-differential equations, L\'{e}vy
processes}

\begin{abstract}
Parabolic integro-differential model Cauchy problem is considered in the
scale of $L_{p}$ -spaces of functions whose regularity is defined by a
scalable Levy measure. Existence and uniqueness of a solution is proved by
deriving apriori estimates. Some rough probability density function
estimates of the associated Levy process are used as well.
\end{abstract}

\maketitle
\tableofcontents

\section{Introduction}

Let $\sigma \in \left( 0,2\right) $ and $\mathfrak{A}^{\sigma }$ be the
class of all nonnegative measures $\pi $\ on $\mathbf{R}_{0}^{d}=\mathbf{R}%
^{d}\backslash \left\{ 0\right\} $ such that $\int \left\vert y\right\vert
^{2}\wedge 1d\pi <\infty $ and 
\begin{equation*}
\sigma =\inf \left\{ \alpha <2:\int_{\left\vert y\right\vert \leq
1}\left\vert y\right\vert ^{\alpha }d\mathfrak{\pi }<\infty \right\} .
\end{equation*}%
In addition, we assume that for $\pi \in \mathfrak{A}^{\sigma },$ 
\begin{eqnarray*}
\int_{\left\vert y\right\vert >1}\left\vert y\right\vert d\pi &<&\infty 
\text{ if }\sigma \in \left( 1,2\right) , \\
\int_{R<\left\vert y\right\vert \leq R^{\prime }}yd\pi &=&0\text{ if }\sigma
=1\text{ for all }0<R<R^{\prime }<\infty .\text{ }
\end{eqnarray*}

In this paper we consider the parabolic Cauchy problem with $\lambda \geq 0$%
\begin{eqnarray}
\partial _{t}u(t,x) &=&Lu(t,x)-\lambda u\left( t,x\right) +f(t,x)\text{ in }%
E=[0,T]\times \mathbf{R}^{d},  \label{1'} \\
u(0,x) &=&g\left( x\right) ,  \notag
\end{eqnarray}%
and integro-differential operator 
\begin{equation*}
L\varphi \left( x\right) =L^{\pi }\varphi \left( x\right) =\int \left[
\varphi (x+y)-\varphi \left( x\right) -\chi _{\sigma }\left( y\right) y\cdot
\nabla \varphi \left( x\right) \right] \pi \left( dy\right) ,\varphi \in
C_{0}^{\infty }\left( \mathbf{R}^{d}\right) ,
\end{equation*}%
where $\pi \in \mathfrak{A}^{\sigma },$ $\chi _{\sigma }\left( y\right) =0$
if $\sigma \in \lbrack 0,1),\chi _{\sigma }\left( y\right) =1_{\left\{
\left\vert y\right\vert \leq 1\right\} }\left( y\right) $ if $\sigma =1$ and 
$\chi _{\sigma }\left( y\right) =1$ if $\sigma \in (1,2).$ The symbol of $L$
is 
\begin{equation*}
\mathfrak{\psi }\left( \xi \right) =\psi ^{\pi }\left( \xi \right) =\int %
\left[ e^{i2\pi \xi \cdot y}-1-i2\pi \chi _{\sigma }\left( y\right) \xi
\cdot y\right] \pi \left( dy\right) ,\xi \in \mathbf{R}^{d}.
\end{equation*}%
Note that $\pi \left( dy\right) =dy/\left\vert y\right\vert ^{d+\sigma }\in 
\mathfrak{A}^{\sigma }$ and, in this case, $L=L^{\pi }=c\left( \sigma
,d\right) \left( -\Delta \right) ^{\sigma /2}$, where $\left( -\Delta
\right) ^{\sigma /2}$\ is a fractional Laplacian. The equation (\ref{1'}) is
backward Kolmogorov equation for the Levy process associated to $\psi ^{\pi }
$. Let $\mu \in \mathfrak{A}^{\sigma }$ and 
\begin{equation}
c_{1}\left\vert \psi ^{\mu }\left( \xi \right) \right\vert \leq |\psi ^{\pi
}\left( \xi \right) |\leq c_{2}\left\vert \psi ^{\mu }\left( \xi \right)
\right\vert ,\xi \in \mathbf{R}^{d},  \label{3'}
\end{equation}%
for some $0<c_{1}\leq c_{2}$. Given $\mu \in \mathfrak{A}^{\sigma },p\in
\lbrack 1,\infty ),s\in \mathbf{R}$, we denote $H_{p}^{s}\left( E\right)
=H_{p}^{\mu ;s}\left( E\right) $ the closure in $L_{p}\left( E\right) $ of $%
C_{0}^{\infty }\left( E\right) $ with respect to the norm 
\begin{equation*}
\left\vert f\right\vert _{H_{p}^{\mu ;s}\left( E\right) }=\left\vert 
\mathcal{F}^{-1}\left( 1-\func{Re}\psi ^{\mu }\right) ^{s}\mathcal{F}%
f\right\vert _{L_{p}\left( \mathbf{R}^{d}\right) },\,
\end{equation*}%
where $\mathcal{F}$ is the Fourier transform in space variable. In this
paper, under certain "scalability" and nondegeneracy assumptions (see
assumptions \textbf{D}$\left( \kappa ,l\right) ,$ \textbf{B}$\left( \kappa
,l\right) $ below), we prove the existence and uniqueness of solutions to (%
\ref{1'}) in $H_{p}^{\mu ;s}\left( \mathbf{R}^{d}\right) $). An apriori
estimate is derived for $u$. For example, if $p\geq 2$ the following
estimate holds:%
\begin{equation}
\left\vert u\right\vert _{H_{p}^{\mu ;s+1}\left( E\right) }\leq C\left[
\left\vert f\right\vert _{H_{p}^{\mu ;s}\left( E\right) }+\left\vert
g\right\vert _{H_{p}^{\mu ;s+1-1/p}\left( \mathbf{R}^{d}\right) }\right] .
\label{4}
\end{equation}%
This paper is a continuation of \cite{MPh}, where (\ref{1'}) with $g=0$ in
the case $s=0$ was considered. Here, we solve (\ref{1'}) in the scale of
spaces $H_{p}^{\mu ;s}$ under slightly different conditions than the ones in 
\cite{MPh}. The symbol $\psi ^{\pi }\left( \xi \right) $ is not smooth in $%
\xi $ and the standard Fourier multiplier results do not apply in this case.
In order to prove (\ref{4}), we apply Calderon-Zygmund theorem by
associating to $L^{\pi }$ a family of balls and verifying H\"{o}rmander
condition (see (\ref{f23}) below) for it. A different \ splitting of the
integral in (\ref{f23}) is used (cf. \cite{kk2}). As an example, we consider 
$\pi \in \mathfrak{A}^{\sigma }$ defined in radial and angular coordinates $%
r=\left\vert y\right\vert ,w=y/r,$ as 
\begin{equation}
\pi \left( \Gamma \right) =\int_{0}^{\infty }\int_{\left\vert w\right\vert
=1}\chi _{\Gamma }\left( rw\right) a\left( r,w\right) j\left( r\right)
r^{d-1}S\left( dw\right) dr,\Gamma \in \mathcal{B}\left( \mathbf{R}%
_{0}^{d}\right) ,  \label{5}
\end{equation}%
where $S\left( dw\right) $ is a finite measure on the unit sphere on $%
\mathbf{R}^{d}$. In \cite{xz}, (\ref{1'}) was considered, with $\pi $ in the
form (\ref{5}) with $a=1,j\left( r\right) =r^{-d-\sigma },$ and such that 
\begin{eqnarray*}
&&\int_{0}^{\infty }\int_{\left\vert w\right\vert =1}\chi _{\Gamma }\left(
rw\right) r^{-1-\sigma }\rho _{0}\left( w\right) S\left( dw\right) dr \\
&\leq &\pi \left( \Gamma \right) =\int_{0}^{\infty }\int_{\left\vert
w\right\vert =1}\chi _{\Gamma }\left( rw\right) r^{-1-\sigma }a\left(
r,w\right) S\left( dw\right) dr \\
&\leq &\int_{0}^{\infty }\int_{\left\vert w\right\vert =1}\chi _{\Gamma
}\left( rw\right) r^{-1-\sigma }S\left( dw\right) dr,\Gamma \in \mathcal{B}%
\left( \mathbf{R}_{0}^{d}\right) ,
\end{eqnarray*}%
and (\ref{3'}) holds with $\psi ^{\mu }\left( \xi \right) =\left\vert \xi
\right\vert ^{\sigma },\xi \in \mathbf{R}^{d}$. In this case, $H_{p}^{\mu
;1}\left( E\right) =H_{p}^{\sigma }\left( E\right) $ is the standard
fractional Sobolev space. The solution estimate (\ref{4}) for (\ref{1'}) was
derived in \cite{xz}, using $L^{\infty }$-$BMO$ type estimate. In \cite{kk},
an elliptic problem in the whole space with $L^{\pi }$ was studied for $\pi $
in the form (\ref{5}) with $S\left( dw\right) =dw$ being a Lebesgue measure
on the unit sphere in $\mathbf{R}^{d}$, with $0<c_{1}\leq a\leq c_{2}$, and
a set of technical assumptions on $j\left( r\right) $. A sharp function
estimate based on the solution H\"{o}lder norm estimate (following the idea
in \cite{KimDong}) was used in \cite{kk}.

The paper is organized as follows. In Section 2, the main theorem is stated,
an example of the form (\ref{5}) considered. In Section 3, we introduce
various equivalent norms of the spaces in which (\ref{1'}) is solved. Note
that for $s>0$ $H_{p}^{\mu ;s}$ are spaces of generalized smoothness (see 
\cite{ka}, \cite{kali} and the references therein). The main theorem is
proved in Section 4.

\section{Notation, function spaces and main results}

\subsection{Notation}

The following notation will be used in the paper.

Let $\mathbf{N}=\{1,2,\ldots \},\mathbf{N}_{0}=\left\{ 0,1,\ldots \right\} ,%
\mathbf{R}_{0}^{d}=\mathbf{R}^{d}\backslash \{0\}.$ If $x,y\in \mathbf{R}%
^{d},$\ we write 
\begin{equation*}
x\cdot y=\sum_{i=1}^{d}x_{i}y_{i},\,|x|=\sqrt{x\cdot x}.
\end{equation*}

We denote by $C_{0}^{\infty }(\mathbf{R}^{d})$ the set of all infinitely
differentiable functions on $\mathbf{R}^{d}$ with compact support.

We denote the partial derivatives in $x$ of a function $u(t,x)$ on $\mathbf{R%
}^{d+1}$ by $\partial _{i}u=\partial u/\partial x_{i}$, $\partial
_{ij}^{2}u=\partial ^{2}u/\partial x_{i}\partial x_{j}$, etc.; $Du=\nabla
u=(\partial _{1}u,\ldots ,\partial _{d}u)$ denotes the gradient of $u$ with
respect to $x$; for a multiindex $\gamma \in \mathbf{N}_{0}^{d}$ we denote%
\begin{equation*}
D_{x}^{{\scriptsize \gamma }}u(t,x)=\frac{\partial ^{|{\scriptsize \gamma |}%
}u(t,x)}{\partial x_{1}^{{\scriptsize \gamma _{1}}}\ldots \partial x_{d}^{%
{\scriptsize \gamma _{d}}}}.
\end{equation*}%
For $\alpha \in (0,2]$ and a function $u(t,x)$ on $\mathbf{R}^{d+1}$, we
write 
\begin{equation*}
\partial ^{{\scriptsize \alpha }}u(t,x)=-\mathcal{F}^{-1}[|\xi |^{%
{\scriptsize \alpha }}\mathcal{F}u(t,\xi )](x),
\end{equation*}%
where 
\begin{equation*}
\mathcal{F}h(t,\xi )=\hat{h}\left( \xi \right) =\int_{\mathbf{R}^{d}}\,%
\mathrm{e}^{-i2\pi \xi \cdot x}h(t,x)dx,\mathcal{F}^{-1}h(t,\xi )=\int_{%
\mathbf{R}^{d}}\,\mathrm{e}^{i2\pi \xi \cdot x}h(t,\xi )d\xi .
\end{equation*}

For $\mu \in \mathfrak{A}^{\sigma }$, we denote $Z_{t}^{\mu },t\geq 0,$ the
Levy process associated to $L^{\mu }$, i.e., $Z^{\mu }$ is cadlag with
independent increments and its characteristic function%
\begin{equation*}
\mathbf{E}e^{i2\pi \xi \cdot Z_{t}^{\mu }}=\exp \left\{ \psi ^{\mu }\left(
\xi \right) t\right\} ,\xi \in \mathbf{R}^{d},t\geq 0.
\end{equation*}%
The letters $C=C(\cdot ,\ldots ,\cdot )$ and $c=c(\cdot ,\ldots ,\cdot )$
denote constants depending only on quantities appearing in parentheses. In a
given context the same letter will (generally) be used to denote different
constants depending on the same set of arguments.

\subsection{Function spaces\label{test}}

Let $\mathcal{S}(\mathbf{R}^{d})$ be the Schwartz space of smooth
real-valued rapidly decreasing functions. The space of tempered
distributions we denote by $\mathcal{S}^{\prime }(\mathbf{R}^{d})$. For $p>1$%
, let $L_{p}\left( \mathbf{R}^{d}\right) $ be the space of all measurable
functions $f$ such that 
\begin{equation*}
\left\vert f\right\vert _{L_{p}}=\left\vert f\right\vert _{L_{p}\left( 
\mathbf{R}^{d}\right) }=\left( \int \left\vert f\left( x\right) \right\vert
^{p}dx\right) ^{1/p}<\infty .
\end{equation*}

We fix $\mu \in \mathfrak{A}^{\sigma }$. Obviously, $\func{Re}\psi ^{\mu
}=\psi ^{\mu _{sym}}$, where 
\begin{equation*}
\mu _{sym}\left( dy\right) =\frac{1}{2}\left[ \mu \left( dy\right) +\mu
\left( -dy\right) \right] .
\end{equation*}%
Let%
\begin{equation*}
Jv=J_{\mu }v=(I-L^{\mu _{sym}})v=v-L^{\mu _{sym}}v,v\in \mathcal{S}\left( 
\mathbf{R}^{d}\right) .
\end{equation*}%
For $s\in \mathbf{R}$ set 
\begin{equation*}
J^{s}v=\left( I-L^{\mu _{sym}}\right) ^{s}v=\mathcal{F}^{-1}[(1-\psi ^{\mu
_{sym}})^{s}\hat{v}],v\in \mathcal{S}\left( \mathbf{R}^{d}\right) .
\end{equation*}

For $p\in \lbrack 1,\infty ),s\in \mathbf{R,}$ we define, following \cite%
{fjs}, the Bessel potential space $H_{p}^{s}\left( \mathbf{R}^{d}\right)
=H_{p}^{\mu ;s}\left( \mathbf{R}^{d}\right) $ as the closure of $\mathcal{S}%
\left( \mathbf{R}^{d}\right) $ in the norm%
\begin{eqnarray*}
\left\vert v\right\vert _{H_{p}^{s}} &=&\left\vert J^{s}v\right\vert
_{L_{p}\left( \mathbf{R}^{d}\right) }=\left\vert \mathcal{F}^{-1}[(1-\psi
^{\mu _{sym}})^{s}\hat{v}]\right\vert _{L_{p}\left( \mathbf{R}^{d}\right) }
\\
&=&\left\vert \left( I-L^{\mu _{sym}}\right) ^{s}v\right\vert _{L_{p}\left( 
\mathbf{R}^{d}\right) },v\in \mathcal{S}\left( \mathbf{R}^{d}\right) .
\end{eqnarray*}%
According to Theorem 2.3.1 in \cite{fjs}, $H_{p}^{t}\left( \mathbf{R}%
^{d}\right) \subseteq H_{p}^{s}\left( \mathbf{R}^{d}\right) \,\ $is
continuously embedded if $p\in \left( 1,\infty \right) ,s<t$, $%
H_{p}^{0}\left( \mathbf{R}^{d}\right) =L_{p}\left( \mathbf{R}^{d}\right) $.
For $s\geq 0,p\in \lbrack 1,\infty ),$ the norm $\left\vert v\right\vert
_{H_{p}^{s}}$ is equivalent to (see Theorem 2.2.7 in \cite{fjs})%
\begin{equation*}
\left\vert \left\vert v\right\vert \right\vert _{H_{p}^{s}}=\left\vert
v\right\vert _{L_{p}}+\left\vert \mathcal{F}^{-1}\left[ (-\psi ^{\mu
_{sym}})^{s}\mathcal{F}v\right] \right\vert _{L_{p}}.
\end{equation*}

Further, for a characterization of our function spaces we will use the
following construction (see \cite{bl}, \cite{tr1}). We fix a continuous
function $\kappa :(0,\infty )\rightarrow (0,\infty )$ such that $%
\lim_{R\rightarrow 0}\kappa \left( R\right) =0,\lim_{R\rightarrow \infty
}\kappa \left( R\right) =\infty $. Assume there is a nondecreasing
continuous function $l\left( \varepsilon \right) ,\varepsilon >0,$ such that 
$\lim_{\varepsilon \rightarrow 0}l\left( \varepsilon \right) =0$ and 
\begin{equation*}
\kappa \left( \varepsilon r\right) \leq l\left( \varepsilon \right) \kappa
(r),r>0,\varepsilon >0.
\end{equation*}%
We say $\kappa $ is a \emph{scaling function }and call $l\left( \varepsilon
\right) ,\varepsilon >0,$ a \emph{scaling factor} of $\kappa $. Fix an
integer $N$ so that $l\left( N^{-1}\right) <1.$

\begin{remark}
\label{rem3}For an integer $N>1$ there exists a function $\phi =\phi ^{N}\in
C_{0}^{\infty }(\mathbf{R}^{d})$ (see Lemma 6.1.7 in \cite{bl}), such that $%
\mathrm{supp}\,\phi =\{\xi :\frac{1}{N}\leqslant |\xi |\leqslant N\}$, $\phi
(\xi )>0$ if $N^{-1}<|\xi |<N$ and 
\begin{equation*}
\sum_{j=-\infty }^{\infty }\phi (N^{-j}\xi )=1\quad \text{if }\xi \neq 0.
\end{equation*}%
Let%
\begin{equation}
\tilde{\phi}\left( \xi \right) =\phi \left( N\xi \right) +\phi \left( \xi
\right) +\phi \left( N^{-1}\xi \right) ,\xi \in \mathbf{R}^{d}.  \label{pp1}
\end{equation}%
Note that supp$~\tilde{\phi}\subseteq \left\{ N^{-2}\leq \left\vert \xi
\right\vert \leq N^{2}\right\} $ and $\tilde{\phi}\phi =\phi $. Let $\varphi
_{k}=\varphi _{k}^{N}=\mathcal{F}^{-1}\phi \left( N^{-k}\cdot \right) ,k\geq
1,$ and $\varphi _{0}=\varphi _{0}^{N}\in \mathcal{S}\left( \mathbf{R}%
^{d}\right) $ is defined as%
\begin{equation*}
\varphi _{0}=\mathcal{F}^{-1}\left[ 1-\sum_{k=1}^{\infty }\phi \left(
N^{-k}\cdot \right) \right] .
\end{equation*}%
Let $\phi _{0}\left( \xi \right) =\mathcal{F}\varphi _{0}\left( \xi \right) ,%
\tilde{\phi}_{0}\left( \xi \right) =\mathcal{F}\varphi _{0}\left( \xi
\right) +\mathcal{F\varphi }_{1}\left( \xi \right) ,\xi \in \mathbf{R}^{d}%
\mathbf{,}\tilde{\varphi}=\mathcal{F}^{-1}\tilde{\phi},\varphi =\mathcal{F}%
^{-1}\phi ,$ and 
\begin{equation*}
\tilde{\varphi}_{k}=\sum_{l=-1}^{1}\varphi _{k+l},k\geq 1,\tilde{\varphi}%
_{0}=\varphi _{0}+\varphi _{1}
\end{equation*}%
that is%
\begin{eqnarray*}
\mathcal{F\tilde{\varphi}}_{k} &=&\phi \left( N^{-k+1}\xi \right) +\phi
\left( N^{-k}\xi \right) +\phi \left( N^{-k-1}\xi \right) \\
&=&\tilde{\phi}\left( N^{-k}\xi \right) ,\xi \in \mathbf{R}^{d},k\geq 1.
\end{eqnarray*}%
Note that $\varphi _{k}=\tilde{\varphi}_{k}\ast \varphi _{k},k\geq 0$.
Obviously, $f=\sum_{k=0}^{\infty }f\ast \varphi _{k}$ in $\mathcal{S}%
^{\prime }\left( \mathbf{R}^{d}\right) $ for $f\in \mathcal{S}\left( \mathbf{%
R}^{d}\right) .$
\end{remark}

Let $s\in \mathbf{R}$ and $p,q\geqslant 1$. For $\mu \in \mathfrak{A}%
^{\sigma }$, we introduce the Besov space $B_{pq}^{s}=B_{pq}^{\mu ,N;s}(%
\mathbf{R}^{d})$ as the closure of $\mathcal{S}\left( \mathbf{R}^{d}\right) $
in the norm%
\begin{equation*}
|v|_{B_{pq}^{s}(\mathbf{R}^{d})}=|v|_{B_{pq}^{\mu ,N;s}(\mathbf{R}%
^{d})}=\left( \sum_{j=0}^{\infty }|J^{s}\varphi _{j}\ast v|_{L_{p}\left( 
\mathbf{R}^{d}\right) }^{q}\right) ^{1/q},
\end{equation*}%
where $J=J_{\mu }=I-L^{\mu _{sym}}.$

Similarly we introduce the corresponding spaces of generalized functions on $%
E=[0,T]\times \mathbf{R}^{d}$ $.$ The spaces $B_{pq}^{\mu ,N;s}(E)$ (resp. $%
H_{p}^{\mu ;s}(E)$) consist of all measurable $B_{pq}^{\mu ,N;s}(\mathbf{R}%
^{d})$ (resp. $H_{p}^{\mu ;s}(\mathbf{R}^{d})$) -valued functions $f$ on $%
[0,T]$ with finite corresponding norms:%
\begin{eqnarray}
|f|_{B_{pq}^{s}(E)} &=&|f|_{B_{pq}^{\mu ,N;s}(E)}=\left(
\int_{0}^{T}|f(t,\cdot )|_{B_{pq}^{\mu ,N;s}(\mathbf{R}^{d})}^{q}dt\right)
^{1/q},  \notag \\
|f|_{H_{p}^{s}(E)} &=&|f|_{H_{p}^{\mu ;s}(E)}=\left( \int_{a}^{b}|f(t,\cdot
)|_{H_{p}^{\mu ,s}(\mathbf{R}^{d})}^{p}dt\right) ^{1/p}.  \label{norm11}
\end{eqnarray}

\subsection{Main results}

We introduce an auxiliary Levy measure $\mu ^{0}$ on $\mathbf{R}_{0}^{d}$
such that the following assumption holds.

\textbf{Assumption A}$_{0}\left( \sigma \right) $. \emph{Let} $\mu ^{0}\in 
\mathfrak{A,}\chi _{\left\{ \left\vert y\right\vert \leq 1\right\} }\mu
^{0}\left( dy\right) =\mu ^{0}\left( dy\right) $, \emph{and}%
\begin{equation*}
\int \left\vert y\right\vert ^{2}\mu ^{0}\left( dy\right) +\int \left\vert
\xi \right\vert ^{4}[1+\lambda \left( \xi \right) ]^{d+3}\exp \left\{ -\psi
_{0}\left( \xi \right) \right\} d\xi \leq n_{0},
\end{equation*}

\emph{where} 
\begin{eqnarray*}
\psi _{0}\left( \xi \right) &=&\int_{\left\vert y\right\vert \leq 1}\left[
1-\cos \left( 2\pi \xi \cdot y\right) \right] \mu ^{0}\left( dy\right) , \\
\lambda \left( \xi \right) &=&\int_{\left\vert y\right\vert \leq 1}\chi
_{\sigma }\left( y\right) \left\vert y\right\vert [\left( \left\vert \xi
\right\vert \left\vert y\right\vert \right) \wedge 1]\mu ^{0}\left(
dy\right) ,\xi \in \mathbf{R}^{d}.
\end{eqnarray*}%
\emph{In addition, we assume that for any} $\xi \in S_{d-1}=\left\{ \xi \in 
\mathbf{R}^{d}:\left\vert \xi \right\vert =1\right\} ,$ 
\begin{equation*}
\int_{\left\vert y\right\vert \leq 1}\left\vert \xi \cdot y\right\vert
^{2}\mu ^{0}\left( dy\right) \geq c_{1}>0.
\end{equation*}

For $\pi \in \mathfrak{A}=\cup _{\sigma \in \left( 0,2\right) }\mathfrak{A}%
^{\sigma }$ and $R>0$, we denote%
\begin{equation*}
\pi _{R}\left( \Gamma \right) =\int \chi _{\Gamma }\left( y/R\right) \pi
\left( dy\right) ,\Gamma \in \mathcal{B}\left( \mathbf{R}_{0}^{d}\right) .
\end{equation*}

\begin{definition}
We say that a continuous function $\kappa :(0,\infty )\rightarrow (0,\infty
) $ is a scaling function if $\lim_{R\rightarrow 0}\kappa \left( R\right)
=0,\lim_{R\rightarrow \infty }\kappa \left( R\right) =\infty $ and there is
a nondecreasing continuous function $l\left( \varepsilon \right)
,\varepsilon >0,$ such that $\lim_{\varepsilon \rightarrow 0}l\left(
\varepsilon \right) =0$ and 
\begin{equation*}
\kappa \left( \varepsilon r\right) \leq l\left( \varepsilon \right) \kappa
(r),r>0,\varepsilon >0.
\end{equation*}%
We call $l\left( \varepsilon \right) ,\varepsilon >0,$ a scaling factor of $%
\kappa $.
\end{definition}

For a scaling function $\kappa $ with a scaling factor $l$ and $\pi \in 
\mathfrak{A}^{\sigma }$ we introduce the following assumptions.

\textbf{D}$\left( \kappa ,l\right) $\textbf{. }\emph{For every} $R>0,$ 
\begin{equation*}
\tilde{\pi}_{R}\left( dy\right) =\kappa \left( R\right) \pi _{R}\left(
dy\right) \geq 1_{\left\{ \left\vert y\right\vert \leq 1\right\} }\mu
^{0}\left( dy\right) ,
\end{equation*}%
\emph{with }$\mu ^{0}=\mu ^{0;\pi }$\emph{\ satisfying Assumption }\textbf{A}%
$_{0}\left( \sigma \right) $. \emph{If }$\sigma =1$ \emph{we, in addition
assume that }$\int_{R<\left\vert y\right\vert \leq R^{\prime }}y\mu
^{0}\left( dy\right) =0$ \emph{for any} $0<R<R^{\prime }\leq 1.$ \emph{\
Here }$\tilde{\pi}_{R}\left( dy\right) =\kappa \left( R\right) \pi
_{R}\left( dy\right) .$

\textbf{B}$\left( \kappa ,l\right) $\textbf{. }\emph{There exist }$\alpha
_{1}$\emph{\ and }$\alpha _{2}$\emph{\ and a constant }$N_{0}>0$ \emph{such
that}%
\begin{equation*}
\int_{\left\vert z\right\vert \leq 1}\left\vert z\right\vert ^{\alpha _{1}}%
\tilde{\pi}_{R}(dz)+\int_{\left\vert z\right\vert >1}\left\vert z\right\vert
^{\alpha _{2}}\tilde{\pi}_{R}(dz)\leq N_{0}\text{ }\forall R>0,
\end{equation*}%
\emph{where }$\alpha _{1},\alpha _{2}\in (0,1]\text{ \emph{if} }\sigma \in
(0,1)\text{; }\alpha _{1},\alpha _{2}\in (1,2]\text{ \emph{if} }\sigma \in
(1,2)$\emph{; }$\alpha _{1}\in (1,2]$\emph{\ and }$\alpha _{2}\in \lbrack
0,1)$\emph{\ if }$\sigma =1$\emph{.}

The main result for (\ref{1'}) is the following statement.

\begin{theorem}
\label{t1}Let $\pi ,\mu \in \mathfrak{A}^{\sigma }$, \thinspace $p\in \left(
1,\infty \right) ,s\in \mathbf{R}$. Assume there is a scaling function $%
\kappa $ with scaling factor $l$ such that \textbf{D}$\left( \kappa
,l\right) $ and \textbf{B}$\left( \kappa ,l\right) $\textbf{\ }hold for
both, $\pi $ and $\mu $. Let $\gamma \left( t\right) =\inf \{t>0:l\left(
r\right) \geq t\},t>0.$ 

Assume%
\begin{equation*}
\int_{1}^{\infty }\frac{dt}{t\gamma \left( t\right) ^{1\wedge \alpha _{2}}}%
<\infty .
\end{equation*}%
and there are $\beta _{1},\beta _{2}>0$ such that%
\begin{equation*}
\int_{0}^{1}\gamma \left( t\right) ^{-\beta _{1}}dt+\int_{0}^{1}l\left(
t\right) ^{\beta _{2}}\frac{dt}{t}<\infty \text{ if }p>2.
\end{equation*}

Then for each $f\in H_{p}^{\mu ;s}(E),g\in B_{pp}^{\mu ,N;s+1-1/p}\left( 
\mathbf{R}^{d}\right) $ there is a unique $u\in H_{p}^{\mu ;s+1}\left(
E\right) $ solving (\ref{1'}). Moreover, there is $C=C\left( d,p,\kappa
,l,n_{0},N_{0},c_{1}\right) $ such that%
\begin{eqnarray*}
\left\vert L^{\mu }u\right\vert _{H_{p}^{\mu ;s}\left( E\right) } &\leq &C
\left[ \left\vert f\right\vert _{H_{p}^{\mu ;s}\left( E\right) }+\left\vert
g\right\vert _{B_{pp}^{\mu ,N;s+1-1/p}\left( \mathbf{R}^{d}\right) }\right] ,
\\
\left\vert u\right\vert _{H_{p}^{\mu ;s}\left( E\right) } &\leq &\rho
_{\lambda }\left\vert f\right\vert _{H_{p}^{\mu ;s}\left( E\right) }+\rho
_{\lambda }^{1/p}\left\vert g\right\vert _{H_{p}^{\mu ;s}\left( \mathbf{R}%
^{d}\right) },
\end{eqnarray*}%
where $\rho _{\lambda }=\frac{1}{\lambda }\wedge T.$
\end{theorem}

\begin{remark}
1. Assumptions \textbf{D}$\left( \kappa ,l\right) ,$ \textbf{B}$\left(
\kappa ,l\right) $ holds for both, $\pi ,\mu $, means that $\kappa ,l,$ and
the parameters $\alpha _{1},\alpha _{2},n_{0},c_{1},N_{0}$ are the same ($%
\mu ^{0}$ could be different).

2. For every $\varepsilon >0$, $B_{pp}^{\mu ,N;s+\varepsilon }\left( \mathbf{%
R}^{d}\right) $ is continuously embedded into $H_{p}^{\mu ;s}\left( \mathbf{R%
}^{d}\right) ,p>1$ (see Remark \ref{renew1} below); for \thinspace $p\geq 2$%
, $H_{p}^{\mu ;s}\left( \mathbf{R}^{d}\right) $ is continuously embedded
into $B_{pp}^{\mu ,N;s}\left( \mathbf{R}^{d}\right) .$
\end{remark}

\subsection{Example}

Let $\Lambda \left( dt\right) $ be a measure on $\left( 0,\infty \right) $
such that $\int_{0}^{\infty }\left( 1\wedge t\right) \Lambda \left(
dt\right) <\infty $, and let 
\begin{equation*}
\phi \left( r\right) =\int_{0}^{\infty }\left( 1-e^{-rt}\right) \Lambda
\left( dt\right) ,r\geq 0,
\end{equation*}%
be a Bernstein function (see \cite{vo}, \cite{kk}). Let%
\begin{equation*}
j\left( r\right) =\int_{0}^{\infty }\left( 4\pi t\right) ^{-\frac{d}{2}}\exp
\left( -\frac{r^{2}}{4t}\right) \Lambda \left( dt\right) ,r>0.
\end{equation*}%
We consider $\pi \in \mathfrak{A}=\cup _{\sigma \in \left( 0,2\right) }%
\mathfrak{A}^{\sigma }\,\ $defined in radial and angular coordinates $%
r=\left\vert y\right\vert ,w=y/r,$ as%
\begin{equation}
\pi \left( \Gamma \right) =\int_{0}^{\infty }\int_{\left\vert w\right\vert
=1}\chi _{\Gamma }\left( rw\right) a\left( r,w\right) j\left( r\right)
r^{d-1}S\left( dw\right) dr,\Gamma \in \mathcal{B}\left( \mathbf{R}%
_{0}^{d}\right) ,  \label{fe1}
\end{equation}%
where $S\left( dw\right) $ is a finite measure on the unite sphere on $%
\mathbf{R}^{d}$. If $S\left( dw\right) =dw$ is the Lebesgue measure on the
unit sphere, then%
\begin{equation*}
\pi \left( \Gamma \right) =\pi ^{J,a}\left( \Gamma \right) =\int_{\mathbf{R}%
^{d}}\chi _{\Gamma }\left( y\right) a\left( \left\vert y\right\vert
,y/\left\vert y\right\vert \right) J\left( y\right) dy,\Gamma \in \mathcal{B}%
\left( \mathbf{R}_{0}^{d}\right) ,
\end{equation*}%
where $J\left( y\right) =j\left( \left\vert y\right\vert \right) ,y\in 
\mathbf{R}^{d}.$ Let $\mu =\pi ^{J,1},$ i.e.,%
\begin{equation}
\mu \left( \Gamma \right) =\int_{\mathbf{R}^{d}}\chi _{\Gamma }\left(
y\right) J\left( y\right) dy,\Gamma \in \mathcal{B}\left( \mathbf{R}%
_{0}^{d}\right) .  \label{fe2}
\end{equation}

We assume

\textbf{H. }(i) \emph{There is} $N>0$ \emph{so that}%
\begin{equation*}
N^{-1}\phi \left( r^{-2}\right) r^{-d}\leq j\left( r\right) \leq N\phi
\left( r^{-2}\right) r^{-d},r>0.
\end{equation*}

(ii) \emph{There are} $0<\delta _{1}\leq \delta _{2}<1$ \emph{and }$N>0$%
\emph{\ so that for} $0<r\leq R$%
\begin{equation*}
N^{-1}\left( \frac{R}{r}\right) ^{\delta _{1}}\leq \frac{\phi \left(
R\right) }{\phi \left( r\right) }\leq N\left( \frac{R}{r}\right) ^{\delta
_{2}}.
\end{equation*}

\textbf{G. }\emph{There is} $\rho _{0}\left( w\right) \geq 0,\left\vert
w\right\vert =1,$ \emph{such that} $\rho _{0}\left( w\right) \leq a\left(
r,w\right) \leq 1,r>0,\left\vert w\right\vert =1,$ \emph{and for all} $%
\left\vert \xi \right\vert =1,$ 
\begin{equation*}
\int_{\left\vert w\right\vert =1}\left\vert \xi \cdot w\right\vert ^{2}\rho
_{0}\left( w\right) S\left( dw\right) \geq c>0
\end{equation*}%
\emph{for some} $c>0$.

For example, in \cite{vo} and \cite{kk} among others the following specific
Bernstein functions satisfying \textbf{H} are listed:

(0) $\phi \left( r\right) =\sum_{i=1}^{n}r^{\alpha _{i}},\alpha _{i}\in
\left( 0,1\right) ,i=1,\ldots ,n;$

(1) $\phi \left( r\right) =\left( r+r^{\alpha }\right) ^{\beta },\alpha
,\beta \in \left( 0,1\right) ;$

(2) $\phi \left( r\right) =r^{\alpha }\left( \ln \left( 1+r\right) \right)
^{\beta },\alpha \in \left( 0,1\right) ,\beta \in \left( 0,1-\alpha \right)
; $

(3) $\phi \left( r\right) =\left[ \ln \left( \cosh \sqrt{r}\right) \right]
^{\alpha },\alpha \in \left( 0,1\right) .$

All the assumptions of Theorem \ref{t1} hold under \textbf{H, G.}

Indeed, \textbf{H} implies that there are $0<c\leq C$ so that 
\begin{eqnarray*}
cr^{-d-2\delta _{1}} &\leq &j\left( r\right) \leq Cr^{-d-2\delta _{2}},r\leq
1, \\
cr^{-d-2\delta _{2}} &\leq &j\left( r\right) \leq Cr^{-d-2\delta _{1}},r>1.
\end{eqnarray*}%
Hence $2\delta _{1}\leq \sigma \leq 2\delta _{2}.$ In this case $\kappa
\left( R\right) =j\left( R\right) ^{-1}R^{-d},R>0,$ is a scaling function,
and $\kappa \left( \varepsilon R\right) \leq l\left( \varepsilon \right)
\kappa \left( R\right) ,\varepsilon ,R>0,$ with%
\begin{equation*}
l\left( \varepsilon \right) =\left\{ 
\begin{array}{cc}
C_{1}\varepsilon ^{2\delta _{1}} & \text{if }\varepsilon \leq 1, \\ 
C_{1}\varepsilon ^{2\delta _{2}} & \text{if }\varepsilon >1%
\end{array}%
\right.
\end{equation*}%
for some $C_{1}>0$. Hence 
\begin{equation*}
\gamma \left( t\right) =l^{-1}\left( t\right) =\left\{ 
\begin{array}{c}
C_{1}^{-1/2\delta _{1}}t^{1/2\delta _{1}}\text{ if }t\leq C_{1}, \\ 
C_{1}^{-1/2\delta _{2}}t^{1/2\delta _{2}}\text{ if }t>C_{1}.%
\end{array}%
\right.
\end{equation*}%
We see easily that $\alpha _{1}$ is any number $>2\delta _{2}$ and $\alpha
_{2}$ is any number $<2\delta _{1}.$ The measure $\mu ^{0}$ for $\pi $ is%
\begin{equation*}
\mu ^{0}\left( dy\right) =\mu ^{0,\pi }\left( dy\right) =c_{1}\int \chi
_{dy}\left( rw\right) \chi _{\left\{ r\leq 1\right\} }r^{-1-2\delta
_{1}}\rho _{0}\left( w\right) S\left( dw\right) dr;
\end{equation*}%
and $\mu ^{0}$ for $\mu $ is%
\begin{equation*}
\mu ^{0}\left( dy\right) =\mu ^{0,\mu }\left( dy\right) =c_{1}^{\prime }\int
\chi _{dy}\left( rw\right) \chi _{\left\{ r\leq 1\right\} }r^{-1-2\delta
_{1}}dwdr
\end{equation*}%
with some $c_{1},c_{1}^{\prime }.$ Integrability conditions easily follow.

\section{Function spaces, equivalent norms}

\subsection{Function spaces}

Let $\tilde{C}^{\infty }\left( \mathbf{R}^{d}\right) $ be the space of all
functions $f$ on $\mathbf{R}^{d}$ such that for any multiindex $\gamma \in 
\mathbf{N}_{0}^{d}$ and for all $1\leq p<\infty $%
\begin{equation*}
\sup_{x\in \mathbf{R}^{d}}\left\vert D^{\gamma }f\left( x\right) \right\vert
+\left\vert D^{\gamma }f\right\vert _{L_{p}\left( R^{d}\right) }<\infty .
\end{equation*}%
Let $\tilde{C}_{p}^{\infty }\left( \mathbf{R}^{d}\right) $ be the space of
all functions $f$ on $\mathbf{R}^{d}$ such that for any multiindex $\gamma
\in \mathbf{N}_{0}^{d}$ 
\begin{equation*}
\sup_{x\in \mathbf{R}^{d}}\left\vert D^{\gamma }f\left( x\right) \right\vert
+\left\vert D^{\gamma }f\right\vert _{L_{p}\left( R^{d}\right) }<\infty .
\end{equation*}

For a separable Hilbert space $\mathbf{G}$ and $r\geq 1$, we denote $%
l_{r}\left( \mathbf{G}\right) $ the space of all sequences $a=\left(
a_{j}\right) ,a_{j}\in \mathbf{G}$, with finite norm%
\begin{equation*}
\left\vert a\right\vert _{l_{r}\left( \mathbf{G}\right) }=\left(
\sum_{j=0}^{\infty }\left\vert a_{j}\right\vert _{\mathbf{G}}^{r}\right)
^{1/r}.
\end{equation*}%
We denote $l_{r}=l_{r}\left( \mathbf{R}\right) $. Let $L_{p}\left( \mathbf{R}%
^{d};\mathbf{G}\right) $ be the space of all $\mathbf{G}$-valued measurable
functions $f$ such that 
\begin{equation*}
\left\vert f\right\vert _{L_{p}\left( \mathbf{R}^{d};\mathbf{G}\right)
}=\left( \int \left\vert f\left( x\right) \right\vert _{\mathbf{G}%
}^{p}dx\right) ^{1/p}<\infty .
\end{equation*}

Let $\mu \in \mathfrak{A}^{\sigma }$, $H_{p}^{s}\left( \mathbf{R}%
^{d};l_{2}\right) =H_{p}^{\mu ;s}\left( \mathbf{R}^{d};l_{2}\right) $ be the
space of all sequences $v=\left( v_{k}\right) _{k\geq 0}$ with $v_{k}\in
H_{p}^{\mu ;s}\left( \mathbf{R}^{d}\right) $ and finite norm%
\begin{eqnarray*}
\left\vert v\right\vert _{H_{p}^{s}\left( \mathbf{R}^{d};l_{2}\right) }
&=&\left\vert \left( \sum_{k=0}^{\infty }\left\vert J^{s}v_{k}\right\vert
^{2}\right) ^{1/2}\right\vert _{L_{p}\left( \mathbf{R}^{d}\right)
}=\left\vert \left( \sum_{k=0}^{\infty }\left\vert \mathcal{F}^{-1}[(1-\psi
^{\mu _{sym}})^{s}\hat{v}_{k}]\right\vert ^{2}\right) ^{1/2}\right\vert
_{L_{p}\left( \mathbf{R}^{d}\right) } \\
&=&\left\vert \left( \sum_{k=0}^{\infty }\left\vert \left( I-L^{\mu
_{sym}}\right) ^{s}v_{k}\right\vert ^{2}\right) ^{1/2}\right\vert
_{L_{p}\left( \mathbf{R}^{d}\right) }.
\end{eqnarray*}

For a scaling function $\kappa $ with a scaling factor $l\left( \varepsilon
\right) ,\varepsilon >0,$ and integer $N>1$ such that $l\left( N^{-1}\right)
<1$ and $s\in \mathbf{R}$, we introduce Besov spaces $\tilde{B}_{pq}^{s}=%
\tilde{B}_{pq}^{\kappa ,N;s}=\tilde{B}_{pq}^{\kappa ,N;s}(\mathbf{R}^{d})$
of generalized functions $v\in \mathcal{S}^{\prime }(\mathbf{R}^{d})$ with
finite norm%
\begin{equation*}
|v|_{\tilde{B}_{pq}^{s}(\mathbf{R}^{d})}=|v|_{\tilde{B}_{pq}^{\kappa ,N;s}(%
\mathbf{R}^{d})}=\left( \sum_{j=0}^{\infty }\kappa \left( N^{-j}\right)
^{-sq}|\varphi _{j}\ast v|_{L_{p}}^{q}\right) ^{1/q},
\end{equation*}%
where $\varphi _{j}=\varphi _{j}^{N},j\geq 0,$ is the system of functions
defined in Remark \ref{rem3}. Let $\tilde{H}_{p}^{s}=\tilde{H}_{p}^{\kappa
,N;s}(\mathbf{R}^{d})$ $\ $\ be the space of $v\in \mathcal{S}^{\prime }(%
\mathbf{R}^{d})$ with finite norm%
\begin{equation}
|v|_{\tilde{H}_{p}^{s}(\mathbf{R}^{d})}=|v|_{\tilde{H}_{p}^{\kappa ,N;s}(%
\mathbf{R}^{d})}=\left\vert \left( \sum_{j=0}^{\infty }\left\vert \kappa
\left( N^{-j}\right) ^{-s}\varphi _{j}\ast v\right\vert ^{2}\right)
^{1/2}\right\vert _{L_{p}\left( \mathbf{R}^{d}\right) }.
\end{equation}

Let $\tilde{H}_{p}^{s}\left( \mathbf{R}^{d};l_{2}\right) =\tilde{H}%
_{p}^{\kappa ,N;s}(\mathbf{R}^{d};l_{2})$ be the space of all sequences $%
v=\left( v_{k}\right) _{k\geq 0}$ with $v_{k}\in \tilde{H}_{p}^{s}\left( 
\mathbf{R}^{d}\right) $ and finite norm%
\begin{equation*}
\left\vert v\right\vert _{\tilde{H}_{p}^{s}\left( \mathbf{R}%
^{d};l_{2}\right) }=\left\vert \left( \sum_{k,j=0}^{\infty }\left\vert
\kappa \left( N^{-j}\right) ^{-s}\varphi _{j}\ast v_{k}\right\vert
^{2}\right) ^{1/2}\right\vert _{L_{p}\left( \mathbf{R}^{d}\right) }.
\end{equation*}

\subsection{Norm equivalence and embedding}

In this section we consider some equivalent norms in $H_{p}^{s}=H_{p}^{\mu
;s}$ and $B_{pq}^{s}=B_{pq}^{\mu ,N;s}.$

Let $\mu \in \mathfrak{A}^{\sigma }$, and $Z_{t}^{\mu },t\geq 0,$ be the
Levy process associated to $L^{\mu }$, i.e., $Z^{\mu }$ is cadlag with
independent increments and its characteristic function%
\begin{equation*}
\mathbf{E}e^{i2\pi \xi \cdot Z_{t}^{\mu }}=\exp \left\{ \psi ^{\mu }\left(
\xi \right) t\right\} ,\xi \in \mathbf{R}^{d},t\geq 0.
\end{equation*}%
By Corollary 5 in \cite{MPh}, for any $g\in \tilde{C}^{\infty }\left( 
\mathbf{R}^{d}\right) $%
\begin{equation*}
g\left( x\right) =\int_{0}^{\infty }e^{-t}\mathbf{E}\left( I-L^{\mu }\right)
g\left( x+Z_{t}^{\mu }\right) dt=\left( I-L^{\pi }\right) \int_{0}^{\infty
}e^{-t}\mathbf{E}g\left( x+Z_{t}^{\mu }\right) dt,x\in \mathbf{R}^{d},
\end{equation*}%
i.e. $I-L^{\mu }:\tilde{C}^{\infty }\left( \mathbf{R}^{d}\right) \rightarrow 
\tilde{C}^{\infty }\left( \mathbf{R}^{d}\right) $ is bijective and%
\begin{equation}
\left( I-L^{\mu }\right) ^{-1}g\left( x\right) =\int_{0}^{\infty }e^{-t}%
\mathbf{E}g\left( x+Z_{t}^{\mu }\right) dt,x\in \mathbf{R}^{d}.  \label{cf2}
\end{equation}

\begin{remark}
\label{rp}Let $p\in \left( 1,\infty \right) $, and $\tilde{C}_{p}^{\infty
}\left( \mathbf{R}^{d}\right) $ be the space of all functions $f$ on $%
\mathbf{R}^{d}$ such that for any multiindex $\gamma \in \mathbf{N}_{0}^{d}$ 
\begin{equation*}
\sup_{x\in \mathbf{R}^{d}}\left\vert D^{\gamma }f\left( x\right) \right\vert
+\left\vert D^{\gamma }f\right\vert _{L_{p}\left( R^{d}\right) }<\infty .
\end{equation*}%
Then for $\mu \in \mathfrak{A}^{\sigma }$ the mapping $I-L^{\mu }:\tilde{C}%
_{p}^{\infty }\left( \mathbf{R}^{d}\right) \rightarrow \tilde{C}_{p}^{\infty
}\left( \mathbf{R}^{d}\right) $ is bijective and%
\begin{equation}
(I-L^{\mu })^{-1}g\left( x\right) =\int_{0}^{\infty }e^{-t}\mathbf{E}g\left(
x+Z_{t}^{\mu }\right) dt,x\in \mathbf{R}^{d}.  \label{cf3}
\end{equation}%
Indeed, if $g\in \tilde{C}_{p}^{\infty }\left( \mathbf{R}^{d}\right) $, then
by Ito formula,%
\begin{equation*}
u\left( x\right) =\int_{0}^{\infty }e^{-t}\mathbf{E}g\left( x+Z_{t}^{\mu
}\right) dt,x\in \mathbf{R}^{d},
\end{equation*}%
is a classical solution to the equation $(I-L^{\mu })u=g$ and $u\in \tilde{C}%
_{p}^{\infty }\left( \mathbf{R}^{d}\right) .$
\end{remark}

We will prove the following statement about $B_{pq}^{\mu ,N;s}.$

\begin{proposition}
\label{pro1}Let \textbf{D}$\left( \kappa ,l\right) $ and \textbf{B}$\left(
\kappa ,l\right) $ hold for $\mu \in \mathfrak{A}^{\sigma }$ with scaling
function $\kappa $ and scaling factor $l$. Let $s\in \mathbf{R},p,q\in
\left( 1,\infty \right) ,N>1,l\left( N^{-1}\right) <1$. Then $\tilde{B}%
_{pq}^{\kappa ,N;s}\left( \mathbf{R}^{d}\right) =B_{pq}^{\mu ,N;s}\left( 
\mathbf{R}^{d}\right) $ and the norms are equivalent.
\end{proposition}

We will use some equivalent norms on $H_{p}^{\mu ;s}$ as well.

\begin{proposition}
\label{pro2}Let \textbf{D}$\left( \kappa ,l\right) $ and \textbf{B}$\left(
\kappa ,l\right) $ hold for $\mu \in \mathfrak{A}^{\sigma }$ with scaling
function $\kappa $ and scaling factor $l$. Let $s\in \mathbf{R},p\in \left(
1,\infty \right) ,N>1,l\left( N^{-1}\right) <1$. Then $\tilde{H}_{p}^{\kappa
,N;s}\left( \mathbf{R}^{d};l_{2}\right) =H_{p}^{\mu ;s}\left( \mathbf{R}%
^{d};l_{2}\right) $ and the norms are equivalent.
\end{proposition}

First we will present some technical auxiliary results that are used in the
proof of Propositions \ref{pro1}, \ref{pro2} that follows afterwards. The
spaces $\tilde{H}_{p}^{\kappa ,N;s},\tilde{B}_{pq}^{\kappa ,N;s}$ belong to
the class of spaces of generalized smoothness studied e.g. in \cite{ka} and 
\cite{kali} (see references therein as well). This allows to characterize $%
H_{p}^{\mu ;s}$ and $B_{pq}^{\mu ;s}$ using differences. This and embedding
into the space of continuous functions is discussed at the end of this
section.

\subsubsection{Auxiliary results}

We start with

\begin{lemma}
\label{l1}Let $N>1$, and $\Phi _{j}\left( x\right) ,x\in \mathbf{R}%
^{d},j\geq 0,$ be a sequence of measurable functions. Assume

(i) There is $\beta >0$ so that%
\begin{equation*}
\int \left\vert x\right\vert ^{\beta }\left\vert \Phi _{j}\left( x\right)
\right\vert dx\leq A,j\geq 0.
\end{equation*}

(ii) There is a nonnegative increasing function $w\left( r\right) ,r\in
\lbrack 0,1],$ so that $\sum_{k=0}^{\infty }w\left( N^{-k}\right) <\infty $
and 
\begin{equation*}
\int \left\vert \Phi _{j}\left( x+y\right) -\Phi _{j}\left( x\right)
\right\vert dx\leq w\left( \left\vert y\right\vert \right) ,\left\vert
y\right\vert \leq 1,j\geq 0.
\end{equation*}

Then for $K_{j}\left( x\right) =N^{jd}\Phi _{j}\left( N^{j}x\right) ,x\in 
\mathbf{R}^{d},j\geq 0$, we have 
\begin{equation}
\sum_{j=0}^{\infty }\int_{\left\vert x\right\vert >4\left\vert y\right\vert
}\left\vert K_{j}\left( x+y\right) -K_{j}\left( x\right) \right\vert dx\leq
B,y\in \mathbf{R}^{d},  \label{0}
\end{equation}%
for some constant $B.$
\end{lemma}

\begin{proof}
For any $y\in \mathbf{R}^{d},$%
\begin{eqnarray*}
&&\sum_{k=0}^{\infty }\int_{\left\vert x\right\vert >4\left\vert
y\right\vert }\left\vert K_{k}\left( x+y\right) -K_{k}\left( x\right)
\right\vert dx \\
&=&\sum_{k=0}^{\infty }\int_{\left\vert x\right\vert >N^{k}4\left\vert
y\right\vert }\left\vert \Phi _{k}\left( x+N^{k}y\right) -\Phi _{k}\left(
x\right) \right\vert dx \\
&\leq &\sum_{k=0}^{\infty }\sup_{j\geq 0}\int_{\left\vert x\right\vert
>N^{k}4\left\vert y\right\vert }\left\vert \Phi _{j}\left( x+N^{k}y\right)
-\Phi _{j}\left( x\right) \right\vert dx=\sum_{k=0}^{\infty }F\left(
N^{k}y\right) ,
\end{eqnarray*}%
where%
\begin{equation*}
F(z)=\sup_{j\geq 0}\int_{\left\vert x\right\vert >4\left\vert z\right\vert
}\left\vert \Phi _{j}\left( x+z\right) -\Phi _{j}\left( x\right) \right\vert
dx,z\in \mathbf{R}^{d}.
\end{equation*}%
Let 
\begin{equation*}
G\left( y\right) =\sum_{k=-\infty }^{\infty }F\left( N^{k}y\right) ,y\in 
\mathbf{R}^{d}.
\end{equation*}%
Since $G\left( Ny\right) =G\left( y\right) ,y\in \mathbf{R}^{d}$, it is
enough to prove that 
\begin{equation}
G\left( y\right) \leq B,1/N\leq \left\vert y\right\vert \leq 1,  \label{f2}
\end{equation}%
for some $B>0.$\ We split the sum%
\begin{eqnarray*}
G\left( y\right) &=&\sum_{k=-\infty }^{\infty }F\left( N^{k}y\right)
=\sum_{k=0}^{\infty }...+\sum_{k=-\infty }^{-1}... \\
&=&G_{1}\left( y\right) +G_{2}\left( y\right) ,1/N\leq \left\vert
y\right\vert \leq 1.
\end{eqnarray*}%
With $1/N\leq \left\vert y\right\vert \leq 1,k\geq 0,$ by Chebyshev
inequality,%
\begin{eqnarray*}
&&\int_{\left\vert x\right\vert >N^{k}4\left\vert y\right\vert }\left\vert
\Phi _{j}\left( x+N^{k}y\right) -\Phi _{j}\left( x\right) \right\vert dx \\
&\leq &\int_{\left\vert x\right\vert >N^{k}4\left\vert y\right\vert
}\left\vert \Phi _{j}\left( x+N^{k}y\right) \right\vert dx+\int_{\left\vert
x\right\vert >N^{k}4\left\vert y\right\vert }\left\vert \Phi _{j}\left(
x\right) \right\vert dx \\
&\leq &C\int_{\left\vert x\right\vert >N^{k}3\left\vert y\right\vert
}\left\vert \Phi _{j}\left( x\right) \right\vert dx\leq C\int_{\left\vert
x\right\vert >N^{k-1}3}\left\vert \Phi _{j}\left( x\right) \right\vert dx \\
&\leq &CN^{-k\beta }\int \left\vert x\right\vert ^{\beta }\left\vert \Phi
_{j}\left( x\right) \right\vert dx\leq CAN^{-k\beta },
\end{eqnarray*}%
and 
\begin{equation*}
G_{1}\left( y\right) \leq CA\sum_{k=0}^{\infty }N^{-k\beta },1/N\leq
\left\vert y\right\vert \leq 1.
\end{equation*}%
For $1/N\leq \left\vert y\right\vert \leq 1,k<0,$%
\begin{eqnarray*}
&&\int_{\left\vert x\right\vert >N^{k}4\left\vert y\right\vert }\left\vert
\Phi _{j}\left( x+N^{k}y\right) -\Phi _{j}\left( x\right) \right\vert dx \\
&\leq &\int \left\vert \Phi _{j}\left( x+N^{k}y\right) -\Phi _{j}\left(
x\right) \right\vert dx\leq w\left( N^{k}\right) ,
\end{eqnarray*}%
and%
\begin{equation*}
G_{2}\left( y\right) \leq \sum_{k=-\infty }^{-1}w\left( N^{k}\right)
,1/N\leq \left\vert y\right\vert \leq 1.
\end{equation*}%
The claim is proved.
\end{proof}

\begin{corollary}
\label{c1}Let the assumptions of Lemma \ref{l1} hold and $\sup_{j,\xi
}\left\vert \hat{\Phi}_{j}\left( \xi \right) \right\vert <\infty $, and let $%
\mathbf{G}$ be a separable Hilbert space. Then

(i) For any $1<p,r<\infty $ there is a constant $C_{p,r}$ so that%
\begin{equation*}
\left\vert \left( \sum_{j}\left\vert K_{j}\ast f_{j}\right\vert ^{r}\right)
^{1/r}\right\vert _{L_{p}\left( \mathbf{R}^{d}\right) }\leq
C_{p,r}\left\vert \left( \sum_{j}\left\vert f_{j}\right\vert ^{r}\right)
^{1/r}\right\vert _{L_{p}\left( \mathbf{R}^{d}\right) }.
\end{equation*}%
for all $f=\left( f_{j}\right) \in L_{p}\left( \mathbf{R}^{d},l_{r}\right) .$

(ii) For any $1<p<\infty $ there is a constant $C>0$ such that%
\begin{equation*}
\left\vert \left( \sum_{j}\left\vert K_{j}\ast f_{j}\right\vert _{\mathbf{G}%
}^{2}\right) ^{1/2}\right\vert _{L_{p}\left( \mathbf{R}^{d}\right) }\leq
C_{p}\left\vert \left( \sum_{j}\left\vert f_{j}\right\vert _{\mathbf{G}%
}^{2}\right) ^{1/2}\right\vert _{L_{p}\left( \mathbf{R}^{d}\right) }
\end{equation*}%
for all $f=\left( f_{j}\right) \in L_{p}\left( \mathbf{R}^{d};l_{2}\left( 
\mathbf{G}\right) \right) $.
\end{corollary}

\begin{proof}
(i) Since (\ref{0}) holds according to Lemma \ref{l1}, the statement follows
by Theorem V.3.11 in \cite{gr}.

(ii) Since $\mathbf{G}$ is isomorphic to $l_{2}$, the statement follows by
Theorem V.3.9 in \cite{gr}.
\end{proof}

As the first application we have

\begin{corollary}
\label{c2}Let $\zeta ,\zeta _{0}\in \mathcal{S}\left( \mathbf{R}^{d}\right) $%
, $\tilde{\zeta}=\mathcal{F}^{-1}\zeta ,j\geq 1,\tilde{\zeta}_{0}=\mathcal{F}%
^{-1}\zeta _{0}$. Let $N>1,\tilde{\zeta}_{j}\left( x\right) =N^{jd}\tilde{%
\zeta}\left( N^{j}x\right) ,x\in \mathbf{R}^{d},j\geq 1.$ Then for each $%
1<p,r<\infty $ there is a constant $C_{p,r}$ so that for all $f=\left(
f_{j}\right) \in L_{p}\left( \mathbf{R}^{d},l_{r}\right) $ 
\begin{equation}
\left\vert \left( \sum_{j}\left\vert f_{j}\ast \tilde{\zeta}_{j}\right\vert
^{r}\right) ^{1/r}\right\vert _{L_{p}\left( \mathbf{R}^{d}\right) }\leq
C_{p,r}\left\vert \left( \sum_{j}\left\vert f_{j}\right\vert ^{r}\right)
^{1/r}\right\vert _{L_{p}\left( \mathbf{R}^{d}\right) }.  \label{ff3}
\end{equation}%
If $r=2$, then (\ref{ff3}) holds for a separable Hilbert space $\mathbf{G}$%
-valued sequences $f=\left( f_{j}\right) \in L_{p}\left( \mathbf{R}%
^{d},l_{r}\left( \mathbf{G}\right) \right) $ (simply absolute value in (\ref%
{ff3}) is replaced by $\mathbf{G}$-norm).
\end{corollary}

\begin{proof}
We apply previous Corollary \ref{c1} with $\Phi _{0}=\tilde{\zeta}_{0},\Phi
_{j}\left( x\right) =\Phi \left( x\right) =\tilde{\zeta}\left( x\right)
,j\geq 1,K_{j}\left( x\right) =N^{jd}\Phi _{j}\left( N^{j}x\right) ,x\in 
\mathbf{R}^{d},j\geq 0.$ Obviously, 
\begin{eqnarray*}
\sup_{\xi }\left[ |\zeta \left( \xi \right) |+\left\vert \zeta _{0}\left(
\xi \right) \right\vert \right] &<&\infty , \\
\int \left\vert x\right\vert \left[ \left\vert \Phi \left( x\right)
\right\vert +\left\vert \Phi _{0}\left( x\right) \right\vert \right] dx
&<&\infty ,
\end{eqnarray*}%
and%
\begin{eqnarray*}
\int \left\vert \Phi \left( x+y\right) -\Phi \left( x\right) \right\vert dx
&\leq &\int \int_{0}^{1}\left\vert \nabla \Phi \left( x+sy\right)
\right\vert ds\left\vert y\right\vert dx \\
&\leq &\left\vert y\right\vert \int \left\vert \nabla \Phi \left( x\right)
\right\vert dx,y\in \mathbf{R}^{d}.
\end{eqnarray*}%
\ Similarly,%
\begin{equation*}
\int \left\vert \Phi _{0}\left( x+y\right) -\Phi _{0}\left( x\right)
\right\vert dx\leq \left\vert y\right\vert \int \left\vert \nabla \Phi
_{0}\left( x\right) \right\vert dx,y\in \mathbf{R}^{d}.
\end{equation*}%
The statement follows by Corollary \ref{c1}.
\end{proof}

We will need the following auxiliary statement.

\begin{lemma}
\label{auxl1}Let \textbf{D}$\left( \kappa ,l\right) $ and \textbf{B}$\left(
\kappa ,l\right) $ hold for $\pi \in \mathfrak{A}^{\sigma }$ with scaling
function $\kappa $ and scaling factor $l$. Let $R>0$ and $Z_{t}^{R}=Z_{t}^{%
\tilde{\pi}_{R}}$ be the Levy process associated to $L^{\tilde{\pi}_{R}}$,
and let $\zeta ,\zeta _{0}\in C_{0}^{\infty }\left( \mathbf{R}^{d}\right) $
be such that supp$\left( \zeta \right) \subseteq \left\{ \xi :0<R_{1}\leq
\left\vert \xi \right\vert \leq R_{2}\right\} $ and 
\begin{equation*}
\max_{\left\vert \gamma \right\vert \leq n}\left\vert D^{\gamma }\zeta
\left( \xi \right) \right\vert \leq N_{1},R_{1}\leq \left\vert \xi
\right\vert \leq R_{2},
\end{equation*}%
with $n=d_{0}+2=\left[ \frac{d}{2}\right] +3$. Let $\tilde{\zeta}=\mathcal{F}%
^{-1}\zeta ,\tilde{\zeta}_{0}=\mathcal{F}^{-1}\zeta _{0}$, and%
\begin{eqnarray*}
H^{R}\left( t,x\right) &=&\mathbf{E}\tilde{\zeta}\left( x+Z_{t}^{R}\right)
,t\geq 0,x\in \mathbf{R}^{d}, \\
H_{0}^{R}\left( t,x\right) &=&\mathbf{E}\tilde{\zeta}_{0}\left(
x+Z_{t}^{R}\right) ,t\geq 0,x\in \mathbf{R}^{d}.
\end{eqnarray*}

Then

(i) There are constants $C_{k}=C_{k}\left(
R_{1},R_{2},N_{1},n_{0},c_{1},N_{0},d\right) ,k=1,2,C_{0}=C_{0}\left(
N_{0}\right) $ $(n_{0},c_{1}$ are constants in assumption \textbf{D}$\left(
\kappa ,l\right) $ and $N_{0}$ is a constant in \textbf{B}$\left( \kappa
,l\right) $) so that 
\begin{eqnarray*}
\int \left( 1+\left\vert x\right\vert ^{\alpha _{2}}\right) \left\vert
H^{R}\left( t,x\right) \right\vert dx &\leq &C_{1}e^{-C_{2}t},t\geq 0, \\
\int \left\vert x\right\vert ^{\alpha _{2}}\left\vert H_{0}^{R}\left(
t,x\right) \right\vert dx &\leq &C_{0}\left( 1+t\right) ,t\geq 0, \\
\int \left\vert H_{0}^{R}\left( t,x\right) \right\vert dx &\leq &C_{0},t\geq
0.
\end{eqnarray*}

(ii) There are constants $C_{k}=C_{k}\left(
R_{1},R_{2},N_{1},n_{0},c_{1},N_{0},d\right) ,k=1,2,$ so that for $y\in 
\mathbf{R}^{d},$ 
\begin{eqnarray*}
\int \left\vert H^{R}\left( t,x+y\right) -H^{R}\left( t,x\right) \right\vert
dx &\leq &C_{1}\left\vert y\right\vert e^{-C_{2}t}, \\
\int \left\vert H_{0}^{R}\left( t,x+y\right) -H_{0}^{R}\left( t,x\right)
\right\vert dx &\leq &\left\vert y\right\vert \int \left\vert \nabla \tilde{%
\zeta}_{0}\left( x\right) \right\vert dx.
\end{eqnarray*}
\end{lemma}

\begin{proof}
(i) Note that%
\begin{equation*}
\mathcal{F}H^{R}(t,\xi )=\exp \left\{ \psi ^{\tilde{\pi}_{R}}\left( \xi
\right) t\right\} \zeta \left( \xi \right) ,\xi \in \mathbf{R}^{d}.
\end{equation*}%
By \textbf{D}$\left( \kappa ,l\right) ,$ we have $\tilde{\pi}_{R}=\mu
^{0}+\nu _{R}$, and $Z_{t}^{R}=Z_{t}^{\mu ^{0}}+Z_{t}^{v_{R}},t>0,$ (in
distribution)$,$ where $Z^{\mu ^{0}}$ (resp. $Z^{\nu _{R}}$) are independent
Levy processes associated to $\mu ^{0}$ (resp. $\nu _{R}).$ Hence%
\begin{equation*}
H^{R}(t,\cdot )=F\left( t,\cdot \right) \ast P_{t},
\end{equation*}%
where%
\begin{equation*}
F\left( t,x\right) =\mathcal{F}^{-1}\left[ \exp \left\{ \psi ^{\mu
_{0}}t\right\} \zeta \right] \left( x\right) =\mathbf{E}\tilde{\zeta}\left(
x+Z_{t}^{\mu ^{0}}\right) ,t\geq 0,x\in \mathbf{R}^{d},
\end{equation*}%
and $P_{t}\left( dy\right) $ is the distribution of $Z_{t}^{v_{R}}$. By
Plancherel (recall assumption \textbf{A}$_{0}$ holds for $\mu ^{0}$), there
are constants $C_{k}=C_{k}\left( R_{1},R_{2},N_{0},n_{0},d\right) ,k=1,2,$
so that for any multiindices $\gamma ,\left\vert \gamma \right\vert \leq
n=d_{0}+2=\left[ \frac{d}{2}\right] +3,$ 
\begin{eqnarray*}
\int \left\vert x^{\gamma }F\left( t,x\right) \right\vert ^{2}dx &\leq
&C\int \left\vert D^{\gamma }\left[ \zeta \left( \xi \right) \exp \left\{
\psi ^{\mu ^{0}}\left( \xi \right) t\right\} \right] \right\vert ^{2}d\xi \\
&\leq &C_{1}e^{-C_{2}t},t\geq 0.
\end{eqnarray*}%
By Cauchy-Schwarz inequality, 
\begin{eqnarray*}
&&\int \left( 1+\left\vert x\right\vert ^{2}\right) \left\vert F\left(
t,x\right) \right\vert dx \\
&=&\int \left( 1+\left\vert x\right\vert ^{2}\right) \left( 1+\left\vert
x\right\vert \right) ^{-d_{0}}\left\vert F\left( t,x\right) \right\vert
\left( 1+\left\vert x\right\vert \right) ^{d_{0}}dx \\
&\leq &\left( \int \left( 1+\left\vert x\right\vert \right)
^{-2d_{0}}dx\right) ^{1/2}\left( \int (1+\left\vert x\right\vert
)^{4}\left\vert F\left( t,x\right) \right\vert ^{2}\left( 1+\left\vert
x\right\vert \right) ^{2d_{0}}dx\right) ^{1/2} \\
&\leq &C\left( \int F\left( t,x\right) ^{2}\left( 1+\left\vert x\right\vert
^{2}\right) ^{d_{0}+2}dx\right) ^{1/2}\leq C_{1}\exp \left\{ -C_{2}t\right\}
,t\geq 0.
\end{eqnarray*}%
By Lemma \ref{al00}, there is $C=C\left( N_{0}\right) $ so that 
\begin{equation*}
\mathbf{E}\left[ \left\vert Z_{t}^{\nu _{R}}\right\vert ^{\alpha _{2}}\right]
=\int \left\vert y\right\vert ^{\alpha _{2}}P_{t}\left( dy\right) \leq
C\left( 1+t\right) ,t\geq 0.
\end{equation*}%
Hence there are constants $C_{k}=C_{k}\left(
R_{1},R_{2},N_{1},n_{0},c_{1},N_{0},d\right) ,k=1,2,$ so that%
\begin{eqnarray*}
&&\int \left\vert x\right\vert ^{\alpha _{2}}\left\vert H^{R}\left(
t,x\right) \right\vert dx=\int \left\vert x\right\vert ^{\alpha
_{2}}\left\vert \int F\left( t,x-y\right) P_{t}\left( dy\right) \right\vert
dx \\
&\leq &\int \int \left\vert x\right\vert ^{\alpha _{2}}\left\vert F\left(
t,x-y\right) \right\vert P_{t}\left( dy\right) dx\leq \int \int \left\vert
x-y\right\vert ^{\alpha _{2}}\left\vert F\left( t,x-y\right) \right\vert
P_{t}\left( dy\right) dx \\
&&+\int \int \left\vert y\right\vert ^{\alpha _{2}}\left\vert F\left(
t,x-y\right) \right\vert P_{t}\left( dy\right) dx\leq C_{1}e^{-C_{2}t},t\geq
0.
\end{eqnarray*}%
Now, by Lemma \ref{al00},%
\begin{eqnarray*}
&&\int \left\vert x\right\vert ^{\alpha _{2}}\left\vert \mathbf{E}\tilde{%
\zeta}_{0}\left( x+Z_{t}^{R}\right) \right\vert dx \\
&\leq &\mathbf{E}\int \left\vert x+Z_{t}^{R}\right\vert ^{\alpha
_{2}}\left\vert \tilde{\zeta}_{0}\left( x+Z_{t}^{R}\right) \right\vert dx+%
\mathbf{E}\left[ \left\vert Z_{t}^{R}\right\vert ^{\alpha _{2}}\right] \int
\left\vert \tilde{\zeta}_{0}\left( x\right) \right\vert dx \\
&\leq &C\left( 1+t\right) .
\end{eqnarray*}

(ii) Similarly as in part (i), for $y\in \mathbf{R}^{d},$%
\begin{eqnarray*}
&&\int \left\vert H^{R}\left( t,x+y\right) -H^{R}\left( t,x\right)
\right\vert dx \\
&=&\int \left\vert \int \int_{0}^{1}\nabla F\left( t,x+sy-z\right) \cdot
ydsP_{t}\left( dz\right) \right\vert dx \\
&\leq &\left\vert y\right\vert \int \left\vert DF\left( t,x\right)
\right\vert dx\leq C_{1}\left\vert y\right\vert e^{-C_{2}t},t>0,
\end{eqnarray*}%
and directly%
\begin{equation*}
\int \left\vert H_{0}^{R}\left( t,x+y\right) -H_{0}^{R}\left( t,x\right)
\right\vert dx\leq \left\vert y\right\vert \int \left\vert \nabla \tilde{%
\zeta}_{0}\left( x\right) \right\vert dx.
\end{equation*}
\end{proof}

\begin{lemma}
\label{l10}Let \textbf{D}$\left( \kappa ,l\right) $ and \textbf{B}$\left(
\kappa ,l\right) $ hold for $\mu \in \mathfrak{A}^{\sigma }$ with scaling
function $\kappa $ and scaling factor $l,N>1.$ Let $Z_{t}^{j}=Z_{t}^{\tilde{%
\mu}_{N^{-j}}}$ be the Levy process associated to $L^{\tilde{\mu}%
_{N^{-j}}},j\geq 1$, and $Z_{t}=Z_{t}^{\mu }$. Let $\zeta ,\zeta _{0}\in
C_{0}^{\infty }\left( \mathbf{R}^{d}\right) $ be such that $0\notin $supp$%
\left( \zeta \right) .$ Let 
\begin{equation*}
\Phi _{j}\left( x\right) =\int_{0}^{\infty }e^{-\kappa \left( N^{-j}\right)
t}\mathbf{E}\tilde{\zeta}\left( x+Z_{t}^{j}\right) dt,j\geq 1,x\in \mathbf{R}%
^{d},
\end{equation*}%
\begin{equation*}
\Phi _{0}\left( x\right) =\int_{0}^{\infty }e^{-t}\mathbf{E}\tilde{\zeta}%
_{0}\left( x+Z_{t}\right) dt,x\in \mathbf{R}^{d},
\end{equation*}%
where $\tilde{\zeta}=\mathcal{F}^{-1}\zeta ,\tilde{\zeta}_{0}=\mathcal{F}%
^{-1}\zeta _{0}$. Let $K_{j}\left( x\right) =N^{jd}\Phi _{j}\left(
N^{j}x\right) ,j\geq 0,x\in \mathbf{R}^{d}.$

Then for $1<p,r<\infty $ there is a constant $C_{p,r}$ such that for all $%
f=\left( f_{j}\right) \in L_{p}\left( \mathbf{R}^{d},l_{r}\right) $ 
\begin{eqnarray}
\left\vert \left( \sum_{j=0}^{\infty }\left\vert K_{j}\ast f_{j}\right\vert
^{r}\right) ^{1/r}\right\vert _{L_{p}\left( \mathbf{R}^{d}\right) } &\leq
&C_{p,r}\left\vert \left( \sum_{j=0}^{\infty }\left\vert f_{j}\ast \tilde{%
\zeta}_{j}\right\vert ^{r}\right) ^{1/r}\right\vert _{L_{p}\left( \mathbf{R}%
^{d}\right) }  \label{ff11} \\
&\leq &C_{p,r}\left\vert \left( \sum_{j=0}^{\infty }\left\vert
f_{j}\right\vert ^{r}\right) ^{1/r}\right\vert _{L_{p}\left( \mathbf{R}%
^{d}\right) },  \notag
\end{eqnarray}%
where $\tilde{\zeta}_{j}=\mathcal{F}^{-1}\left[ \zeta \left( N^{-j}\cdot
\right) \right] ,j\geq 1.$

If $r=2$, then (\ref{ff11}) holds for a separable Hilbert space $\mathbf{G}$%
-valued sequences $f=\left( f_{j}\right) \in L_{p}\left( \mathbf{R}%
^{d},l_{r}\left( \mathbf{G}\right) \right) $ (simply absolute value in (\ref%
{ff11}) is replaced by $\mathbf{G}$-norm).

In particular, there is a constant $C$ so that 
\begin{equation}
\left\vert K_{j}\ast f\right\vert _{L_{p}\left( \mathbf{R}^{d};\mathbf{G}%
\right) }\leq C\left\vert f\right\vert _{L_{p}\left( \mathbf{R}^{d};\mathbf{G%
}\right) },j\geq 0,f\in L_{p}\left( \mathbf{R}^{d};\mathbf{G}\right) .
\label{f0}
\end{equation}
\end{lemma}

\begin{proof}
Let $\eta ,\eta _{0}\in C_{0}^{\infty }\left( \mathbf{R}^{d}\right) $ be
such that%
\begin{equation*}
\eta \zeta =\zeta ,\eta _{0}\zeta _{0}=\zeta _{0}.
\end{equation*}%
Let $\tilde{\eta}=\mathcal{F}^{-1}\eta ,\tilde{\eta}_{0}=\mathcal{F}%
^{-1}\eta _{0}$, and 
\begin{equation*}
\tilde{\Phi}_{j}\left( x\right) =\int_{0}^{\infty }e^{-\kappa \left(
N^{-j}\right) t}\mathbf{E}\tilde{\eta}\left( x+Z_{t}^{j}\right) dt,j\geq
1,x\in \mathbf{R}^{d},
\end{equation*}%
\begin{equation*}
\tilde{\Phi}_{0}\left( x\right) =\int_{0}^{\infty }e^{-t}\mathbf{E}\tilde{%
\eta}_{0}\left( x+Z_{t}\right) dt,x\in \mathbf{R}^{d},
\end{equation*}

Let $\tilde{K}_{j}\left( x\right) =N^{jd}\tilde{\Phi}_{j}\left(
N^{j}x\right) ,j\geq 0,x\in \mathbf{R}^{d}.$ Obviously,%
\begin{equation*}
K_{j}\ast f=\tilde{K}_{j}\ast f\ast \tilde{\zeta}_{j},j\geq 0.
\end{equation*}

We will check the assumptions of Lemma \ref{l1} for $\tilde{\Phi}_{j},j\geq
0 $.

(i) We will prove that%
\begin{equation}
\int \left\vert x\right\vert ^{\alpha _{2}}\left\vert \tilde{\Phi}_{j}\left(
x\right) \right\vert dx\leq A,j\geq 0,  \label{ff10}
\end{equation}%
where $\alpha _{2}$ is exponent in \textbf{B}$\left( \kappa ,l\right) $. By
Lemma \ref{auxl1}, there is a constant $C=C\left( N_{0}\right) $ so that%
\begin{eqnarray*}
&&\int \left\vert x\right\vert ^{\alpha _{2}}\left\vert \tilde{\Phi}%
_{0}\left( x\right) \right\vert dx\leq \int \int_{0}^{\infty
}e^{-t}\left\vert x\right\vert ^{\alpha _{2}}\left\vert \mathbf{E}\tilde{\eta%
}_{0}\left( x+Z_{t}\right) \right\vert dtdx \\
&\leq &C\int_{0}^{\infty }e^{-t}\left( 1+t\right) dt,
\end{eqnarray*}%
and%
\begin{eqnarray*}
\int \left\vert x\right\vert ^{\alpha _{2}}\left\vert \tilde{\Phi}_{j}\left(
x\right) \right\vert dx &\leq &\int_{0}^{\infty }\int \left\vert
x\right\vert ^{\alpha _{2}}\left\vert \mathbf{E}\tilde{\eta}\left(
x+Z_{t}^{j}\right) \right\vert dxdt \\
&\leq &C\int_{0}^{\infty }C_{1}e^{-C_{2}t}dt,j\geq 1.
\end{eqnarray*}

(ii) We prove 
\begin{equation}
\int \left\vert \tilde{\Phi}_{j}\left( x+y\right) -\tilde{\Phi}_{j}\left(
x\right) \right\vert dx\leq A\left\vert y\right\vert ,\left\vert
y\right\vert \leq 1,j\geq 0.  \label{ff4}
\end{equation}

By Lemma \ref{auxl1}, for any $y\in \mathbf{R}^{d},$%
\begin{equation}
\int \left\vert \tilde{\Phi}_{0}\left( x+y\right) -\tilde{\Phi}_{0}\left(
x\right) \right\vert \leq C\left\vert y\right\vert \int_{0}^{\infty
}e^{-t}dt,  \notag
\end{equation}%
and 
\begin{eqnarray*}
&&\int \left\vert \tilde{\Phi}_{j}\left( x+y\right) -\tilde{\Phi}_{j}\left(
x\right) \right\vert dx \\
&\leq &\int_{0}^{\infty }\int \left\vert \mathbf{E}\tilde{\eta}\left(
x+y+Z_{t}^{j}\right) -\mathbf{E}\tilde{\eta}\left( x+Z_{t}^{j}\right)
\right\vert dxdt\leq C_{1}\left\vert y\right\vert \int_{0}^{\infty
}e^{-C_{2}t}dt,j\geq 1.
\end{eqnarray*}

(iii) We prove that%
\begin{equation}
\left\vert \mathcal{F}\tilde{\Phi}_{j}\left( \xi \right) \right\vert \leq
A,j\geq 1,\xi \in \mathbf{R}^{d}.  \label{f8}
\end{equation}%
Indeed, by Lemma 7 in \cite{MPh} there is $c>0$ so that%
\begin{equation*}
\exp \left\{ \psi ^{\tilde{\pi}_{N^{-j}}}\left( \xi \right) t\right\} \leq
e^{-ct},t>0,1
\end{equation*}%
we have%
\begin{equation*}
\left\vert \mathcal{F}\tilde{\Phi}_{j}\left( \xi \right) \right\vert \leq
\int_{0}^{\infty }e^{-ct}\eta \left( \xi \right) dt\leq A,\xi \in \mathbf{R}%
^{d},j\geq 1,
\end{equation*}%
and, obviously,%
\begin{equation*}
\left\vert \mathcal{F}\tilde{\Phi}_{0}\left( \xi \right) \right\vert \leq
\int_{0}^{\infty }e^{-t}\left\vert \exp \left\{ \psi ^{\mu }\left( \xi
\right) t\right\} \eta _{0}\left( \xi \right) \right\vert dt\leq A.
\end{equation*}%
Therefore (\ref{ff11}) follows from Corollary \ref{c1}.

(iv) The estimate (\ref{f0}) is an obvious consequence of (\ref{ff11}) (take 
$f=\left( f_{k}\right) $ with $f_{k}=0$ if $k\neq j.$

The statement is proved.
\end{proof}

\begin{lemma}
\label{l2}Let \textbf{B}$\left( \kappa ,l\right) $ hold for $\mu \in 
\mathfrak{A}^{\sigma }$ with scaling function $\kappa $ and scaling factor $%
l $. Let $N>1,\zeta ,\zeta _{0}\in C_{0}^{\infty }\left( \mathbf{R}%
^{d}\right) $, $\tilde{\zeta}=\mathcal{F}^{-1}\zeta ,\tilde{\zeta}_{0}=%
\mathcal{F}^{-1}\zeta _{0}$, and 
\begin{eqnarray*}
\Phi _{j}\left( x\right) &=&\kappa \left( N^{-j}\right) \tilde{\zeta}\left(
x\right) -L^{\tilde{\mu}_{N^{-j}}}\tilde{\zeta}\left( x\right) ,x\in \mathbf{%
R}^{d},j\geq 1, \\
\Phi _{0}\left( x\right) &=&\tilde{\zeta}_{0}\left( x\right) -L^{\pi ^{0}}%
\tilde{\zeta}_{0}\left( x\right) ,x\in \mathbf{R}^{d}.
\end{eqnarray*}%
Let $K_{j}\left( x\right) =N^{jd}\Phi _{j}\left( N^{j}x\right) ,j\geq 0,x\in 
\mathbf{R}^{d}.$

Then for $1<p,r<\infty $ there is a constant $C_{p,r}$ such that for all $%
f=\left( f_{j}\right) \in L_{p}\left( \mathbf{R}^{d},l_{r}\right) $ 
\begin{eqnarray}
\left\vert \left( \sum_{j=0}^{\infty }\left\vert K_{j}\ast f_{j}\right\vert
^{r}\right) ^{1/r}\right\vert _{L_{p}\left( \mathbf{R}^{d}\right) } &\leq
&C_{p,r}\left\vert \left( \sum_{j=0}^{\infty }\left\vert f_{j}\ast \tilde{%
\zeta}_{j}\right\vert ^{r}\right) ^{1/r}\right\vert _{L_{p}\left( \mathbf{R}%
^{d}\right) }  \label{f11} \\
&\leq &C_{p,r}\left\vert \left( \sum_{j=0}^{\infty }\left\vert
f_{j}\right\vert ^{r}\right) ^{1/r}\right\vert _{L_{p}\left( \mathbf{R}%
^{d}\right) },  \label{f12}
\end{eqnarray}%
where $\tilde{\zeta}_{j}=\mathcal{F}^{-1}\left[ \zeta \left( N^{-j}\cdot
\right) \right] ,j\geq 1.$ If $r=2$, then ((\ref{f11}), (\ref{f12}) hold for
a separable Hilbert space $\mathbf{G}$-valued sequences $f=\left(
f_{j}\right) \in L_{p}\left( \mathbf{R}^{d},l_{r}\left( \mathbf{G}\right)
\right) $ (simply absolute value is replaced by $\mathbf{G}$-norm).

In particular, there is a constant $C$ so that 
\begin{equation}
\left\vert K_{j}\ast f\right\vert _{L_{p}\left( \mathbf{R}^{d};\mathbf{G}%
\right) }\leq C\left\vert f\right\vert _{L_{p}\left( \mathbf{R}^{d};\mathbf{G%
}\right) },j\geq 0,f\in L_{p}\left( \mathbf{R}^{d};\mathbf{G}\right) .
\label{f00}
\end{equation}
\end{lemma}

\begin{proof}
Let $\eta ,\eta _{0}\in C_{0}^{\infty }\left( \mathbf{R}^{d}\right) $ and%
\begin{equation*}
\eta \zeta =\zeta ,\eta _{0}\zeta _{0}=\zeta _{0},
\end{equation*}%
$\tilde{\eta}=\mathcal{F}^{-1}\eta ,\tilde{\eta}_{0}=\mathcal{F}^{-1}\eta
_{0}$, and%
\begin{eqnarray*}
\tilde{\Phi}_{j}\left( x\right) &=&\kappa \left( N^{-j}\right) \tilde{\eta}%
\left( x\right) -L^{\tilde{\mu}_{N^{-j}}}\tilde{\eta}\left( x\right) ,x\in 
\mathbf{R}^{d},j\geq 1, \\
\tilde{\Phi}_{0}\left( x\right) &=&\tilde{\eta}_{0}\left( x\right) -L^{\pi
^{0}}\tilde{\eta}_{0}\left( x\right) ,x\in \mathbf{R}^{d}.
\end{eqnarray*}%
Let $\tilde{K}_{j}\left( x\right) =N^{jd}\tilde{\Phi}_{j}\left(
N^{j}x\right) ,j\geq 0,x\in \mathbf{R}^{d}.$ Again,%
\begin{equation*}
K_{j}\ast f=\tilde{K}_{j}\ast f\ast \tilde{\zeta}_{j},j\geq 0.
\end{equation*}

We will check the assumptions of Corollary \ref{c1} for $\tilde{\Phi}%
_{j},j\geq 0$.

(i) First we prove that 
\begin{equation*}
\int \left\vert x\right\vert ^{\alpha _{2}}\left\vert \tilde{\Phi}_{j}\left(
x\right) \right\vert dx\leq A,j\geq 0.
\end{equation*}%
Obviously,%
\begin{eqnarray*}
\int \kappa \left( N^{-j}\right) \left\vert x\right\vert ^{\alpha
_{2}}\left\vert \tilde{\eta}\left( x\right) \right\vert dx &\leq &l\left(
1\right) \kappa \left( 1\right) \int \left\vert x\right\vert ^{\alpha
_{2}}\left\vert \tilde{\eta}\left( x\right) \right\vert dx<\infty , \\
\int \left\vert x\right\vert ^{\alpha _{2}}\left\vert \tilde{\eta}_{0}\left(
x\right) \right\vert dx &<&\infty .
\end{eqnarray*}%
We split%
\begin{eqnarray*}
L^{\tilde{\mu}_{N^{-j}}}\tilde{\eta}\left( x\right) &=&\int_{\left\vert
z\right\vert \leq 1}[\tilde{\eta}\left( x+z\right) -\tilde{\eta}\left(
x\right) -\chi _{\sigma }\left( z\right) z\cdot \nabla \tilde{\eta}\left(
x\right) ]\tilde{\mu}_{N^{-j}}\left( dz\right) \\
&&+\int_{\left\vert z\right\vert >1}[\tilde{\eta}\left( x+z\right) -\tilde{%
\eta}\left( x\right) -\chi _{\sigma }\left( z\right) z\cdot \nabla \tilde{%
\eta}\left( x\right) ]\tilde{\mu}_{N^{-j}}\left( dz\right) \\
&=&A_{j}\left( x\right) +B_{j}\left( x\right) ,x\in \mathbf{R}^{d},j\geq 1.
\end{eqnarray*}%
For $\sigma \in \lbrack 1,2),$%
\begin{equation*}
A_{j}\left( x\right) =\int_{\left\vert z\right\vert \leq 1}\int_{0}^{1}(1-s)%
\tilde{\eta}_{x_{i}x_{j}}\left( x+sz\right) z_{i}z_{j}\tilde{\mu}%
_{N^{-j}}\left( dz\right) ds,
\end{equation*}%
and 
\begin{eqnarray*}
&&\int \left\vert x\right\vert ^{\alpha _{2}}\left\vert A_{j}\left( x\right)
\right\vert dx \\
&\leq &C\int_{0}^{1}\int \int_{\left\vert z\right\vert \leq 1}\left\vert
x+sz\right\vert ^{\alpha _{2}}\left\vert D^{2}\tilde{\eta}\left( x+sz\right)
\right\vert \left\vert z\right\vert ^{2}\tilde{\mu}_{N^{-j}}\left( dz\right)
dxds \\
&&+C\int_{0}^{1}\int \int_{\left\vert z\right\vert \leq 1}\left\vert D^{2}%
\tilde{\eta}\left( x+sz\right) \right\vert \left\vert z\right\vert
^{2+\alpha _{2}}\tilde{\mu}_{N^{-j}}\left( dz\right) dxds \\
&\leq &C\int (\left\vert x\right\vert ^{\alpha _{2}}+1)\left\vert D^{2}%
\tilde{\eta}\left( x\right) \right\vert dx,j\geq 1,
\end{eqnarray*}%
For $\sigma \in \left( 0,1\right) ,$%
\begin{equation*}
A_{j}\left( x\right) =\int_{\left\vert z\right\vert \leq
1}\int_{0}^{1}\nabla \tilde{\eta}\left( x+sz\right) \cdot z\tilde{\mu}%
_{N^{-j}}\left( dz\right) ,x\in \mathbf{R}^{d},
\end{equation*}%
and%
\begin{eqnarray*}
&&\int \left\vert x\right\vert ^{\alpha _{2}}\left\vert A_{j}\left( x\right)
\right\vert dx \\
&\leq &\int \int_{\left\vert z\right\vert \leq 1}\int_{0}^{1}\left\vert
x+sz\right\vert ^{\alpha _{2}}\left\vert \nabla \tilde{\eta}\left(
x+sz\right) \right\vert ds\left\vert z\right\vert \tilde{\mu}_{N^{-j}}\left(
dz\right) dx \\
&&+\int_{0}^{1}\int \int_{\left\vert z\right\vert \leq 1}\left\vert
z\right\vert ^{\alpha _{2}}\left\vert \nabla \tilde{\eta}\left( x+sz\right)
\right\vert ~\left\vert z\right\vert \tilde{\mu}_{N^{-j}}\left( dz\right)
dxds \\
&\leq &C\int (\left\vert x\right\vert ^{\alpha _{2}}+1)\left\vert \nabla 
\tilde{\eta}\left( x\right) \right\vert dx\int_{\left\vert z\right\vert \leq
1}\left\vert z\right\vert \tilde{\mu}_{N^{-j}}\left( dz\right) , \\
&\leq &C\int (\left\vert x\right\vert ^{\alpha _{2}}+1)\left\vert \nabla 
\tilde{\eta}\left( x\right) \right\vert dx,j\geq 1.
\end{eqnarray*}

Now$,$%
\begin{eqnarray*}
&&\int \left\vert x\right\vert ^{\alpha _{2}}\left\vert B_{j}\left( x\right)
\right\vert dx \\
&\leq &\int \int_{\left\vert z\right\vert >1}\left\vert x+z\right\vert
^{\alpha _{2}}\left\vert \tilde{\eta}\left( x+z\right) \right\vert \tilde{\mu%
}_{N^{-j}}\left( dz\right) dx \\
&&+\int \int_{\left\vert z\right\vert >1}\left\vert z\right\vert ^{\alpha
_{2}}\left\vert \tilde{\eta}\left( x+z\right) \right\vert \tilde{\mu}%
_{N^{-j}}\left( dz\right) dx+\int \int_{\left\vert z\right\vert
>1}\left\vert x\right\vert ^{\alpha _{2}}\left\vert \tilde{\eta}\left(
x\right) \right\vert \tilde{\mu}_{N^{-j}}\left( dz\right) dx \\
&&+\int \left\vert x\right\vert ^{\alpha _{2}}\left\vert \nabla \tilde{\eta}%
\left( x\right) \right\vert dx\int_{\left\vert z\right\vert >1}\chi _{\sigma
}\left( z\right) \tilde{\mu}_{N^{-j}}\left( dz\right) \\
&\leq &C,j\geq 1.
\end{eqnarray*}%
Similarly, by splitting we show that%
\begin{equation*}
\int \left\vert x\right\vert ^{\alpha _{2}}\left\vert L^{\pi }\tilde{\eta}%
_{0}\left( x\right) \right\vert dx<\infty .
\end{equation*}

(ii) Now we prove that\ 
\begin{equation*}
\int \left\vert \tilde{\Phi}_{j}\left( x+y\right) -\tilde{\Phi}_{j}\left(
x\right) \right\vert dx\leq A\left\vert y\right\vert ,\left\vert
y\right\vert \leq 1,j\geq 0.
\end{equation*}%
First obviously,%
\begin{eqnarray*}
\kappa \left( N^{-j}\right) \int |\tilde{\eta}\left( x+y\right) -\tilde{\eta}%
\left( x\right) |dx &\leq &\kappa \left( N^{-j}\right) \int
\int_{0}^{1}|\nabla \tilde{\eta}\left( x+sy\right) |\left\vert y\right\vert
dsdx \\
&\leq &\kappa \left( 1\right) l\left( 1\right) \left\vert y\right\vert \int
\left\vert \nabla \tilde{\eta}\left( x\right) \right\vert dx
\end{eqnarray*}%
and, similarly,%
\begin{equation*}
\int |\tilde{\eta}_{0}\left( x+y\right) -\tilde{\eta}_{0}\left( x\right)
|dx\leq \left\vert y\right\vert \int \left\vert \nabla \tilde{\eta}%
_{0}\left( x\right) \right\vert dx.
\end{equation*}%
Now, for $\left\vert y\right\vert \leq 1,$%
\begin{eqnarray*}
&&\int \left\vert L^{\tilde{\mu}_{N^{-j}}}\tilde{\eta}\left( x+y\right) -L^{%
\tilde{\mu}_{N^{-j}}}\tilde{\eta}\left( x\right) \right\vert dx \\
&\leq &\int \int_{0}^{1}\left\vert L^{\tilde{\mu}_{N^{-j}}}\nabla \tilde{\eta%
}\left( x+sy\right) \right\vert \left\vert y\right\vert dsdx\leq \left\vert
y\right\vert \int \int_{0}^{1}\left\vert L^{\tilde{\mu}_{N^{-j}}}\nabla 
\tilde{\eta}\left( x\right) \right\vert dsdx \\
&\leq &C\left\vert y\right\vert ,j\geq 1,
\end{eqnarray*}%
and, similarly,%
\begin{eqnarray*}
&&\int \left\vert L^{\pi ^{0}}\tilde{\eta}_{0}\left( x+y\right) -L^{\pi ^{0}}%
\tilde{\eta}_{0}\left( x\right) \right\vert dx \\
&\leq &C\left\vert y\right\vert ,y\in \mathbf{R}^{d}.
\end{eqnarray*}

(iii) We prove that%
\begin{equation}
\left\vert \mathcal{F}\tilde{\Phi}_{j}\left( \xi \right) \right\vert \leq
A,j\geq 1,\xi \in \mathbf{R}^{d}.  \label{fo1}
\end{equation}%
Indeed, by Lemma 7 in \cite{MPh}, there is a constant $C$ independent of $j$
so that 
\begin{equation*}
\left\vert \mathcal{F[}L^{\tilde{\mu}_{N^{-j}}}\tilde{\eta}]\left( \xi
\right) \right\vert =\left\vert \psi ^{\tilde{\mu}_{N^{-j}}}\left( \xi
\right) \eta \left( \xi \right) \right\vert \leq C,j\geq 1,\xi \in \mathbf{R}%
^{d}.
\end{equation*}%
Similarly,%
\begin{equation*}
\left\vert \mathcal{F[}L^{\pi }\tilde{\eta}_{0}]\left( \xi \right)
\right\vert =\left\vert \psi ^{\pi }\left( \xi \right) \eta _{0}\left( \xi
\right) \right\vert \leq C,\xi \in \mathbf{R}^{d}.
\end{equation*}

(iv) We have (\ref{f11}) by Corollary \ref{c1}, and (\ref{f12}) follows by
Corollary \ref{c2}. The estimate (\ref{f00}) follows, obviously, from (\ref%
{f12}).
\end{proof}

Now we prove Proposition \ref{pro1}.

\subsubsection{Proof of Proposition \protect\ref{pro1} (equivalent norms of
Besov spaces)}

Let $p\in \left( 1,\infty \right) ,f\in \mathcal{S}^{\prime }\left( \mathbf{R%
}^{d}\right) $ and $f\ast \varphi _{j}\in L_{p}\left( \mathbf{R}^{d}\right)
,j\geq 0.$ It is enough to prove that for each $s\in \mathbf{R}$ there are
constants $C,c$ (independent of $f$ and $j$) so that%
\begin{equation}
\left\vert J^{s}f\ast \varphi _{j}\right\vert _{L_{p}\left( \mathbf{R}%
^{d}\right) }\leq C\kappa \left( N^{-j}\right) ^{-s}\left\vert f\ast \varphi
_{j}\right\vert _{L_{p}\left( \mathbf{R}^{d}\right) },  \label{cf55}
\end{equation}%
and%
\begin{equation}
\kappa \left( N^{-j}\right) ^{-s}\left\vert f\ast \varphi _{j}\right\vert
_{L_{p}\left( \mathbf{R}^{d}\right) }\leq c\left\vert J^{s}f\ast \varphi
_{j}\right\vert _{L_{p}\left( \mathbf{R}^{d}\right) }.  \label{cf44}
\end{equation}%
First, denoting $\pi =\mu _{sym},$%
\begin{eqnarray*}
Jf\ast \varphi _{j} &=&\mathcal{F}^{-1}\left[ (1-\psi ^{\pi })\phi \left(
N^{-j}\cdot \right) \hat{f}\right] ,j\geq 1, \\
Jf\ast \varphi _{0} &=&\mathcal{F}^{-1}\left[ (1-\psi ^{\pi })\phi _{0}\hat{f%
}\right] ,
\end{eqnarray*}%
and for $\xi \in \mathbf{R}^{d},$%
\begin{eqnarray*}
(1-\psi ^{\pi }\left( \xi \right) )\phi \left( N^{-j}\xi \right) &=&(1-\psi
^{\pi _{N^{-j}}}\left( N^{-j}\xi \right) )\phi \left( N^{-j}\xi \right) \\
&=&\kappa \left( N^{-j}\right) ^{-1}\left[ (\kappa \left( N^{-j}\right)
-\psi ^{\tilde{\pi}_{N^{-j}}}\left( N^{-j}\xi \right) )\phi \left( N^{-j}\xi
\right) \right] .
\end{eqnarray*}%
Hence (\ref{cf55}) with $s=1$ follows by Lemma \ref{l2}. Applying repeatedly
(\ref{cf55}) with $s=1$, we see that (\ref{cf55}) holds for any integer $%
s\geq 0.$

On the other hand, for $j\geq 1$ (recall $\varphi =\mathcal{F}^{-1}\phi $), 
\begin{eqnarray*}
J^{-1}\varphi _{j} &=&\int_{0}^{\infty }e^{-t}\mathbf{E}\varphi _{j}\left(
\cdot +Z_{t}^{\pi }\right) dt=\mathcal{F}^{-1}\int_{0}^{\infty
}e^{-t}e^{\psi ^{\pi }\left( \xi \right) t}\phi \left( N^{-j}\xi \right) dt
\\
&=&\kappa \left( N^{-j}\right) \mathcal{F}^{-1}\int_{0}^{\infty }e^{-\kappa
\left( N^{-j}\right) t}e^{\psi ^{\tilde{\pi}_{N^{-j}}}\left( N^{-j}\xi
\right) t}\phi \left( N^{-j}\xi \right) dt \\
&=&\kappa \left( N^{-j}\right) N^{jd}\int_{0}^{\infty }e^{-\kappa \left(
N^{-j}\right) t}\mathbf{E}\varphi \left( N^{j}\cdot +Z_{t}^{j}\right) dt,
\end{eqnarray*}%
where $Z_{t}^{j}=Z_{t}^{\tilde{\pi}_{N^{-j}}}$ is the Levy process
associated to $L^{\tilde{\pi}_{N^{-j}}}$. For $j=0,$ 
\begin{equation*}
J^{-1}\varphi _{0}=\mathcal{F}^{-1}\int_{0}^{\infty }e^{-t}e^{\psi ^{\pi
}\left( \xi \right) t}\phi _{0}\left( \xi \right) dt=\int_{0}^{\infty }e^{-t}%
\mathbf{E}\varphi _{0}\left( \cdot +Z_{t}\right) dt,
\end{equation*}%
where $Z_{t}=Z_{t}^{\pi },t>0$. Hence (\ref{cf55}) with $s=-1$ follows by
Lemma \ref{l10}. Applying repeatedly (\ref{cf55}) with $s=-1$, we see that (%
\ref{cf55}) holds for any negative integer $s.$

Applying interpolation inequality we get (\ref{cf55}) for all $s\in \mathbf{R%
}$. Let $k\in \mathbf{Z=}\left\{ 0,\pm 1,\ldots \right\} $ and $s=(1-\theta
)k+\theta \left( k+1\right) \in \left( k,k+1\right) $ with $\theta \in
\left( 0,1\right) $.

According to Theorem 2.4.6 in \cite{fjs}, $H_{p}^{s}\left( \mathbf{R}%
^{d}\right) =\left[ H_{p}^{k},H_{p}^{k+1}\right] _{\theta }$, $H_{p}^{s}$ is
the complex interpolation space between $H_{p}^{k}$ and $H_{p}^{k+1}.$ By
Theorem 1.9.3 in \cite{tr1}, there is a constant $C,$ independent of $f,j$,
so that%
\begin{eqnarray*}
\left\vert J^{s}f\ast \varphi _{j}\right\vert _{L_{p}\left( \mathbf{R}%
^{d}\right) } &=&\left\vert f\ast \varphi _{j}\right\vert _{H_{p}^{s}\left( 
\mathbf{R}^{d}\right) }\leq C\kappa \left( N^{-j}\right) ^{-(1-\theta
)k-\theta \left( k+1\right) }\left\vert f\ast \varphi _{j}\right\vert
_{L_{p}\left( \mathbf{R}^{d}\right) } \\
&=&C\kappa \left( N^{-j}\right) ^{-s}\left\vert f\ast \varphi
_{j}\right\vert _{L_{p}\left( \mathbf{R}^{d}\right) }.
\end{eqnarray*}%
Now, we prove (\ref{cf44}). Let $f\ast \varphi _{j}\in L_{p}\left( \mathbf{R}%
^{d}\right) ,j\geq 0,s\in \mathbf{R}$. By (\ref{cf55}), $J^{s}f\ast \varphi
_{j}\in L_{p}\left( \mathbf{R}^{d}\right) ,s\in \mathbf{R}$, and%
\begin{equation*}
\left\vert f\ast \varphi _{j}\right\vert _{L_{p}\left( \mathbf{R}^{d}\right)
}=\left\vert J^{-s}J^{s}f\ast \varphi _{j}\right\vert _{L_{p}\left( \mathbf{R%
}^{d}\right) }\leq C\kappa \left( N^{-j}\right) ^{s}\left\vert J^{s}f\ast
\varphi _{j}\right\vert _{L_{p}\left( \mathbf{R}^{d}\right) }
\end{equation*}%
and (\ref{cf44}) follows. Thus for $s\in \mathbf{R},p,q\in \left( 1,\infty
\right) ,$ we have $\tilde{B}_{pq}^{\kappa ,N;s}\left( \mathbf{R}^{d}\right)
=B_{pq}^{\mu ,N;s}\left( \mathbf{R}^{d}\right) $ and the norms are
equivalent. In addition, for any $t,s\in \mathbf{R,}$ the mapping \thinspace 
$J^{t}:B_{pq}^{\mu ,N;s}\left( \mathbf{R}^{d}\right) \rightarrow B_{pq}^{\mu
,N;s-t}\left( \mathbf{R}^{d}\right) $ is an isomorphism.

\subsubsection{Proof of Proposition \protect\ref{pro2} (equivalent norms in $%
H_{p}^{s}$)}

We start with

\begin{lemma}
\label{l4}Let $p,q\in \left( 1,\infty \right) $. Then for each integer $m,s$
there is a constant $C$ so that for all $f=\left( f_{j}\right) \in
L_{p}\left( \mathbf{R}^{d},l_{r}\right) $,%
\begin{eqnarray}
&&\left\vert \left( \sum_{j=0}^{\infty }|\kappa \left( N^{-j}\right)
^{m}J^{s}f_{j}\ast \varphi _{j}|^{q}\right) ^{1/q}\right\vert _{L_{p}\left( 
\mathbf{R}^{d}\right) }  \label{cf} \\
&\leq &C\left\vert \left( \sum_{j=0}^{\infty }\left\vert \kappa \left(
N^{-j}\right) ^{m-s}f_{j}\ast \varphi _{j}\right\vert ^{q}\right)
^{1/q}\right\vert _{L_{p}\left( \mathbf{R}^{d}\right) },  \notag \\
&&\left\vert \left( \sum_{j=0}^{\infty }|\kappa \left( N^{-j}\right)
^{m}f_{j}\ast \varphi _{j}|^{q}\right) ^{1/q}\right\vert _{L_{p}\left( 
\mathbf{R}^{d}\right) }  \label{cf1} \\
&\leq &C\left\vert \left( \sum_{j=0}^{\infty }\left\vert \kappa \left(
N^{-j}\right) ^{m+s}J^{s}f_{j}\ast \varphi _{j}\right\vert ^{q}\right)
^{1/q}\right\vert _{L_{p}\left( \mathbf{R}^{d}\right) }.  \notag
\end{eqnarray}

If $q=2$, then (\ref{cf}), (\ref{cf1}) hold for a separable Hilbert space $%
\mathbf{G}$-valued sequences $f=\left( f_{j}\right) \in L_{p}\left( \mathbf{R%
}^{d},l_{r}\left( \mathbf{G}\right) \right) $ (simply absolute values in (%
\ref{cf}), (\ref{cf1}) are replaced by $\mathbf{G}$-norms).
\end{lemma}

\begin{proof}
Denote $\pi =\mu _{sym}$. Let $K_{j}\left( x\right) =N^{jd}\Phi _{j}\left(
N^{j}x\right) ,j\geq 0$, with \ 
\begin{eqnarray*}
\Phi _{j}\left( x\right) &=&\kappa \left( N^{-j}\right) \varphi \left(
x\right) -L^{\tilde{\pi}_{N^{-j}}}\varphi \left( x\right) ,x\in \mathbf{R}%
^{d},j\geq 1, \\
\Phi _{0}\left( x\right) &=&\varphi _{0}\left( x\right) -L^{\pi }\varphi
_{0}\left( x\right) ,x\in \mathbf{R}^{d}.
\end{eqnarray*}

For $f\in \tilde{C}^{\infty }\left( \mathbf{R}^{d}\right) $,%
\begin{eqnarray*}
Jf\ast \varphi _{j} &=&\mathcal{F}^{-1}\left[ (1-\psi ^{\pi })\phi \left(
N^{-j}\cdot \right) \hat{f}\right] \\
&=&\mathcal{F}^{-1}\left[ (\kappa \left( N^{-j}\right) -\psi ^{\tilde{\pi}%
_{N^{-j}}}\left( N^{-j}\cdot \right) )\kappa \left( N^{-j}\right) ^{-1}\phi
\left( N^{-j}\cdot \right) \hat{f}\right] \\
&=&\kappa \left( N^{-j}\right) ^{-1}K_{j}\ast f,j\geq 1, \\
Jf\ast \varphi _{0} &=&K_{0}\ast f.
\end{eqnarray*}%
By Lemma \ref{l2}, for $f=\left( f_{j}\right) $ with $f_{j}\in \tilde{C}%
^{\infty }\left( \mathbf{R}^{d}\right) ,$%
\begin{eqnarray*}
\left\vert \left( \sum_{j=0}^{\infty }|\kappa \left( N^{-j}\right)
^{m}Jf_{j}\ast \varphi _{j}|^{q}\right) ^{1/q}\right\vert _{L_{p}\left( 
\mathbf{R}^{d}\right) } &=&\left\vert \left( \sum_{j=0}^{\infty }|\kappa
\left( N^{-j}\right) ^{m-1}K_{j}\ast f_{j}|^{q}\right) ^{1/q}\right\vert
_{L_{p}\left( \mathbf{R}^{d}\right) } \\
&\leq &C\left\vert \left( \sum_{j=0}^{\infty }\left\vert \kappa \left(
N^{-j}\right) ^{m-1}f_{j}\ast \varphi _{j}\right\vert ^{q}\right)
^{1/q}\right\vert _{L_{p}\left( \mathbf{R}^{d}\right) }.
\end{eqnarray*}%
Applying this inequality repeatedly we find that (\ref{cf}) holds for any $%
s\in \mathbf{N,}m\in \mathbf{Z.}$

Let $Z_{t}^{j}$ be the Levy process associated to $L^{\tilde{\pi}%
_{N^{-j}}},j\geq 1$, and $Z_{t}$ be the Levy process associated to $L^{\pi }$%
. Let $K_{j}\left( x\right) =N^{jd}\Phi _{j}\left( N^{j}x\right) ,j\geq
0,x\in \mathbf{R}^{d},$ with 
\begin{equation*}
\Phi _{j}\left( x\right) =\int_{0}^{\infty }e^{-\kappa \left( N^{-j}\right)
t}\mathbf{E}\varphi \left( x+Z_{t}^{j}\right) dt,j\geq 1,x\in \mathbf{R}^{d},
\end{equation*}%
and 
\begin{equation*}
\Phi _{0}\left( x\right) =\int_{0}^{\infty }e^{-t}\mathbf{E}\varphi
_{0}\left( x+Z_{t}\right) dt,x\in \mathbf{R}^{d}.
\end{equation*}

Then for $f\in \tilde{C}^{\infty }\left( \mathbf{R}^{d}\right) $,%
\begin{eqnarray*}
J^{-1}f\ast \varphi _{j} &=&\mathcal{F}^{-1}\left\{ \int_{0}^{\infty
}e^{-t}\exp \left\{ \psi ^{\pi }t\right\} \phi \left( N^{-j}\cdot \right) 
\hat{f}dt\right\} \\
&=&\mathcal{F}^{-1}\left[ \kappa \left( N^{-j}\right) \int_{0}^{\infty
}e^{-\kappa \left( N^{-j}\right) t}\exp \left\{ \psi ^{\tilde{\pi}%
_{N^{-j}}}\left( N^{-j}\cdot \right) t\right\} \phi \left( N^{-j}\cdot
\right) \hat{f}dt\right] \\
&=&\kappa \left( N^{-j}\right) K_{j}\ast f,\,j\geq 1, \\
J^{-1}f\ast \varphi _{0} &=&K_{0}\ast f.
\end{eqnarray*}%
By Lemma \ref{l10}, for $f=\left( f_{j}\right) $ with $f_{j}\in \tilde{C}%
^{\infty }\left( \mathbf{R}^{d}\right) ,$ 
\begin{equation*}
\left\vert \left( \sum_{j=0}^{\infty }|\kappa \left( N^{-j}\right)
^{m}J^{-1}f_{j}\ast \varphi _{j}|^{q}\right) ^{1/q}\right\vert _{L_{p}\left( 
\mathbf{R}^{d}\right) }\leq C\left\vert \left( \sum_{j=0}^{\infty
}\left\vert \kappa \left( N^{-j}\right) ^{m+1}f_{j}\ast \varphi
_{j}\right\vert ^{q}\right) ^{1/q}\right\vert _{L_{p}\left( \mathbf{R}%
^{d}\right) }.
\end{equation*}%
Applying this inequality repeatedly we find that (\ref{cf}) holds for any
negative integer $s$\textbf{\ }and\textbf{\ }$m\in \mathbf{Z.}$
\end{proof}

First we prove that $\tilde{H}_{p}^{\kappa ,N;s}\left( \mathbf{R}^{d}\right)
=H_{p}^{\mu ;s}\left( \mathbf{R}^{d}\right) $ and the norms are equivalent
in the scalar case, i.e. the sequence $\left( f_{k}\right) _{k\geq 0}$ has
one nonzero component $f_{1}=f$. If $s\in \mathbf{Z}$ ($s$ is an integer),
then by well known characterization of $L_{p}$ and Lemma \ref{l4},%
\begin{eqnarray}
&&\left\vert f\right\vert _{H_{p}^{s}\left( \mathbf{R}^{d}\right) }
\label{p1} \\
&=&\left\vert J^{s}f\right\vert _{L_{p}\left( \mathbf{R}^{d}\right) }\leq
C\left\vert \left( \sum_{j=0}^{\infty }\left\vert J^{s}f\ast \varphi
_{j}\right\vert ^{2}\right) ^{1/2}\right\vert _{L_{p}\left( \mathbf{R}%
^{d}\right) }  \notag \\
&\leq &C\left\vert \left( \sum_{j=0}^{\infty }\left\vert \kappa \left(
N^{-j}\right) ^{-s}f\ast \varphi _{j}\right\vert ^{2}\right)
^{1/2}\right\vert _{L_{p}\left( \mathbf{R}^{d}\right) },f\in \tilde{C}%
^{\infty }\left( \mathbf{R}^{d}\right) .  \label{p2}
\end{eqnarray}%
On the other hand, by Lemma \ref{l4} and characterization of $L_{p}$ again,%
\begin{eqnarray}
&&\left\vert \left( \sum_{j=0}^{\infty }\left\vert \kappa \left(
N^{-j}\right) ^{-s}f\ast \varphi _{j}\right\vert ^{2}\right)
^{1/2}\right\vert _{L_{p}\left( \mathbf{R}^{d}\right) }  \label{p3} \\
&\leq &C\left\vert \left( \sum_{j=0}^{\infty }\left\vert J^{s}f\ast \varphi
_{j}\right\vert ^{2}\right) ^{1/2}\right\vert _{L_{p}\left( \mathbf{R}%
^{d}\right) }  \notag \\
&\leq &C\left\vert J^{s}f\right\vert _{L_{p}\left( \mathbf{R}^{d}\right) }.
\label{p4}
\end{eqnarray}%
for all $f\in \tilde{C}^{\infty }\left( \mathbf{R}^{d}\right) .$

We use interpolation to prove equivalence for all $s\in \mathbf{R}$. Assume $%
s\in \left( m,m+1\right) $ and $s=\left( 1-\theta \right) m+\theta \left(
m+1\right) $ with $m\in \mathbf{Z}.$ Let 
\begin{equation*}
a_{j}^{0}=\kappa \left( N^{-j}\right) ^{-m},j\geq 0,a_{j}^{1}=\kappa \left(
N^{-j}\right) ^{-(m+1)},j\geq 0,a_{j}^{\theta }=\kappa \left( N^{-j}\right)
^{-s},j\geq 0.
\end{equation*}%
Set%
\begin{equation*}
l_{p}^{k}=\left\{ x=\left( x_{j}\right) :\left\vert x\right\vert
_{a^{k},p}=\left( \sum_{j}a_{j}^{k}\left\vert x_{j}\right\vert ^{p}\right)
^{1/p}<\infty \right\} ,k=0,1,\theta .
\end{equation*}

By Theorem 2.4.6 in \cite{fjs} ($\psi ^{\pi }\,$is continuous negative
definite function), 
\begin{equation}
H_{p}^{l}=\left[ H_{p}^{m},H_{p}^{m+1}\right] _{\theta },  \label{if1}
\end{equation}%
the complex interpolation space between $H_{p}^{m}$ and $H_{p}^{m+1}$. By
Theorem 5.5.3 in \cite{bl}, $l_{2}^{\theta }=\left[ l_{2}^{0},l_{2}^{1}%
\right] _{\theta }$, complex interpolation space between $l_{p}^{0}$ and $%
l_{p}^{1}$. Hence by Theorem 1.18.4 in \cite{tr1},%
\begin{equation}
\left[ L_{p}\left( \mathbf{R}^{d};l_{2}^{0}\right) ,L_{p}\left( \mathbf{R}%
^{d};l_{2}^{1}\right) \right] _{\theta }=L_{p}\left( \mathbf{R}%
^{d};l_{2}^{\theta }\right) ,  \label{if2}
\end{equation}%
the complex interpolation space between $L_{p}\left( \mathbf{R}%
^{d};l_{2}^{0}\right) $ and $L_{p}\left( \mathbf{R}^{d};l_{2}^{1}\right) .$

Consider the mapping 
\begin{equation*}
S:H_{p}^{m}\left( \mathbf{R}^{d}\right) \backepsilon f\mapsto \left( f\ast
\varphi _{j}\right) _{j\geq 0}\in L_{p}\left( \mathbf{R}^{d};l_{2}^{0}%
\right) .
\end{equation*}%
According to Lemma \ref{l4} (see (\ref{p3}), (\ref{p4})), $S:H_{p}^{m}\left( 
\mathbf{R}^{d}\right) \rightarrow L_{p}\left( \mathbf{R}^{d};l_{2}^{0}%
\right) $ is continuous and $S$ maps continuously $H_{p}^{m+1}$ into $%
L_{p}\left( \mathbf{R}^{d};l_{2}^{1}\right) $ (note $H_{p}^{m+1}\subseteq
H_{p}^{m}$). Consider the continuous mapping 
\begin{equation*}
R:L_{p}\left( \mathbf{R}^{d};l_{2}^{0}\right) \backepsilon f=\left(
f_{j}\right) _{j\geq 0}\mapsto \sum_{j=0}^{\infty }f_{j}\ast \tilde{\varphi}%
_{j}\in H_{p}^{m}\left( \mathbf{R}^{d}\right) .
\end{equation*}%
Indeed, if $f=\left( f_{j}\right) _{j\geq 0}\in L_{p}\left( \mathbf{R}%
^{d};l_{2}^{0}\right) $, then $g=\sum_{j=0}^{\infty }f_{j}\ast \tilde{\varphi%
}_{j}\in H_{p}^{m}\left( \mathbf{R}^{d}\right) $, and%
\begin{eqnarray*}
g\ast \varphi _{j} &=&\sum_{k=-2}^{2}f_{j+k}\ast \tilde{\varphi}_{j+k}\ast
\varphi _{j},j\geq 2, \\
g\ast \varphi _{1} &=&\sum_{k=-1}^{2}f_{1+k}\ast \tilde{\varphi}_{1+k}\ast
\varphi _{1},g\ast \varphi _{0}=\sum_{k=0}^{2}f_{k}\ast \tilde{\varphi}%
_{k}\ast \varphi _{0}.
\end{eqnarray*}%
Let%
\begin{eqnarray*}
\tilde{f}_{j} &=&\sum_{k=-2}^{2}f_{j+k}\ast \tilde{\varphi}_{j+k},j\geq 2, \\
\tilde{f}_{1} &=&\sum_{k=-1}^{2}f_{1+k}\ast \tilde{\varphi}_{1+k},\tilde{f}%
_{0}=\sum_{k=0}^{2}f_{k}\ast \tilde{\varphi}_{k}
\end{eqnarray*}%
Hence by Corollary \ref{c2},%
\begin{eqnarray*}
&&\left\vert \left( \sum_{j=0}^{\infty }\left\vert \kappa \left(
N^{-j}\right) ^{-m}g\ast \tilde{\varphi}_{j}\right\vert ^{2}\right)
^{1/2}\right\vert _{L_{p}\left( \mathbf{R}^{d}\right) } \\
&=&\left\vert \left( \sum_{j=0}^{\infty }\left\vert \kappa \left(
N^{-j}\right) ^{-m}\tilde{f}_{j}\ast \tilde{\varphi}_{j}\right\vert
^{2}\right) ^{1/2}\right\vert _{L_{p}\left( \mathbf{R}^{d}\right) }\leq
C\left\vert \left( \sum_{j=0}^{\infty }\left\vert \kappa \left(
N^{-j}\right) ^{-m}\tilde{f}_{j}\right\vert ^{2}\right) ^{1/2}\right\vert
_{L_{p}\left( \mathbf{R}^{d}\right) } \\
&\leq &C\left\vert \left( \sum_{j=0}^{\infty }\left\vert \kappa \left(
N^{-j}\right) ^{-m}f_{j}\ast \tilde{\varphi}_{j}\right\vert ^{2}\right)
^{1/2}\right\vert _{L_{p}\left( \mathbf{R}^{d}\right) }\leq C\left\vert
\left( \sum_{j=0}^{\infty }\left\vert \kappa \left( N^{-j}\right)
^{-m}f_{j}\right\vert ^{2}\right) ^{1/2}\right\vert _{L_{p}\left( \mathbf{R}%
^{d}\right) },
\end{eqnarray*}%
i.e. the mapping $R:L_{p}\left( \mathbf{R}^{d};l_{2}^{0}\right) \rightarrow
H_{p}^{m}\left( \mathbf{R}^{d}\right) $ is continuous. Similarly we prove
that $R:L_{p}\left( \mathbf{R}^{d};l_{2}^{1}\right) \rightarrow
H_{p}^{m+1}\left( \mathbf{R}^{d}\right) $ is continuous. Obviously, $RS=I$
(identity map on $H_{p}^{m}\left( \mathbf{R}^{d}\right) $). Now by (\ref{if1}%
), (\ref{if2}) and Theorem 1.2.4 in \cite{tr1}, $S:H_{p}^{s}\left( \mathbf{R}%
^{d}\right) \rightarrow L_{p}\left( \mathbf{R}^{d};l_{2}^{\theta }\right) $
is isomorphic mapping onto a subspace of $L_{p}\left( \mathbf{R}%
^{d};l_{2}^{\theta }\right) $, i.e., there are constants $0<c_{1}<c_{2}$ so
that%
\begin{equation*}
c_{1}\left\vert Sv\right\vert _{L_{p}\left( \mathbf{R}^{d};l_{2}^{\theta
}\right) }\leq \left\vert v\right\vert _{H_{p}^{s}\left( \mathbf{R}%
^{d}\right) }\leq c_{2}\left\vert Sv\right\vert _{L_{p}\left( \mathbf{R}%
^{d};l_{2}^{\theta }\right) }.
\end{equation*}%
Hence (\ref{p1})-(\ref{p4}) hold for any $s\in \mathbf{R}$.

Now we prove that $\tilde{H}_{p}^{\kappa ,N;s}\left( \mathbf{R}%
^{d};l_{2}\right) =H_{p}^{\mu ;s}\left( \mathbf{R}^{d};l_{2}\right) $ and
the norms are equivalent by reducing it to a scalar case. Let $f=\left(
f_{k}\right) _{k\geq 0}$ with $f_{k}\in \tilde{C}^{\infty }\left( \mathbf{R}%
^{d}\right) $ and only finite number of $f_{k}$ be nonzero$.$ Let $\zeta
_{k},k\geq 0,$ be a sequence of independent standard normal r.v. and%
\begin{equation*}
\xi \left( x\right) =\sum_{k=0}^{\infty }\zeta _{k}f_{k}\left( x\right)
,x\in \mathbf{R}^{d}.
\end{equation*}%
According to (\ref{p1})-(\ref{p4}), there are constants $0<c_{1}<c_{2}$ so
that $\mathbf{P}$-a.s.%
\begin{equation*}
c_{1}\left\vert \xi \right\vert _{\tilde{H}_{p}^{\kappa ,N;s}\left( \mathbf{R%
}^{d}\right) }^{p}\leq \left\vert \xi \right\vert _{H_{p}^{\mu ;s}\left( 
\mathbf{R}^{d}\right) }^{p}\leq c_{2}\left\vert \xi \right\vert _{\tilde{H}%
_{p}^{\kappa ,N;s}\left( \mathbf{R}^{d}\right) }^{p},
\end{equation*}%
and%
\begin{equation*}
c_{1}\mathbf{E}\left\vert \xi \right\vert _{\tilde{H}_{p}^{\kappa
,N;s}\left( \mathbf{R}^{d}\right) }^{p}\leq \mathbf{E}\left\vert \xi
\right\vert _{H_{p}^{\mu ;s}\left( \mathbf{R}^{d}\right) }^{p}\leq c_{2}%
\mathbf{E}\left\vert \xi \right\vert _{\tilde{H}_{p}^{\kappa ,N;s}\left( 
\mathbf{R}^{d}\right) }^{p}.
\end{equation*}%
All the equivalences follow easily from Lemma \ref{gm}.

\subsubsection{Description of function spaces using differences}

We will show that the spaces introduced above belong to the class of spaces
of generalized smoothness studied in \cite{ka}, \cite{kali}, \cite{fl} (see
references therein as well).

\begin{lemma}
\label{ll1}Let $\kappa $ be a scaling function with a scaling factor $l$.
Let $N>1$ be an integer such that $l\left( N^{-1}\right) <1$. Let $%
s>0,\alpha _{k}=\kappa \left( N^{-k}\right) ^{-s},k\geq 0.$

a) There is a constant $\bar{c}>1$ and $0<\theta _{1}\leq \theta _{0}$ so
that%
\begin{eqnarray*}
\bar{c}^{-1}(r^{\theta _{1}}\wedge r^{\theta _{0}}) &\leq &\kappa \left(
r\right) \leq \bar{c}(r^{\theta _{0}}\vee r^{\theta _{1}}),r\geq 0, \\
l\left( r\right) &\geq &\bar{c}^{-1}(r^{\theta _{1}}\wedge r^{\theta
_{0}}),r\geq 0, \\
\left[ \gamma \left( r\right) ^{\theta _{1}}\wedge \gamma \left( r\right)
^{\theta _{0}}\right] &\leq &\bar{c}r,r\geq 0,
\end{eqnarray*}%
where $\gamma \left( r\right) =\inf \left\{ t:l\left( t\right) \geq
r\right\} ,r>0.$

b) There are constants $c,C>0$ and a positive integer $k_{0}$ so that 
\begin{eqnarray*}
\alpha _{k+1} &\leq &C\alpha _{k}\text{ \ for all }k, \\
\alpha _{k} &\geq &c\alpha _{m}\text{ \ for all }k\geq m\text{,}
\end{eqnarray*}%
and%
\begin{equation*}
\alpha _{k}\geq 2\alpha _{m}\text{ for all }k\geq m+k_{0}.
\end{equation*}%
Moreover, for any $q>0,$%
\begin{equation*}
\sum_{k=0}^{\infty }\alpha _{k}^{-q}<\infty .
\end{equation*}

c) Let \textbf{D}$\left( \kappa ,l\right) $ and \textbf{B}$\left( \kappa
,l\right) $ hold for $\pi \in \mathfrak{A}^{\sigma }$. For each $\sigma
^{\prime }\geq \alpha _{1}$ there is a constant $c$ so that
\end{lemma}

\begin{equation*}
\frac{r^{\sigma ^{\prime }}}{\kappa \left( r\right) }\leq c\frac{A^{\sigma
^{\prime }}}{\kappa \left( A\right) }\text{ for any }r\leq A\leq 1;
\end{equation*}%
in addition, $\frac{A^{\sigma ^{\prime }}}{\kappa \left( A\right) }%
\rightarrow 0$ as $A\rightarrow 0$ for any $\sigma ^{\prime }>\sigma .$

\begin{proof}
a) Let $j\geq 0$, 
\begin{equation*}
\kappa \left( N^{-j-1}\right) \leq l\left( N^{-1}\right) \kappa \left(
N^{-j}\right) \leq \ldots \leq l\left( N^{-1}\right) ^{j+1}\kappa \left(
1\right) =c_{0}^{j+1}\kappa \left( 1\right) ,j\geq 0,
\end{equation*}%
and%
\begin{equation*}
\kappa \left( N^{-j-1}\right) \geq l\left( N\right) ^{-1}\kappa \left(
N^{-j}\right) \geq \ldots \geq l\left( N\right) ^{-j-1}\kappa \left(
1\right) =C_{0}^{j+1}\kappa \left( 1\right) ,j\geq 0.
\end{equation*}%
If $r\in \left[ N^{-j-1},N^{-j}\right] ,$ then \ $N^{-1}\leq rN^{j}\leq 1$,
and%
\begin{eqnarray*}
\kappa \left( r\right) &\leq &l\left( rN^{j}\right) \kappa \left(
N^{-j}\right) \leq l\left( 1\right) \kappa \left( 1\right) c_{0}^{j}=\kappa
\left( 1\right) l\left( 1\right) c_{0}^{-1}N^{-(j+1)\log _{N}c_{0}^{-1}} \\
&\leq &\kappa \left( 1\right) l\left( 1\right) l\left( N^{-1}\right)
^{-1}r^{\log _{N}c_{0}^{-1}},
\end{eqnarray*}%
and%
\begin{eqnarray*}
\kappa \left( r\right) &\geq &l\left( \left( N^{j}r\right) ^{-1}\right)
^{-1}\kappa \left( N^{-j}\right) \geq l\left( N\right) ^{-1}C_{0}^{j}\kappa
\left( 1\right) \\
&=&\kappa \left( 1\right) l\left( N\right) ^{-1}N^{-j\log
_{N}C_{0}^{-1}}\geq \kappa \left( 1\right) l\left( N\right) ^{-1}r^{\log
_{N}C_{0}^{-1}},
\end{eqnarray*}%
with%
\begin{equation*}
c_{0}=l\left( N^{-1}\right) ,1<c_{0}^{-1}=l\left( N^{-1}\right) ^{-1}\leq
l\left( N\right) =C_{0}^{-1}.
\end{equation*}%
That is 
\begin{equation*}
\kappa \left( 1\right) l\left( N\right) ^{-1}r^{\theta _{0}}\leq \kappa
\left( r\right) \leq \kappa \left( 1\right) l\left( 1\right) l\left(
N^{-1}\right) ^{-1}r^{\theta _{1}},r\in \left[ 0,1\right] ,
\end{equation*}%
where%
\begin{equation*}
\theta _{0}=\log _{N}l\left( N\right) =\log _{N}C_{0}^{-1}\geq \theta
_{1}=\log _{N}l\left( N^{-1}\right) ^{-1}=\log _{N}c_{0}^{-1}.
\end{equation*}%
Using similar arguments, for $r\in \left[ N^{j},N^{j+1}\right] $
(equivalently $1\leq rN^{-j}\leq N$), we find that%
\begin{eqnarray*}
\kappa \left( r\right) &\leq &l\left( rN^{-j}\right) \kappa \left(
N^{j}\right) \leq l\left( N\right) C_{0}^{-j}\kappa \left( 1\right) =\kappa
\left( 1\right) l\left( N\right) N^{j\log _{N}C_{0}^{-1}} \\
&\leq &\kappa \left( 1\right) l\left( N\right) r^{\log _{N}C_{0}^{-1}},
\end{eqnarray*}%
and%
\begin{eqnarray*}
\kappa \left( r\right) &\geq &l\left( N^{j}/r\right) ^{-1}\kappa \left(
N^{j}\right) \geq l\left( 1\right) ^{-1}l\left( N^{-1}\right) ^{-j}\kappa
\left( 1\right) \\
&\geq &\kappa \left( 1\right) l\left( 1\right) ^{-1}l\left( N^{-1}\right)
r^{\log _{N}c_{0}^{-1}}.
\end{eqnarray*}%
Thus for $r>1,$%
\begin{equation*}
\kappa \left( 1\right) l\left( 1\right) ^{-1}l\left( N^{-1}\right) r^{\log
_{N}c_{0}^{-1}}\leq \kappa \left( r\right) \leq \kappa \left( 1\right)
l\left( N\right) r^{\log _{N}C_{0}^{-1}}.
\end{equation*}%
Summarizing, 
\begin{eqnarray*}
\kappa \left( 1\right) l\left( 1\right) ^{-1}l\left( N^{-1}\right) r^{\theta
_{1}} &\leq &\kappa \left( r\right) \leq \kappa \left( 1\right) l\left(
N\right) r^{\theta _{0}}\text{ if }r>1, \\
\kappa \left( 1\right) l\left( N\right) ^{-1}r^{\theta _{0}} &\leq &\kappa
\left( r\right) \leq \kappa \left( 1\right) l\left( 1\right) l\left(
N^{-1}\right) ^{-1}r^{\theta _{1}}\text{ if }r\in \left[ 0,1\right] ,
\end{eqnarray*}%
and a) follows.

b) Let $s=1$. For $k\geq m,$%
\begin{equation*}
\kappa \left( N^{-k}\right) =\kappa \left( N^{-(k-m)}N^{-m}\right) \leq
l\left( N^{-(k-m)}\right) \kappa \left( N^{-m}\right) \leq l\left( 1\right)
\kappa \left( N^{-m}\right) ;
\end{equation*}%
Let $l\left( N^{-k_{0}}\right) \leq 1/2$. Then for $k\geq k_{0}+m,$ 
\begin{eqnarray*}
\kappa \left( N^{-k}\right) &=&\kappa \left(
N^{-(k-k_{0}-m)}N^{-k_{0}}N^{-m}\right) \leq l\left(
N^{-(k-k_{0}-m)}N^{-k_{0}}\right) \kappa \left( N^{-m}\right) \\
&\leq &l\left( N^{-k_{0}}\right) \kappa \left( N^{-m}\right) \leq \frac{1}{2}%
\kappa \left( N^{-m}\right) .
\end{eqnarray*}%
For any $k\geq 0,$%
\begin{equation*}
\kappa \left( N^{-k}\right) =\kappa \left( N^{-(k+1)}N\right) \leq l\left(
N\right) \kappa \left( N^{-(k+1)}\right) .
\end{equation*}%
Finally, since $\kappa \left( r\right) \leq \bar{c}r^{\theta _{1}},r\in %
\left[ 0,1\right] ,$ we have for any $q>0,$ 
\begin{equation*}
\sum_{k=0}^{\infty }\alpha _{k}^{-q}=\sum_{k=0}^{\infty }\kappa \left(
N^{-k}\right) ^{q}\leq \bar{c}\sum_{k=0}^{\infty }N^{-k\theta _{1}q}<\infty .
\end{equation*}

Similarly we derive the estimates with any $s>0.$

c) Let $A\leq 1$. For any $\sigma ^{\prime }>\sigma $ there is a constant $%
c>0$ such that for any $r\leq A$ with $\mu ^{0}$ from assumption \textbf{D}$%
\left( \kappa ,l\right) ,$ 
\begin{eqnarray}
\int_{\left\vert y\right\vert \leq r}\left\vert y\right\vert ^{\sigma
^{\prime }}d\pi &=&\frac{r^{\sigma ^{\prime }}}{\kappa \left( r\right) }%
\kappa \left( r\right) \int_{\left\vert y\right\vert \leq r}\left\vert
y/r\right\vert ^{\sigma ^{\prime }}d\pi =\frac{r^{\sigma ^{\prime }}}{\kappa
\left( r\right) }\int_{\left\vert y\right\vert \leq 1}\left\vert
y\right\vert ^{\sigma ^{\prime }}d\tilde{\pi}_{r}  \label{cpf2} \\
&\geq &\frac{r^{\sigma ^{\prime }}}{\kappa \left( r\right) }\int_{\left\vert
y\right\vert \leq 1}\left\vert y\right\vert ^{\sigma ^{\prime }}d\mu ^{0}=c%
\frac{r^{\sigma ^{\prime }}}{\kappa \left( r\right) },  \notag
\end{eqnarray}%
because 
\begin{equation*}
\infty >\int_{\left\vert y\right\vert \leq 1}\left\vert y\right\vert
^{\sigma ^{\prime }}\mu ^{0}\left( dy\right) \neq 0\text{ if }\sigma
^{\prime }>\sigma .
\end{equation*}%
If $\sigma ^{\prime }\geq \alpha _{1}$, then, by (\ref{cpf2}) and assumption 
\textbf{B}$\left( \kappa ,l\right) ,$ there is a constant $C$ so that for
any $r\leq A,$%
\begin{eqnarray*}
C\frac{A^{\sigma ^{\prime }}}{\kappa \left( A\right) } &\geq &\frac{%
A^{\sigma ^{\prime }}}{\kappa \left( A\right) }\int_{\left\vert y\right\vert
\leq 1}\left\vert y\right\vert ^{\sigma ^{\prime }}d\tilde{\pi}%
_{A}=\int_{\left\vert y\right\vert \leq A}\left\vert y\right\vert ^{\sigma
^{\prime }}d\pi \geq \int_{\left\vert y\right\vert \leq r}\left\vert
y\right\vert ^{\sigma ^{\prime }}d\pi \\
&\geq &c\frac{r^{\sigma ^{\prime }}}{\kappa \left( r\right) }.
\end{eqnarray*}

The statement is proved.
\end{proof}

\begin{remark}
\label{renew1}Let $N>1,l\left( N^{-1}\right) <1$. Then by Lemma \ref{ll1}, 
\begin{equation*}
\sum_{k=0}^{\infty }\kappa \left( N^{-k}\right) ^{\varepsilon }<\infty
,\varepsilon >0.
\end{equation*}%
Hence for any $\varepsilon >0$ we have the following continuous embeddings:%
\begin{equation*}
\tilde{H}_{p}^{\kappa ,N;s+\varepsilon }\left( \mathbf{R}^{d}\right)
\subseteq \tilde{B}_{pp}^{\kappa ,N;s}\left( \mathbf{R}^{d}\right) \subseteq 
\tilde{H}_{p}^{\kappa ,N;s-\varepsilon }\left( \mathbf{R}^{d}\right) ,p>1.
\end{equation*}
\end{remark}

Let $\kappa $ be a scaling function with a scaling factor $l$. Let $N>1$ be
an integer such that $l\left( N^{-1}\right) <1$. For $p,q\in \left( 1,\infty
\right) ,s>0,$ let $\mathbf{L}_{p,2}^{s}\left( \mathbf{R}^{d}\right) $,
(resp. $\mathbf{B}_{p,q}^{s}\left( \mathbf{R}^{d}\right) $) be the set of
all functions $f\in L_{p}\left( \mathbf{R}^{d}\right) $ that can be
represented by a series of entire functions $f_{k}$ of exponential type $%
N_{k}=N^{k+1},k\geq 0,$ converging in $L_{p}$ 
\begin{equation}
f=\sum_{k=0}^{\infty }f_{k}\text{ in }L_{p}  \label{k1}
\end{equation}%
such that%
\begin{equation}
\left\vert \left( \sum_{k=0}^{\infty }\left\vert \kappa \left( N^{-k}\right)
^{-s}f_{k}\right\vert ^{2}\right) ^{1/2}\right\vert _{L_{p}\left( \mathbf{R}%
^{d}\right) }<\infty  \label{k2}
\end{equation}%
(or resp.%
\begin{equation}
\left\vert f\right\vert _{\mathbf{B}_{p,q}^{s}}=\left( \sum_{k=0}^{\infty
}\left\vert \kappa \left( N^{-k}\right) ^{-s}f_{k}\right\vert _{L_{p}\left( 
\mathbf{R}^{d}\right) }^{q}\right) ^{1/q}<\infty ).  \label{k3}
\end{equation}%
Recall that by Paley-Wiener-Schwartz theorem a function $g\in L_{p}\left( 
\mathbf{R}^{d}\right) $ is entire analytic of type $t$ iff supp$\left( 
\mathcal{F}f\right) \subseteq \left\{ \left\vert \xi \right\vert :\left\vert
\xi \right\vert \leq t\right\} $ (see \cite{tr1}, 2.5.4, p.197). The norm $%
\left\vert f\right\vert _{\mathbf{L}_{p,2}^{s}}$ (resp. $\left\vert
f\right\vert _{\mathbf{B}_{p,q}^{s}}$) is defined as a sum of $\left\vert
f\right\vert _{L_{p}\left( \mathbf{R}^{d}\right) }$ and infimum of (\ref{k2}%
) (resp. (\ref{k3}) over all series (\ref{k1}). The function spaces $\mathbf{%
L}_{p,2}^{s}\left( \mathbf{R}^{d}\right) $, $\mathbf{B}_{p,q}^{s}\left( 
\mathbf{R}^{d}\right) $ belong to the class of spaces of generalized
smoothness (see lemma \ref{ll1}) (see e.g. \cite{kali}). The following
statement holds.

\begin{proposition}
\label{pk1}Let \textbf{D}$\left( \kappa ,l\right) $ and \textbf{B}$\left(
\kappa ,l\right) $ hold for $\mu \in \mathfrak{A}^{\sigma }$ with scaling
function $\kappa $ and scaling factor $l$. Let $s>0,p,q\in \left( 1,\infty
\right) ,N>1,l\left( N^{-1}\right) <1$. Then $H_{p}^{\mu ;s}\left( \mathbf{R}%
^{d}\right) =\mathbf{L}_{p,2}^{s}\left( \mathbf{R}^{d}\right) $, $%
B_{p,q}^{\mu ,N;s}\left( \mathbf{R}^{d}\right) =\mathbf{B}_{p,q}^{s}\left( 
\mathbf{R}^{d}\right) ,$ and the norms are equivalent.
\end{proposition}

\begin{proof}
Let $f\in \tilde{H}_{p}^{\kappa ,N;s}\left( \mathbf{R}^{d}\right) .$ Since $%
f=\sum_{j=0}^{\infty }f_{j}$ (with $f_{j}=f\ast \varphi _{j},$ see
description of the sequence $\varphi _{j}$ in Remark \ref{rem3}) converges
in $L_{p}$, and supp$\left( \mathcal{F}f_{j}\right) \subseteq \left\{
\left\vert \xi \right\vert \leq N^{j+1}\right\} $, it follows that 
\begin{equation*}
\left\vert f\right\vert _{\mathbf{L}_{p,2}^{s}\left( \mathbf{R}^{d}\right)
}\leq \left\vert f\right\vert _{\tilde{H}_{p}^{\kappa ,N;s}\left( \mathbf{R}%
^{d}\right) }.
\end{equation*}%
Let $f\in $ $\mathbf{L}_{p,q}^{s}\left( \mathbf{R}^{d}\right) $, and $K_{j}=%
\left[ -N^{j+1},N^{j+1}\right] ^{d}\backslash \left[ -N^{j},N^{j}\right]
^{d},j\geq 1,K_{0}=\left[ -N,N\right] ^{d}.$ Let $h_{j}=\mathcal{F}%
^{-1}K_{j},j\geq 0$. By Theorem 1 in \cite{ka}, 
\begin{equation*}
f=\sum_{k=0}^{\infty }f\ast h_{k}\text{ in }L_{p},
\end{equation*}%
and the norm $\left\vert f\right\vert _{\mathbf{L}_{p,2}^{s}\left( \mathbf{R}%
^{d}\right) }$ is equivalent to the norm%
\begin{equation*}
\left\vert f\right\vert _{\mathbf{L}_{p,2}^{s}\left( \mathbf{R}^{d}\right)
}^{\symbol{126}}=\left\vert f\right\vert _{L_{p}\left( \mathbf{R}^{d}\right)
}+\left\vert \left( \sum_{j=0}^{\infty }\left\vert \kappa \left(
N^{-j}\right) ^{-s}f\ast h_{j}\right\vert ^{2}\right) ^{1/2}\right\vert
_{L_{p}\left( \mathbf{R}^{d}\right) }.
\end{equation*}%
Now,%
\begin{equation*}
f\ast \varphi _{j}=\sum_{k=(j-2-p_{0})\vee 0}^{j}f\ast h_{k}\ast \varphi
_{j}=\tilde{f}_{j}\ast \varphi _{j},j\geq 0,
\end{equation*}%
with 
\begin{equation*}
\tilde{f}_{j}=\sum_{k=(j-2-p_{0})\vee 0}^{j}f\ast h_{k},j\geq 0.
\end{equation*}%
where $p_{0}$ is the smallest positive integer so that $\sqrt{d}%
/N^{p_{0}}\leq 1.$ Since for $(j-2-p_{0})\vee 0\leq k\leq j$ we have $1\leq
N^{j}N^{-k}\leq N^{p_{0}+2}$ and%
\begin{eqnarray*}
\kappa \left( N^{-k}\right) &=&\kappa \left( N^{j}N^{-k}N^{-j}\right) \leq
l\left( N^{j}N^{-k}\right) \kappa \left( N^{-j}\right) \\
&\leq &l\left( N^{p_{0}+2}\right) \kappa \left( N^{-j}\right) ,
\end{eqnarray*}%
it follows by Corollary \ref{c2}, 
\begin{eqnarray*}
\left\vert f\right\vert _{\tilde{H}_{p}^{\kappa ,N;s}\left( \mathbf{R}%
^{d}\right) } &\leq &C_{p,q}\left\vert \left( \sum_{j=0}^{\infty }\left\vert
\kappa \left( N^{-j}\right) ^{-s}\tilde{f}_{j}\right\vert ^{2}\right)
^{1/2}\right\vert _{L_{p}\left( \mathbf{R}^{d}\right) } \\
&\leq &C\left\vert f\right\vert _{\mathbf{L}_{p,2}^{s}\left( \mathbf{R}%
^{d}\right) }^{\symbol{126}}
\end{eqnarray*}%
Similarly we prove that $B_{p,q}^{s}\left( \mathbf{R}^{d}\right) =\mathbf{B}%
_{p,q}^{s}\left( \mathbf{R}^{d}\right) $ and the norms are equivalent.
\end{proof}

We apply the results in \cite{kali} to describe the norms by averaged local
oscillations. Given $f:\mathbf{R}^{d}\rightarrow \mathbf{R}$ and $y\in 
\mathbf{R}^{d}$, let%
\begin{equation*}
\Delta _{y}f\left( x\right) =\Delta _{y}^{1}f\left( x\right) =f\left(
x+y\right) -f\left( x\right) ,x\in \mathbf{R}^{d}.
\end{equation*}%
Then%
\begin{equation*}
\Delta _{y}^{m}f\left( x\right) =\sum_{j=0}^{m}\left( -1\right) ^{j}\binom{m%
}{j}f\left( x+jy\right) ,x,y\in \mathbf{R}^{d}.
\end{equation*}%
Let%
\begin{equation*}
Q_{t}^{m}f\left( x\right) =\int_{\left\vert y\right\vert \leq 1}\left\vert
\Delta _{ty}^{m}f\left( x\right) \right\vert dy,x\in \mathbf{R}^{d},t>0.
\end{equation*}%
A simple consequence of Theorem 4.2 in \cite{kali} is the following
statement.

\begin{proposition}
\label{pk2}Let \textbf{D}$\left( \kappa ,l\right) $ and \textbf{B}$\left(
\kappa ,l\right) $ hold for $\mu \in \mathfrak{A}^{\sigma }$ with scaling
function $\kappa $ and scaling factor $l$. Let $s>0,p,q\in \left( 1,\infty
\right) $. Let $m_{0}$ be the least integer $m$ such that $m>s\alpha _{1}$ ($%
\alpha _{1}$ is exponent in assumption \textbf{B}$\left( \kappa ,l\right) $%
). Then

(i) For any $m\geq m_{0}$ the norm of $H_{p}^{\mu ;s}\left( \mathbf{R}%
^{d}\right) $ is equivalent to the norm%
\begin{equation*}
\left\vert \left\vert f\right\vert \right\vert _{H_{p}^{\mu ;s}\left( 
\mathbf{R}^{d}\right) }=\left\vert f\right\vert _{L_{p}\left( \mathbf{R}%
^{d}\right) }+\left\vert \left( \int_{0}^{1}\left\vert Q_{t}^{m}f\left(
x\right) \right\vert ^{2}\frac{dt}{t\kappa \left( t\right) ^{2s}}\right)
^{1/2}\right\vert _{L_{p}\left( \mathbf{R}^{d}\right) }.
\end{equation*}

(ii) For any $m\geq m_{0}$ the norm of $B_{p,q}^{\mu ,N;s}\left( \mathbf{R}%
^{d}\right) $ is equivalent to the norm%
\begin{equation*}
\left\vert \left\vert f\right\vert \right\vert _{B_{p,q}^{\mu ,N;s}\left( 
\mathbf{R}^{d}\right) }=\left\vert f\right\vert _{L_{p}\left( \mathbf{R}%
^{d}\right) }+\left( \int_{0}^{1}\left\vert Q_{t}^{m}f\right\vert
_{L_{p}\left( \mathbf{R}^{d}\right) }^{q}\frac{dt}{t\kappa \left( t\right)
^{qs}}\right) ^{1/q}.
\end{equation*}
\end{proposition}

\begin{proof}
In order to apply Theorem 4.2 in \cite{kali}, we will show that for any
integer $m\geq m_{0}=\inf \left\{ k\geq 0:k/s>\alpha _{1}\right\} $ the
sequence 
\begin{equation}
\gamma _{k}=\frac{N^{-km}}{^{\kappa \left( N^{-k}\right) ^{s}}}=\left( \frac{%
\left( N^{-k}\right) ^{m/s}}{^{\kappa \left( N^{-k}\right) }}\right) ^{s}%
\text{ is strongly decreasing,}  \label{kf3}
\end{equation}%
that is $\gamma _{k}\leq c\gamma _{j}~\ $for all $k\geq j$ and some $c>0$,
and there is $k_{0}$ so that $\gamma _{k}\leq 2^{-1}\gamma _{j}$ for all
\thinspace $k\geq j+k_{0}$. Indeed, by Lemma \ref{ll1}, there is $C>0$ so
that with any $\sigma ^{\prime }>\alpha _{1},r\leq A\leq 1,$%
\begin{eqnarray}
\frac{r^{\sigma ^{\prime }}}{\kappa \left( r\right) } &=&r^{\sigma ^{\prime
}-\alpha _{1}}\frac{r^{\alpha _{1}}}{\kappa \left( r\right) }\leq Cr^{\sigma
^{\prime }-\alpha _{1}}\frac{A^{\alpha _{1}}}{\kappa \left( A\right) }
\label{kf4} \\
&=&C\left( \frac{r}{A}\right) ^{\sigma ^{\prime }-\alpha _{1}}\frac{%
A^{\sigma ^{\prime }}}{\kappa \left( A\right) }\leq \left( \frac{1}{2}%
\right) ^{1/s}\frac{A^{\sigma ^{\prime }}}{\kappa \left( A\right) }  \notag
\end{eqnarray}%
if $r/A$ is sufficiently small. The claim follows by Theorem 4.2 in \cite%
{kali}.
\end{proof}

We would like to mention some other application of the results in \cite{kali}%
.

\begin{remark}
\label{kre1}Let $j_{0}\geq 1,$ and 
\begin{equation*}
s_{0}=\inf \left\{ s>0:l\left( r\right) =o\left( r^{j_{0}/s}\right) \right\}
.
\end{equation*}%
Let $f\in H_{p}^{\mu ;s}\left( \mathbf{R}^{d}\right) ,s>s_{0}>0,p\in \left(
1,\infty \right) $. It can be shown by checking the assumptions of Theorem
3.5 in \cite{kali} that $D^{j}f\in L_{p}\left( \mathbf{R}^{d}\right) ,j\leq
j_{0}.$
\end{remark}

\subsubsection{Embedding into the space of continuous functions}

We start with

\begin{lemma}
\label{crl1}Let \textbf{D}$\left( \kappa ,l\right) $ and \textbf{B}$\left(
\kappa ,l\right) $ hold for $\pi \in \mathfrak{A}^{\sigma }$ with scaling
function $\kappa $ and scaling factor $l.$ Let $\delta \in (0,1],q\geq
1,\gamma \left( t\right) =\inf \left\{ r:l\left( r\right) \geq t\right\}
,t>0 $. Assume%
\begin{equation}
\int_{0}^{1}t^{\delta -1}\gamma \left( t\right) ^{-d+d/q}dt+\int_{1}^{\infty
}t^{\delta -1}\gamma \left( t\right) ^{-1-d+d/q}dt<\infty .  \label{1}
\end{equation}%
Let $p^{\left\vert z\right\vert }\left( t,\cdot \right) $ be the pdf of the
Levy process $Z_{t}^{\tilde{\pi}_{\left\vert z\right\vert }},t>0,$
associated to the Levy measure $\tilde{\pi}_{\left\vert z\right\vert }\left(
dy\right) =\kappa \left( \left\vert z\right\vert \right) \pi \left(
\left\vert z\right\vert dy\right) ,z\in \mathbf{R}^{d}$ and%
\begin{eqnarray*}
&&b^{\pi ;\delta }\left( y,z\right) \\
&=&\left\vert z\right\vert ^{-d}\int_{0}^{\infty }t^{\delta }\left\vert
p^{\left\vert z\right\vert }\left( t,\frac{y}{\left\vert z\right\vert }+\hat{%
z}\right) -p^{\left\vert z\right\vert }\left( t,\frac{y}{\left\vert
z\right\vert }\right) \right\vert \frac{dt}{t},y\in \mathbf{R}^{d},z\neq 0,
\end{eqnarray*}%
where $\hat{z}=z/\left\vert z\right\vert .$

Then there is $C$ so that%
\begin{equation*}
\left( \int \left\vert b^{\pi ;\delta }\left( y,z\right) \right\vert
^{q}dy\right) ^{1/q}\leq C\left\vert z\right\vert ^{-d+d/q},z\in \mathbf{R}%
^{d}\backslash \left\{ 0\right\} .
\end{equation*}
\end{lemma}

\begin{proof}
We split%
\begin{equation*}
b^{\pi ;\delta }=\left\vert z\right\vert ^{-d}\int_{0}^{1}...dt+\left\vert
z\right\vert ^{-d}\int_{1}^{\infty }...dt=b_{1}+b_{2}.
\end{equation*}%
By Minkowski inequality,%
\begin{eqnarray*}
&&\left( \int \left\vert b_{1}\left( y,z\right) \right\vert ^{q}dy\right)
^{1/q} \\
&\leq &C\left\vert z\right\vert ^{-d+d/q}\int_{0}^{1}t^{\delta }\left( \int
\left\vert p^{\left\vert z\right\vert }\left( t,y\right) \right\vert
^{q}dy\right) ^{1/q}\frac{dt}{t},z\neq 0.
\end{eqnarray*}%
By Lemma 5 b) in \cite{MPh} and Minkowski inequality,%
\begin{equation*}
\left\vert p^{\left\vert z\right\vert }\left( t,\cdot \right) \right\vert
_{L_{q}\left( \mathbf{R}^{d}\right) }\leq C\gamma \left( t\right) ^{-d+d/q}.
\end{equation*}%
Hence for any $z\neq 0,$ 
\begin{equation*}
\left( \int \left\vert b_{1}\left( y,z\right) \right\vert ^{q}dy\right)
^{1/q}\leq C\left\vert z\right\vert ^{-d+d/q}\int_{0}^{1}t^{\delta -1}\gamma
\left( t\right) ^{-d+d/q}dt=C\left\vert z\right\vert ^{-d+d/q}.
\end{equation*}%
Since for $y\in \mathbf{R}^{d},z\neq 0,$%
\begin{equation*}
|b_{2}\left( z,y\right) |\leq \left\vert z\right\vert ^{-d}\int_{1}^{\infty
}t^{\delta }\int_{0}^{1}\left\vert \nabla p^{\left\vert z\right\vert }\left(
t,\frac{y}{\left\vert z\right\vert }+s\hat{z}\right) \right\vert ds\frac{dt}{%
t},
\end{equation*}%
it follows by Minkowski inequality, 
\begin{equation*}
\left\vert b_{2}\left( \cdot ,z\right) \right\vert _{L_{q}\left( \mathbf{R}%
^{d}\right) }\leq C\left\vert z\right\vert ^{-d+d/q}\int_{1}^{\infty
}t^{\delta }\left( \int \left\vert \nabla p^{\left\vert z\right\vert }\left(
t,y\right) \right\vert ^{q}dy\right) ^{1/q}\frac{dt}{t}.
\end{equation*}%
By Lemma 5 b) in \cite{MPh} and Minkowski inequality, 
\begin{equation*}
\left\vert \nabla p^{\left\vert z\right\vert }\left( t,\cdot \right)
\right\vert _{L_{q}\left( \mathbf{R}^{d}\right) }\leq C\gamma \left(
t\right) ^{-(1+d-d/q)},t>0.
\end{equation*}%
Hence for $z\neq 0,$ 
\begin{equation*}
\left( \int \left\vert b_{2}\left( y,z\right) \right\vert ^{q}dy\right)
^{1/q}\leq C\left\vert z\right\vert ^{-d+d/q}\int_{1}^{\infty }t^{\delta
-1}\gamma \left( t\right) ^{-1-d+d/q}dt\leq C\left\vert z\right\vert
^{-d+d/q}.
\end{equation*}
\end{proof}

According to Lemma \ref{crl1}, for $\pi \in \mathfrak{A}^{\sigma }$
satisfying \textbf{D}$\left( \kappa ,l\right) $ and \textbf{B}$\left( \kappa
,l\right) $, under assumption (\ref{1}), the following function is well
defined%
\begin{eqnarray}
&&\bar{k}^{\pi ;\delta }\left( y,z\right)  \label{11} \\
&=&\kappa \left( \left\vert z\right\vert \right) ^{\delta }\left\vert
z\right\vert ^{-d}\int_{0}^{\infty }t^{\delta }\left[ p^{\left\vert
z\right\vert }\left( t,\frac{y}{\left\vert z\right\vert }+\hat{z}\right)
-p^{\left\vert z\right\vert }\left( t,\frac{y}{\left\vert z\right\vert }%
\right) \right] \frac{dt}{t},y\in \mathbf{R}^{d},z\neq 0,  \notag
\end{eqnarray}%
where $p^{\left\vert z\right\vert }$ is the pdf of the Levy process $Z_{t}^{%
\tilde{\pi}_{\left\vert z\right\vert }},t>0,$ associated to the Levy measure 
$\tilde{\pi}_{\left\vert z\right\vert }\left( dy\right) =\kappa \left(
\left\vert z\right\vert \right) \pi \left( \left\vert z\right\vert dy\right)
,z\neq 0.$

Now we derive a representation of an increment of $f\in \mathcal{S}\left( 
\mathbf{R}^{d}\right) $.

\begin{lemma}
\label{kl1}Let \textbf{D}$\left( \kappa ,l\right) $ and \textbf{B}$\left(
\kappa ,l\right) $ hold for $\pi \in \mathfrak{A}^{\sigma }$ with scaling
function $\kappa $ and scaling factor $l.$ Let 
\begin{equation*}
\psi ^{\pi ;\delta }=\left\{ 
\begin{array}{cc}
\psi ^{\pi } & \text{if }\delta =1, \\ 
-\left( -\func{Re}\psi ^{\pi }\right) ^{\delta } & \text{if }\delta \in
\left( 0,1\right) ,%
\end{array}%
\right.
\end{equation*}%
and%
\begin{equation*}
L^{\pi ;\delta }v=\mathcal{F}^{-1}\left[ \psi ^{\pi ;\delta }\mathcal{F}v%
\right] ,v\in \mathcal{S}\left( \mathbf{R}^{d}\right)
\end{equation*}%
(in particular, $L^{\pi ;1}=L^{\pi }$). Assume for some $q\geq 1,$%
\begin{equation}
\int_{0}^{1}t^{\delta -1}\gamma \left( t\right) ^{-d+d/q}dt+\int_{1}^{\infty
}t^{\delta -1}\gamma \left( t\right) ^{-1-d+d/q}dt<\infty .  \label{22}
\end{equation}

Then there is a constant $c>0$ so that for any $f\in \mathcal{S}\left( 
\mathbf{R}^{d}\right) $%
\begin{equation}
f\left( x+z\right) -f\left( x\right) =c\int L^{\pi ;\delta }f\left(
x-y\right) k^{\pi ;\delta }\left( y,z\right) dy,x\in \mathbf{R}^{d},z\neq 0,
\label{44}
\end{equation}%
where $k^{\pi ;1}=\bar{k}^{\pi ;1},k^{\pi ;\delta }=\bar{k}^{\pi
_{sym};\delta },\delta \in \left( 0,1\right) ,$ are functions defined by (%
\ref{11}), and $\pi _{sym}\left( dy\right) =\frac{1}{2}\left[ \pi \left(
dy\right) +\pi \left( -dy\right) \right] .$
\end{lemma}

\begin{proof}
Let $\varepsilon >0,f\in \mathcal{S}\left( \mathbf{R}^{d}\right) $ and $%
\tilde{F}_{\varepsilon }=[\varepsilon -\psi ^{\pi }]\hat{f}$ if $\delta =1$,
and 
\begin{equation*}
\tilde{F}_{\varepsilon }=\left[ \varepsilon -\func{Re}\psi ^{\pi }\right]
^{\delta }\hat{f}\text{ if }\delta \in \left( 0,1\right) .
\end{equation*}%
For $\delta \in \left( 0,1\right) ,$%
\begin{eqnarray*}
&&\mathcal{F}\left[ f\left( \cdot +z\right) -f\right] \left( \xi \right) \\
&=&\left( e^{i2\pi \xi \cdot z}-1\right) \left( \varepsilon -\func{Re}\psi
^{\pi }\left( \xi \right) \right) ^{-\delta }\left( \varepsilon -\func{Re}%
\psi ^{\pi }\left( \xi \right) \right) ^{\delta }\hat{f}\left( \xi \right) \\
&=&c\int_{0}^{\infty }t^{\delta }\left( e^{i2\pi \xi \cdot z}-1\right) \exp
\left\{ (\func{Re}\psi ^{\pi }\left( \xi \right) -\varepsilon )t\right\} 
\tilde{F}_{\varepsilon }\left( \xi \right) \frac{dt}{t},\xi \in \mathbf{R}%
^{d},z\neq 0,
\end{eqnarray*}%
and by Corollary 5 in \cite{MPh} for $\delta =1,$%
\begin{eqnarray*}
&&\mathcal{F}\left[ f\left( \cdot +z\right) -f\right] \left( \xi \right) \\
&=&\int_{0}^{\infty }\left( e^{i2\pi \xi \cdot z}-1\right) \exp \left\{
(\psi ^{\pi }\left( \xi \right) -\varepsilon )t\right\} \tilde{F}%
_{\varepsilon }\left( \xi \right) dt,\xi \in \mathbf{R}^{d},z\neq 0.
\end{eqnarray*}%
Changing the variable of integration and denoting $\hat{z}=z/\left\vert
z\right\vert $, we find for $\xi \in \mathbf{R}^{d},z\neq 0,$%
\begin{eqnarray}
&&\mathcal{F}\left[ f\left( \cdot +z\right) -f\right] \left( \xi \right)
\label{ff2} \\
&=&\kappa \left( \left\vert z\right\vert \right) ^{\delta }c\int_{0}^{\infty
}e^{-\varepsilon t}t^{\delta }\left( e^{i2\pi \left\vert z\right\vert \xi
\cdot \hat{z}}-1\right) \exp \left\{ \func{Re}\psi ^{\tilde{\pi}_{\left\vert
z\right\vert }}\left( \left\vert z\right\vert \xi \right) t\right\} \tilde{F}%
_{\varepsilon }\left( \xi \right) \frac{dt}{t}  \notag
\end{eqnarray}%
if $\delta \in \left( 0,1\right) $, and similar formula with obvious changes
holds for $\delta =1.$ Hence 
\begin{eqnarray}
&&f\left( x+z\right) -f\left( x\right)  \label{33} \\
&=&c\int H^{\varepsilon }(x-y)k_{\varepsilon }^{\pi ;\delta }\left(
y,z\right) dy,x\in \mathbf{R}^{d},z\neq 0,  \notag
\end{eqnarray}%
where for $y\in \mathbf{R}^{d},z\neq 0,$%
\begin{eqnarray*}
&&k_{\varepsilon }^{\pi ;\delta }\left( y,z\right) \\
&=&\kappa \left( \left\vert z\right\vert \right) ^{\delta }\left\vert
z\right\vert ^{-d}\int_{0}^{\infty }e^{-\varepsilon t}t^{\delta }\left[
p^{\left\vert z\right\vert }\left( t,\frac{y}{\left\vert z\right\vert }+\hat{%
z}\right) -p^{\left\vert z\right\vert }\left( t,\frac{y}{\left\vert
z\right\vert }\right) \right] \frac{dt}{t},
\end{eqnarray*}%
and $H^{\varepsilon }=\mathcal{F}^{-1}\tilde{F}_{\varepsilon }.$ Since $%
\func{Re}\psi ^{\pi }=\psi ^{\pi _{sym}}$, it follows for $\delta \in \left(
0,1\right) ,$%
\begin{equation*}
\left( \varepsilon -\func{Re}\psi ^{\pi }\right) ^{\delta
}=c\int_{0}^{\infty }t^{-\delta }\left[ \exp \left\{ \psi ^{\pi
_{sym}}\left( \xi \right) t-\varepsilon t\right\} -1\right] \frac{dt}{t},\xi
\in \mathbf{R}^{d}.
\end{equation*}%
Hence for $f\in \mathcal{S}\left( \mathbf{R}^{d}\right) $ and $\delta \in
\left( 0,1\right) ,$%
\begin{equation}
H^{\varepsilon }\left( x\right) =c\int_{0}^{\infty }t^{-\delta }\mathbf{E}%
\left[ e^{-\varepsilon t}f\left( x+Z_{t}^{\pi _{sym}}\right) -f\left(
x\right) \right] \frac{dt}{t},x\in \mathbf{R}^{d}.  \label{331}
\end{equation}%
For $f\in \mathcal{S}\left( \mathbf{R}^{d}\right) $ and $\delta =1,$
obviously,%
\begin{equation}
H^{\varepsilon }\left( x\right) =\varepsilon f\left( x\right) -L^{\pi
}f\left( x\right) ,x\in \mathbf{R}^{d}.  \label{332}
\end{equation}

The statement follows by passing to the limit in (\ref{33}), (\ref{331}) and
(\ref{332}) as $\varepsilon \rightarrow 0$ and using Lemma \ref{crl1} (H\"{o}%
lder inequality in (\ref{33}) if $q>1$ as well).
\end{proof}

The following obvious consequence holds (take $\delta =1,q=1$ in Lemma \ref%
{kl1}).

\begin{corollary}
\label{ccc1}Let \textbf{D}$\left( \kappa ,l\right) $ and \textbf{B}$\left(
\kappa ,l\right) $ hold for $\pi \in \mathfrak{A}^{\sigma }$ with scaling
function $\kappa $ and scaling factor $l.$ Assume%
\begin{equation*}
\int_{1}^{\infty }\gamma \left( t\right) ^{-1}dt<\infty .
\end{equation*}%
Then there is $c>0$ so that for $f\in \mathcal{S}\left( \mathbf{R}%
^{d}\right) $ 
\begin{equation*}
f\left( x+z\right) -f\left( x\right) =c\int L^{\pi }f\left( x-y\right)
k^{\pi ;1}\left( y,z\right) dy,x\in \mathbf{R}^{d},z\neq 0,
\end{equation*}%
and there is $C>0$ so that%
\begin{equation}
\left\vert f\left( x+z\right) -f\left( x\right) \right\vert \leq C\left\vert
L^{\pi }f\right\vert _{\infty }\kappa \left( \left\vert z\right\vert \right)
,x\in \mathbf{R}^{d},z\neq 0,f\in \mathcal{S}\left( \mathbf{R}^{d}\right)
\label{111}
\end{equation}
\end{corollary}

Now, we prove an embedding statement.

\begin{proposition}
\label{pro4}Let \textbf{D}$\left( \kappa ,l\right) $ and \textbf{B}$\left(
\kappa ,l\right) $ hold for $\pi \in \mathfrak{A}^{\sigma }$ with scaling
function $\kappa $ and scaling factor $l.$ Let $\delta \in (0,1],p\in
(1,\infty )$. Assume 
\begin{equation*}
\int_{0}^{1}t^{\delta -1}\gamma \left( t\right) ^{-d/p}dt+\int_{1}^{\infty
}t^{\delta -1}\gamma \left( t\right) ^{-1-d/p}dt<\infty .
\end{equation*}%
Then there is $C_{1}$ so that%
\begin{equation}
\sup_{x}\left\vert f\left( x+z\right) -f\left( x\right) \right\vert \leq
C_{1}\kappa \left( \left\vert z\right\vert \right) \left\vert z\right\vert
^{-d/p}\left\vert L^{\pi ;\delta }f\right\vert _{L_{p}\left( \mathbf{R}%
^{d}\right) },f\in \mathcal{S}\left( \mathbf{R}^{d}\right) .  \label{110}
\end{equation}%
Moreover, there is $C_{1}$ so that for any $f\in \mathcal{S}\left( \mathbf{R}%
^{d}\right) ,$%
\begin{equation*}
\sup_{x}\left\vert f\left( x\right) \right\vert \leq \left\vert f\right\vert
_{L_{p}\left( \mathbf{R}^{d}\right) }+C_{1}\left\vert L^{\pi ;\delta
}f\right\vert _{L_{p}\left( \mathbf{R}^{d}\right) }\int_{\left\vert
z\right\vert \leq 1}\kappa \left( \left\vert z\right\vert \right) \left\vert
z\right\vert ^{-d/p}dz.
\end{equation*}
\end{proposition}

\begin{proof}
According to Lemma \ref{kl1}, the representation (\ref{44}) holds. Applying H%
\"{o}lder inequality and Lemma \ref{crl1} with $1/p+1/q=1,$%
\begin{eqnarray*}
\left\vert f\left( x+z\right) -f\left( x\right) \right\vert  &\leq
&C\left\vert L^{\pi ;\delta }f\right\vert _{L_{p}\left( \mathbf{R}%
^{d}\right) }\left( \int \left\vert k^{\pi ;\delta }\left( y,z\right)
\right\vert ^{q}dy\right) ^{1/q} \\
&\leq &C_{1}\kappa \left( \left\vert z\right\vert \right) \left\vert
z\right\vert ^{-d/p}\left\vert L^{\pi ;\delta }f\right\vert _{L_{p}\left( 
\mathbf{R}^{d}\right) },x\in \mathbf{R}^{d},z\neq 0.
\end{eqnarray*}

Let $x\in \mathbf{R}^{d}$. Then for any $z\in \mathbf{R}^{d},$%
\begin{equation*}
\left\vert f\left( x\right) \right\vert \leq \left\vert f\left( x+z\right)
-f\left( x\right) \right\vert +\left\vert f\left( x+z\right) \right\vert .
\end{equation*}%
Integrating both sides over the unit ball $B_{1}=\left\{ z\in \mathbf{R}%
^{d}:\left\vert z\right\vert \leq 1\right\} ,$%
\begin{equation*}
\left\vert f\left( x\right) \right\vert \leq \frac{1}{\left\vert
B_{1}\right\vert }\int_{B_{1}}\left\vert f\left( x+z\right) -f\left(
x\right) \right\vert dz+\frac{1}{\left\vert B_{1}\right\vert }%
\int_{B_{1}}\left\vert f\left( x+z\right) \right\vert dz,
\end{equation*}%
and the last inequality follows by H\"{o}lder inequality and (\ref{110}).
\end{proof}

\section{Proof of Theorem \protect\ref{t1}}

First we prove some auxiliary results.

\subsection{Auxiliary results}

We start with

\subsubsection{Scaling function properties}

\begin{lemma}
\label{ll2}Let $\kappa $ be a scaling function with a scaling factor $l$. Let%
\begin{eqnarray*}
a\left( r\right) &=&\inf \left\{ t:\kappa \left( t\right) \geq r\right\}
,r>0,a^{-1}\left( s\right) =\inf \left\{ t:a\left( t\right) \geq s\right\}
,s>0, \\
\gamma \left( t\right) &=&\inf \left\{ r:l\left( r\right) \geq t\right\}
,t>0.
\end{eqnarray*}%
Then

\begin{eqnarray*}
a^{-1}\left( r\right) &=&\sup_{s\leq r}\kappa \left( s\right) \leq l\left(
1\right) \kappa \left( r\right) ,r>0, \\
a^{-1}\left( r\varepsilon \right) &\leq &l\left( \varepsilon \right)
a^{-1}\left( r\right) ,\varepsilon ,r>0.
\end{eqnarray*}%
and%
\begin{equation*}
a\left( \varepsilon r\right) \geq a\left( r\right) \gamma \left( \varepsilon
\right) ,r,\varepsilon >0.
\end{equation*}%
In particular, $\gamma \left( \varepsilon \right) \leq a\left( \varepsilon
\right) a\left( 1\right) ^{-1}$, and%
\begin{equation}
\frac{a\left( r\right) }{a\left( r^{\prime }\right) }\leq \gamma \left( 
\frac{r^{\prime }}{r}\right) ^{-1},r^{\prime },r>0.  \label{pf1}
\end{equation}
\end{lemma}

\begin{proof}
Let $B\left( t\right) =\max_{s\leq t}\kappa \left( t\right) ,t\geq 0.$ Since 
$B$ is continuous, and 
\begin{equation*}
a\left( r\right) =\inf \left\{ t:\kappa \left( t\right) \geq r\right\} =\inf
\left\{ t:B\left( t\right) \geq r\right\} ,r>0,
\end{equation*}%
we have $B\left( t\right) =a^{-1}\left( t\right) ,t\geq 0.$ Hence for any $%
r>0$%
\begin{equation*}
a^{-1}\left( r\right) =\sup_{\varepsilon \leq 1}\kappa \left( \varepsilon
r\right) \leq \sup_{\varepsilon \leq 1}l\left( \varepsilon \right) \kappa
\left( r\right) =l\left( 1\right) \kappa \left( r\right) .
\end{equation*}%
For any $r,\varepsilon >0,$%
\begin{equation*}
a^{-1}\left( \varepsilon r\right) =B\left( \varepsilon r\right)
=\max_{\varepsilon ^{\prime }\leq \varepsilon }\kappa \left( \varepsilon
^{\prime }r\right) \leq \max_{\varepsilon ^{\prime }\leq \varepsilon
}l\left( \varepsilon ^{\prime }\right) \kappa \left( r\right) =l\left(
\varepsilon \right) \kappa \left( r\right) \leq l\left( \varepsilon \right)
a^{-1}\left( r\right) .
\end{equation*}%
Since for any $\varepsilon ,r>0,$ 
\begin{equation*}
\max_{\varepsilon ^{\prime }\leq \gamma (\varepsilon )}\kappa \left(
\varepsilon ^{\prime }a\left( r\right) \right) \leq \sup_{\varepsilon
^{\prime }\leq \gamma (\varepsilon )}l\left( \varepsilon ^{\prime }\right)
\kappa \left( a\left( r\right) \right) =l\left( \gamma \left( \varepsilon
\right) \right) r=\varepsilon r,
\end{equation*}%
we have $a\left( \varepsilon r\right) \geq a\left( r\right) \gamma \left(
\varepsilon \right) ,\varepsilon ,r>0$.
\end{proof}

\subsubsection{Probability density estimates}

We will need some probability density estimates. Given $\mu \in \mathfrak{A}%
^{\sigma },t>0,$ we denote $p^{\mu }\left( t,x\right) ,x\in \mathbf{R}^{d},$
the probability density function of $Z_{t}^{\mu }$ provided such a density
exists.

\begin{lemma}
\label{al0}Let \textbf{D}$\left( \kappa ,l\right) $ and \textbf{B}$\left(
\kappa ,l\right) $ hold for $\mu \in \mathfrak{A}^{\sigma }$ with scaling
function $\kappa $ and scaling factor $l$. Let $R>0$ and $Z_{t}^{R}$ be the
Levy process associated to $\tilde{\mu}_{R}$.

a) For each $t>0$, we have $Z_{t}^{R}=\eta _{t}+\eta _{t}^{\prime }$ (in
distribution)$,$ $\eta _{t}$ and $\tilde{\eta}_{t}$ are independent with%
\begin{equation}
\mathbf{E}e^{i2\pi \xi \cdot \eta _{t}}=\exp \{\psi ^{\mu ^{0}}\left( \xi
\gamma \left( t\right) \right) \},\xi \in \mathbf{R}^{d},  \label{a}
\end{equation}%
and $\mu _{\gamma \left( t\right) ^{-1}}^{0}$ $\leq t\tilde{\mu}_{R}$, where 
$\mu ^{0}=\mu ^{0;\mu },$ $\gamma \left( t\right) =l^{-1}\left( t\right)
=\inf \left( s:l\left( s\right) \geq t\right) .$

b)For every $t>0,\,$the process $Z_{t}^{R}$ (equivalently $R^{-1}Z_{\kappa
\left( R\right) t}^{\mu }$) has a bounded continuous probability density
function 
\begin{equation*}
p^{R}\left( t,x\right) =\gamma \left( t\right) ^{-d}\int p_{0}\left( \frac{%
x-y}{\gamma \left( t\right) }\right) P_{t,R}\left( dy\right) ,x\in \mathbf{R}%
^{d},
\end{equation*}%
where $P_{t,R}\left( dy\right) $ is the distribution measure of $\eta
_{t}^{\prime }$ on $\mathbf{R}^{d}$ and $p_{0}=\mathcal{F}^{-1}\left[ \exp
\left\{ \psi ^{\mu ^{0}}\right\} \right] $. For each $t>0$, $p^{R}\left(
t,x\right) $ has $4$ bounded continuous derivatives such that for any
multiindex $\left\vert k\right\vert \leq 4,$ 
\begin{eqnarray*}
\int \left\vert \partial ^{k}p^{R}\left( t,x\right) \right\vert dx &\leq
&\gamma (t)^{-\left\vert k\right\vert }\int \left\vert \partial
^{k}p_{0}\left( x\right) \right\vert dx, \\
\sup_{x\in \mathbf{R}^{d}}\left\vert \partial ^{k}p^{R}\left( t,x\right)
\right\vert &\leq &\gamma \left( t\right) ^{-d-\left\vert k\right\vert
}\sup_{x}\left\vert \partial ^{k}p_{0}\left( x\right) \right\vert .
\end{eqnarray*}

c) Moreover there is $C=C\left( N\right) $ ($N$ is a constant in \textbf{B}$%
\left( \kappa ,l\right) $) so that for $\left\vert k\right\vert \leq 4,$ 
\begin{eqnarray}
\int \left\vert x\right\vert ^{\alpha _{2}}\left\vert D^{k}p^{R}\left(
t,x\right) \right\vert dx &\leq &C\gamma \left( t\right) ^{-\left\vert
k\right\vert }[1+t+\gamma \left( t\right) ^{\alpha _{2}}],t>0,  \label{d11}
\\
\int \left\vert y\right\vert ^{\alpha _{2}}P_{t,R}\left( dy\right) &\leq
&C\left( 1+t\right) ,t>0.
\end{eqnarray}
\end{lemma}

\begin{proof}
a) and b) parts are proved in Lemma 5 of \cite{MPh}. We prove only (\ref{d11}%
).

Fix $s>0$. Let $\rho _{R}=\tilde{\mu}_{R}-s^{-1}\mu _{\gamma \left( s\right)
^{-1}}^{0}$ and $Y_{t}$ be the Levy process corresponding to $\rho _{R}$,
i.e. 
\begin{equation*}
\mathbf{E}e^{i2\pi Y_{t}\cdot \xi }=\exp \left\{ \psi \left( \xi \right)
t\right\} ,t\geq 0,\xi \in \mathbf{R}^{d},
\end{equation*}%
with%
\begin{equation*}
\psi \left( \xi \right) =\int \left[ e^{i2\pi \xi \cdot y}-1-i2\pi \chi
_{\sigma }\left( y\right) y\cdot \xi \right] d\rho _{R},\xi \in \mathbf{R}%
^{d}.
\end{equation*}%
Then 
\begin{eqnarray*}
&&\int_{\left\vert y\right\vert \leq 1}\left\vert y\right\vert ^{\alpha
_{1}}d\rho _{R}+\int_{\left\vert y\right\vert >1}\left\vert y\right\vert
^{\alpha _{2}}d\rho _{R} \\
&\leq &\int_{\left\vert y\right\vert \leq 1}\left\vert y\right\vert ^{\alpha
_{1}}d\tilde{\mu}_{R}+\int_{\left\vert y\right\vert >1}\left\vert
y\right\vert ^{\alpha _{2}}d\tilde{\mu}_{R}\leq N,R>0,
\end{eqnarray*}%
where $\alpha _{1},\alpha _{2}$ are exponents in assumption \textbf{B}$%
\left( \kappa ,l\right) .$ By Lemma \ref{al00} in Appendix there is $%
C=C(N_{0})$ such that%
\begin{equation*}
\mathbf{E}\left[ \left\vert Y_{t}\right\vert ^{\alpha _{2}}\right] \leq
C(1+t),t\geq 0;
\end{equation*}%
in particular, for $t=s,$ 
\begin{equation*}
\mathbf{E}\left( \left\vert \eta _{s}^{\prime }\right\vert ^{\alpha
_{2}}\right) =\mathbf{E}\left( \left\vert Y_{s}\right\vert ^{\alpha
_{2}}\right) \leq C(1+s).
\end{equation*}%
Since $s$ is arbitrary, there is $C=C\left( N_{0}\right) $ such that%
\begin{equation}
\int \left\vert y\right\vert ^{\alpha _{2}}P_{s,R}\left( dy\right) =\mathbf{E%
}\left( \left\vert \eta _{s}^{\prime }\right\vert ^{\alpha _{2}}\right) \leq
C(1+s),s\geq 0.  \label{d502}
\end{equation}

Let $\left\vert k\right\vert \leq 4$. Then according to (\ref{d502}) and
Lemma 3 in \cite{MPh}, 
\begin{eqnarray*}
\int \left\vert x\right\vert ^{\alpha _{2}}\left\vert D^{k}p^{R}\left(
t,x\right) \right\vert dx &\leq &\gamma \left( t\right) ^{-d-\left\vert
k\right\vert }\int \int \left\vert x\right\vert ^{\alpha _{2}}\left\vert
(D^{k}p_{0})\left( \frac{x-y}{\gamma \left( t\right) }\right) \right\vert
dxP_{t,R}\left( dy\right) \\
&\leq &\gamma \left( t\right) ^{-|k|}\int \int \left\vert y+\gamma \left(
t\right) z\right\vert ^{\alpha _{2}}\left\vert D^{k}p_{0}\left( z\right)
\right\vert dzP_{t,R}\left( dy\right) \\
&\leq &C\gamma \left( t\right) ^{-\left\vert k\right\vert }[1+t+\gamma
\left( t\right) ^{\alpha _{2}}].
\end{eqnarray*}
\end{proof}

We write $\pi \in \mathfrak{A}_{sign}=\mathfrak{A}^{\sigma }-\mathfrak{A}%
^{\sigma }$ if $\pi =\nu -\eta $ with $\nu ,\eta \in \mathfrak{A}^{\sigma }$%
, and $L^{\pi }=L^{\nu }+L^{\eta }$. Given $\pi \in \mathfrak{A}%
_{sign}^{\sigma }$, we denote $\left\vert \pi \right\vert $ its variation
measure. Obviously, $\left\vert \pi \right\vert \in \mathfrak{A}^{\sigma }.$

\begin{corollary}
\label{cc1}Let assumptions of Lemma \ref{al0} hold for $\mu \in \mathfrak{A}%
^{\sigma }$. Let $\pi \in \mathfrak{A}_{sign}^{\sigma }$ and 
\begin{equation*}
\int_{\left\vert y\right\vert \leq 1}\left\vert y\right\vert ^{\alpha _{1}}d%
\widetilde{\left\vert \pi \right\vert }_{R}+\int_{\left\vert y\right\vert
>1}\left\vert y\right\vert ^{\alpha _{2}}d\widetilde{\left\vert \pi
\right\vert }_{R}\leq M,R>0
\end{equation*}%
($\alpha _{1},\alpha _{2}$ are exponents in \textbf{B}$\left( \kappa
,l\right) $). Then there is $C=C\left( N\right) $ ($N$ is a constant in 
\textbf{B}$\left( \kappa ,l\right) $) so that for $\left\vert k\right\vert
\leq 2,$ 
\begin{eqnarray*}
\int \left( 1+\left\vert x\right\vert ^{\alpha _{2}}\right) \left\vert
D^{k}L^{\tilde{\pi}_{R}}p^{R}\left( 1,x\right) \right\vert dx &\leq &CM, \\
\int (1+\left\vert x\right\vert ^{\alpha _{2}})\left\vert L^{\tilde{\pi}%
_{R}}L^{\tilde{\mu}_{R}\ast }p^{R}\left( 1,x\right) \right\vert dx &\leq &CM.
\end{eqnarray*}
\end{corollary}

\begin{proof}
Let $\sigma \in \left( 0,1\right) $. Then by Lemma \ref{al0}, for $%
\left\vert k\right\vert \leq 3,$%
\begin{eqnarray}
&&\int \left( 1+\left\vert x\right\vert ^{\alpha _{2}}\right) \left\vert
D^{k}L^{\tilde{\pi}_{R}}p^{R}\left( 1,x\right) \right\vert dx  \label{cc2} \\
&\leq &\int \int_{\left\vert y\right\vert \leq 1}\int_{0}^{1}\left[
1+\left\vert x+sy\right\vert ^{\alpha _{2}}+\left\vert y\right\vert ^{\alpha
_{2}}\right] \left\vert y\right\vert \left\vert D^{k+1}p^{R}\left(
1,x+sy\right) \right\vert dsd\widetilde{\left\vert \pi \right\vert }_{R}dx 
\notag \\
&&+\int \int_{\left\vert y\right\vert >1}\left[ 1+\left\vert x+y\right\vert
^{\alpha _{2}}+\left\vert y\right\vert ^{\alpha _{2}}\right] \left\vert
D^{k}p^{R}\left( 1,x+y\right) \right\vert d\widetilde{\left\vert \pi
\right\vert }_{R}dx  \notag \\
&&+M\int (1+\left\vert x\right\vert ^{\alpha _{2}})\left\vert
D^{k}p^{R}\left( 1,x\right) \right\vert dx  \notag \\
&\leq &CM.
\end{eqnarray}%
The same way, similarly to (\ref{cc2}) and using Lemma \ref{al0} (see
Corollary 2 in \cite{MPh} as well),%
\begin{eqnarray*}
&&\int \left( 1+\left\vert x\right\vert ^{\alpha _{2}}\right) \left\vert L^{%
\tilde{\pi}_{R}}L^{\tilde{\mu}_{R}}p^{R}\left( 1,x\right) \right\vert dx \\
&\leq &\int \int_{\left\vert y\right\vert \leq 1}\left( 1+\left\vert
x+sy\right\vert ^{\alpha _{2}}+\left\vert y\right\vert ^{\alpha _{2}}\right)
\left\vert y\right\vert \left\vert L^{\tilde{\mu}_{R}}\nabla p^{R}\left(
1,x+sy\right) \right\vert ds\widetilde{\left\vert \pi \right\vert }%
_{R}\left( dy\right) dx \\
&&+\int \int_{\left\vert y\right\vert >1}\left( 1+\left\vert x+sy\right\vert
^{\alpha _{2}}+\left\vert y\right\vert ^{\alpha _{2}}\right) \left\vert L^{%
\tilde{\mu}_{R}}p^{R}\left( 1,x+y\right) \right\vert \widetilde{\left\vert
\pi \right\vert }_{R}\left( dy\right) dx \\
&&+M\int \left( 1+\left\vert x\right\vert ^{\alpha _{2}}\right) \left\vert
L^{\tilde{\mu}_{R}}p^{R}\left( 1,x\right) \right\vert dx \\
&\leq &CM.
\end{eqnarray*}%
Similarly we handle the cases $\alpha \in \left( 1,2\right) $ and $\alpha
=1. $
\end{proof}

\begin{lemma}
\label{al1}Let \textbf{D}$\left( \kappa ,l\right) $ and \textbf{B}$\left(
\kappa ,l\right) $ hold for $\mu \in \mathfrak{A}^{\sigma }$ with scaling
function $\kappa $ and scaling factor $l$. Let $\pi ,\pi ^{\prime }\in 
\mathfrak{A}_{sign}^{\sigma }.$ Then%
\begin{eqnarray*}
p^{\mu }\left( t,x\right) &=&a\left( t\right) ^{-d}p^{\tilde{\mu}_{a\left(
t\right) }}\left( 1,xa\left( t\right) ^{-1}\right) ,x\in \mathbf{R}^{d},t>0,
\\
L^{\pi }p^{\mu }\left( t,x\right) &=&\frac{1}{t}a\left( t\right) ^{-d}(L^{%
\tilde{\pi}_{a\left( t\right) }}p^{\tilde{\mu}_{a\left( t\right) }})\left(
1,xa\left( t\right) ^{-1}\right) ,x\in \mathbf{R}^{d},t>0, \\
L^{\pi ^{\prime }}L^{\pi }p^{\mu }\left( t,x\right) &=&\frac{1}{t^{2}}%
a\left( t\right) ^{-d}(L^{\widetilde{\pi _{a\left( t\right) }^{\prime }}}L^{%
\tilde{\pi}_{a\left( t\right) }}p^{\tilde{\mu}_{a\left( t\right) }})\left(
1,xa\left( t\right) ^{-1}\right) ,x\in \mathbf{R}^{d},t>0,
\end{eqnarray*}%
where $a\left( t\right) =\inf \left\{ r\geq 0:\kappa \left( r\right) \geq
t\right\} ,t>0.$
\end{lemma}

\begin{proof}
Indeed, by Lemma 5 in \cite{MPh}, for each $t>0$, the density $p^{\tilde{\mu}%
_{a\left( t\right) }}\left( r,x\right) ,r>0,x\in \mathbf{R}^{d},$ is 4 times
continuously differentiable in $x$ and integrable. Using Fourier transform,
we see that%
\begin{equation*}
\exp \left\{ \psi ^{\mu }\left( \xi \right) t\right\} =\exp \left\{ \psi ^{%
\tilde{\mu}_{a\left( t\right) }}\left( a\left( t\right) \xi \right) \right\}
,t>0,\xi \in \mathbf{R}^{d},
\end{equation*}%
and%
\begin{eqnarray*}
&&\psi ^{\pi }\left( \xi \right) \exp \left\{ \psi ^{\mu }\left( \xi \right)
t\right\} \\
&=&\frac{1}{t}\psi ^{\tilde{\pi}_{a\left( t\right) }}\left( a\left( t\right)
\xi \right) \exp \left\{ \psi ^{\tilde{\mu}_{a\left( t\right) }}\left(
a\left( t\right) \xi \right) \right\} ,t>0,\xi \in \mathbf{R}^{d}.
\end{eqnarray*}%
Similarly, the third equality can be derived. The claim follows.
\end{proof}

\begin{lemma}
\label{al2}Let \textbf{D}$\left( \kappa ,l\right) $ and \textbf{B}$\left(
\kappa ,l\right) $ hold for $\mu \in \mathfrak{A}^{\sigma }$ with scaling
function $\kappa $ and scaling factor $l$. Let $\pi \in \mathfrak{A}%
_{sign}^{\sigma }$. Assume 
\begin{equation*}
\int_{\left\vert y\right\vert \leq 1}\left\vert y\right\vert ^{\alpha _{1}}d%
\widetilde{\left\vert \pi \right\vert }_{R}+\int_{\left\vert y\right\vert
>1}\left\vert y\right\vert ^{\alpha _{2}}d\widetilde{\left\vert \pi
\right\vert }_{R}\leq M,R>0
\end{equation*}%
($\alpha _{1},\alpha _{2}$ are exponents in \textbf{B}$\left( \kappa
,l\right) $). Then there exists $C=C\left( \kappa ,l\right) >0$ such that
for $\left\vert k\right\vert \leq 2,$\ 
\begin{eqnarray*}
\int_{\left\vert z\right\vert >c}\left\vert L^{\pi }D^{k}p^{\mu }\left(
t,z\right) \right\vert dz &\leq &CMt^{-1}a\left( t\right) ^{\alpha
_{2}-\left\vert k\right\vert }c^{-\alpha _{2}}, \\
\int \left\vert L^{\pi }D^{k}p^{\mu }\left( t,z\right) \right\vert dz &\leq
&CMt^{-1}a\left( t\right) ^{-\left\vert k\right\vert },
\end{eqnarray*}%
with $a\left( t\right) =\inf \left\{ r\geq 0:\kappa \left( r\right) \geq
t\right\} ,t>0.$ Recall $\alpha _{1},\alpha _{2}\in (0,1]$ if $\sigma \in
\left( 0,1\right) ;\alpha _{1},\alpha _{2}\in (1,2]$ if $\sigma \in \left(
1,2\right) $ and $\alpha _{2}\in \left( \lbrack 0,1\right) ,\alpha _{1}\in
(1,2]$ if $\sigma =1.$
\end{lemma}

\begin{proof}
Indeed, by Lemma \ref{al1}, Chebyshev inequality, and Corollary \ref{cc1},
for $\left\vert k\right\vert \leq 2,$%
\begin{eqnarray*}
&&\int_{\left\vert z\right\vert >c}\left\vert L^{\pi }D^{k}p^{\mu }\left(
t,\cdot \right) \left( z\right) \right\vert dz \\
&=&\frac{1}{t}a\left( t\right) ^{-d-k}\int_{\left\vert x\right\vert
>c}\left\vert L^{\tilde{\pi}_{a\left( t\right) }}D^{k}p^{\tilde{\mu}%
_{a\left( t\right) }}\left( 1,\frac{x}{a\left( t\right) }\right) \right\vert
dx \\
&=&\frac{1}{t}\int_{\left\vert x\right\vert >ca\left( t\right)
^{-1}}\left\vert L^{\tilde{\pi}_{a\left( t\right) }}D^{k}p^{\tilde{\mu}%
_{a\left( t\right) }}\left( 1,x\right) \right\vert dx \\
&\leq &\frac{a\left( t\right) ^{\alpha _{2}-k}c^{-\alpha _{2}}}{t}\int
\left\vert x\right\vert ^{\alpha _{2}}\left\vert L^{\tilde{\pi}_{a\left(
t\right) }}D^{k}p^{\tilde{\mu}_{a\left( t\right) }}\left( 1,x\right)
\right\vert dx\leq CM\frac{a\left( t\right) ^{\alpha _{2}-k}c^{-\alpha _{2}}%
}{t}.
\end{eqnarray*}

and%
\begin{eqnarray*}
&&\int \left\vert D^{k}L^{\pi }p^{\mu }\left( t,\cdot \right) \left(
z\right) \right\vert dz \\
&=&\frac{1}{t}a\left( t\right) ^{-d-k}\int \left\vert L^{\tilde{\pi}%
_{a\left( t\right) }}D^{k}p^{\tilde{\mu}_{a\left( t\right) }}\left( 1,\frac{x%
}{a\left( t\right) }\right) \right\vert dx \\
&=&\frac{1}{t}a\left( t\right) ^{-k}\int \left\vert L^{\tilde{\pi}_{a\left(
t\right) }}D^{k}p^{\tilde{\mu}_{a\left( t\right) }}\left( 1,x\right)
\right\vert dx\leq CM\frac{1}{t}a\left( t\right) ^{-k},t>0.
\end{eqnarray*}
\end{proof}

\begin{lemma}
\label{mvt}Let \textbf{D}$\left( \kappa ,l\right) $ and \textbf{B}$\left(
\kappa ,l\right) $ hold for $\mu \in \mathfrak{A}^{\sigma }$ with scaling
function $\kappa $ and scaling factor $l$. Let $\pi \in \mathfrak{A}%
_{sign}^{\sigma }$. Assume 
\begin{equation*}
\int_{\left\vert y\right\vert \leq 1}\left\vert y\right\vert ^{\alpha _{1}}d%
\widetilde{\left\vert \pi \right\vert }_{R}+\int_{\left\vert y\right\vert
>1}\left\vert y\right\vert ^{\alpha _{2}}d\widetilde{\left\vert \pi
\right\vert }_{R}\leq M,R>0
\end{equation*}%
($\alpha _{1},\alpha _{2}$ are exponents in \textbf{B}$\left( \kappa
,l\right) $). Then

a) There exists $C=C\left( \kappa ,l\right) >0$ such that%
\begin{equation*}
\int_{\mathbf{R}^{d}}\left\vert L^{\pi }p^{\mu }\left( t,x-y\right) -L^{\pi
}p^{\mu }\left( t,x\right) \right\vert dx\leq CM\frac{\left\vert
y\right\vert }{ta\left( t\right) },t>0,\bar{y},y\in \mathbf{R}^{d},
\end{equation*}%
where $a\left( t\right) =\inf \left\{ r:\kappa \left( r\right) \geq
t\right\} ,t>0.$

b) There is a constant $C=C\left( \kappa ,l,N\right) $ such that

\begin{align}
& \int_{2a}^{\infty }\int \left\vert L^{\pi }p^{\mu }\left( t-s,x\right)
-L^{\pi }p^{\mu }\left( t,x\right) \right\vert dxdt  \label{eq:MVTtime} \\
& \leq CM,\left\vert s\right\vert \leq a<\infty .  \notag
\end{align}
\end{lemma}

\begin{proof}
By Lemma \ref{al1} and Corollary \ref{cc1},%
\begin{eqnarray*}
&&\int_{\mathbf{R}^{d}}\left\vert L^{\pi }p^{\mu }\left( t,x-y\right)
-L^{\pi }p^{\mu }\left( t,x\right) \right\vert dx \\
&=&\frac{1}{t}\int \left\vert L^{\tilde{\pi}_{a\left( t\right) }}p^{\tilde{%
\mu}_{a\left( t\right) }}\left( 1,x-\frac{y}{a\left( t\right) }\right) -L^{%
\tilde{\pi}_{a\left( t\right) }}p^{\tilde{\mu}_{a\left( t\right) }}\left(
1,x\right) \right\vert dx \\
&\leq &\frac{1}{t}\int_{0}^{1}\int \left\vert \nabla L^{\tilde{\pi}_{a\left(
t\right) }}p^{\tilde{\mu}_{a\left( t\right) }}\left( 1,x-s\frac{y}{a\left(
t\right) }\right) \right\vert \frac{\left\vert y\right\vert }{a\left(
t\right) }dxds \\
&\leq &C\frac{\left\vert y\right\vert }{ta\left( t\right) }\int \left\vert
L^{\tilde{\pi}_{a\left( t\right) }}\nabla p^{\tilde{\mu}_{a(t)}}\left(
1,x\right) \right\vert dx\leq CM\frac{\left\vert y\right\vert }{ta\left(
t\right) }.
\end{eqnarray*}%
Similarly, we derive the estimate (\ref{eq:MVTtime}). By Lemma \ref{al1} and
Corollary \ref{cc1},%
\begin{eqnarray*}
&&\int_{2a}^{\infty }\int \left\vert L^{\pi }p^{\mu }\left( t-s,x\right)
-L^{\pi }p^{\mu }\left( t,x\right) \right\vert dxdt \\
&\leq &\left\vert s\right\vert \int_{2a}^{\infty }\int
\int_{0}^{1}\left\vert L^{\pi }\partial _{t}p^{\mu }\left( t-\theta
s,x\right) \right\vert d\theta dxdt \\
&=&\left\vert s\right\vert \int_{2a}^{\infty }\int \int_{0}^{1}\left\vert
L^{\pi }L^{\mu ^{\ast }}p^{\mu }\left( t-\theta s,x\right) \right\vert
d\theta dxdt \\
&=&\left\vert s\right\vert \int_{2a}^{\infty }\int \int_{0}^{1}\left(
t-\theta s\right) ^{-2}\left\vert (L^{\tilde{\pi}_{a\left( t-\theta s\right)
}}L^{\tilde{\mu}_{a\left( t-\theta s\right) }^{\ast }}p^{\tilde{\mu}%
_{a\left( t-\theta s\right) }})\left( 1,x\right) \right\vert d\theta dxdt \\
&\leq &CM\left\vert s\right\vert \int_{0}^{1}\frac{1}{(2a-\theta s)}d\theta
\leq CM.
\end{eqnarray*}
\end{proof}

\subsubsection{Operator continuity}

Now we prove continuity estimate.

\begin{lemma}
\label{prol}Let \textbf{D}$\left( \kappa ,l\right) $ and \textbf{B}$\left(
\kappa ,l\right) $ hold for $\mu \in \mathfrak{A}^{\sigma }$ with scaling
function $\kappa $ and scaling factor $l$. Let $\pi \in \mathfrak{A}%
_{sign}^{\sigma }$. Assume 
\begin{equation*}
\int_{\left\vert y\right\vert \leq 1}\left\vert y\right\vert ^{\alpha _{1}}d%
\widetilde{\left\vert \pi \right\vert }_{R}+\int_{\left\vert y\right\vert
>1}\left\vert y\right\vert ^{\alpha _{2}}d\widetilde{\left\vert \pi
\right\vert }_{R}\leq M,R>0
\end{equation*}%
and%
\begin{equation*}
\int_{1}^{\infty }\frac{1}{\gamma \left( r\right) ^{1\wedge \alpha _{2}}}%
\frac{dr}{r}<\infty ,
\end{equation*}%
where $\alpha _{1},\alpha _{2}$ are exponents in \textbf{B}$\left( \kappa
,l\right) $. Then for each $p\in \left( 1,\infty \right) $ there is a
constant $C=C\left( d,p,\kappa ,l,N_{0}\right) $ such that%
\begin{equation*}
\left\vert L^{\pi }v\right\vert _{L_{p}}\leq CM\left\vert L^{\mu
}v\right\vert _{L_{p}},v\in \widetilde{C}^{\infty }\left( \mathbf{R}%
^{d}\right)
\end{equation*}
\end{lemma}

\begin{proof}
Let $\varepsilon >0$, $v\in \tilde{C}^{\infty }\left( \mathbf{R}^{d}\right)
,g=-\left( L^{\mu }-\varepsilon \right) v$. According to Corollary 5 in \cite%
{MPh},%
\begin{equation*}
v\left( x\right) =\int_{0}^{\infty }e^{-\varepsilon t}\mathbf{E}g\left(
x+Z_{t}^{\mu }\right) dt,x\in \mathbf{R}^{d},
\end{equation*}%
and%
\begin{equation*}
L^{\pi }v\left( x\right) =Hg\left( x\right) :=\int_{0}^{\infty
}e^{-\varepsilon t}L^{\pi }\mathbf{E}g\left( x+Z_{t}^{\mu }\right) dt.x\in 
\mathbf{R}^{d}.
\end{equation*}%
Consider 
\begin{eqnarray*}
H_{\varepsilon }g\left( x\right) &=&\int_{\varepsilon }^{\infty
}e^{-\varepsilon t}L^{\pi }\mathbf{E}g\left( x+Z_{t}^{\mu }\right) dt \\
&=&\int m_{\varepsilon }\left( x-y\right) g\left( y\right) dy,x\in \mathbf{R}%
^{d},
\end{eqnarray*}%
with%
\begin{equation*}
m_{\varepsilon }\left( x\right) =\int_{\varepsilon }^{\infty
}e^{-\varepsilon t}L^{\pi }p^{\mu ^{\ast }}\left( t,x\right) dt,x\in \mathbf{%
R}^{d}.
\end{equation*}%
We prove that for each $p\in \left( 1,\infty \right) $ there is $C$ so that 
\begin{equation}
\left\vert H_{\varepsilon }g\right\vert _{L_{p}}\leq C\left\vert
g\right\vert _{L_{p}},\varepsilon >0,g\in L_{p}\left( \mathbf{R}^{d}\right) .
\label{ff1}
\end{equation}%
Obviously,%
\begin{equation*}
\mathcal{F}\left( H_{\varepsilon }g\right) \left( \xi \right) =\hat{m}%
_{\varepsilon }\left( \xi \right) \hat{g}\left( \xi \right) ,\xi \in \mathbf{%
R}^{d},
\end{equation*}%
where%
\begin{equation*}
\hat{m}_{\varepsilon }\left( \xi \right) =\int_{\varepsilon }^{\infty }\psi
^{\pi }\left( \xi \right) \exp \left\{ \psi ^{\mu }\left( \xi \right)
t-\varepsilon t\right\} dt,\xi \in \mathbf{R}^{d}.
\end{equation*}%
By Lemma 7 in \cite{MPh}, $\left\vert \hat{m}_{\varepsilon }\left( \xi
\right) \right\vert \leq AM,\xi \in \mathbf{R}^{d},\varepsilon \geq 0,$ for
some $A>0$. According to Theorem 3 of Chapter 1 in \cite{stein1}, it is
enough to verify that

\begin{equation}
\int_{\left\vert x\right\vert \geq 3\left\vert s\right\vert }\left\vert
m_{\varepsilon }\left( x-s\right) -m_{\varepsilon }\left( x\right)
\right\vert dx\leq CM,\forall s\neq 0,  \label{fff2}
\end{equation}

i.e.

\begin{equation*}
B=\int_{\left\vert x\right\vert \geq 3\left\vert s\right\vert }\left\vert
\int_{\epsilon }^{\infty }e^{-\varepsilon t}\left[ L^{\pi }p^{\mu ^{\ast
}}\left( t,x-s\right) -L^{\pi }p^{\mu ^{\ast }}\left( t,x\right) \right]
dt\right\vert dx\leq C.
\end{equation*}

Obviously,%
\begin{align*}
B\leq & \int_{\left\vert x\right\vert \geq 3\left\vert s\right\vert
}\left\vert \int_{0}^{a^{-1}\left( s\right) }...\right\vert
dx+\int_{\left\vert x\right\vert \geq 3\left\vert s\right\vert }\left\vert
\int_{a^{-1}\left( s\right) }^{\infty }...\right\vert dx \\
=& A_{1}+A_{2}.
\end{align*}

By Lemma \ref{al2} and (\ref{pf1}),

\begin{eqnarray*}
A_{1} &\leq &2\int_{\left\vert x\right\vert \geq 2\left\vert s\right\vert
}\int_{0}^{a^{-1}\left( \left\vert s\right\vert \right) }\left\vert L^{\pi
}p^{\mu ^{\ast }}\left( t,x\right) \right\vert dtdx \\
&\leq &CM\int_{0}^{a^{-1}\left( \left\vert s\right\vert \right)
}t^{-1}\left\vert s\right\vert ^{-\alpha _{2}}a\left( t\right) ^{\alpha
_{2}}dt\leq CM\int_{0}^{a^{-1}\left( \left\vert s\right\vert \right) }t^{-1}%
\frac{a\left( t\right) ^{\alpha _{2}}}{a\left( a^{-1}\left( \left\vert
s\right\vert \right) \right) ^{\alpha _{2}}}dt \\
&\leq &CM\int_{0}^{a^{-1}\left( \left\vert s\right\vert \right)
}t^{-1}\gamma \left( \frac{a^{-1}\left( \left\vert s\right\vert \right) }{t}%
\right) ^{-\alpha _{2}}dt\leq CM\int_{1}^{\infty }\gamma \left( r\right)
^{-\alpha _{2}}\frac{dr}{r}\leq CM.
\end{eqnarray*}

We estimate $A_{2}$ using Lemma \ref{al2} and (\ref{pf1}):

\begin{eqnarray*}
A_{2} &=&\int_{\left\vert x\right\vert \geq 3\left\vert s\right\vert
}\left\vert \int_{a^{-1}\left( s\right) }^{\infty }\left[ L^{\pi }p^{\mu
^{\ast }}\left( t,x-s\right) -L^{\pi }p^{\mu ^{\ast }}\left( t,x\right) %
\right] \right\vert dx \\
&\leq &\int_{a^{-1}\left( s\right) }^{\infty }\left\vert s\right\vert
\int_{0}^{1}\int \left\vert L^{\pi }\nabla p^{\mu ^{\ast }}\left( t,x-\tau
s\right) \right\vert d\tau dx \\
&\leq &CM\int_{a^{-1}\left( s\right) }^{\infty }\frac{a\left( t\right)
^{-1}\left\vert s\right\vert }{t}dt\leq CM\int_{a^{-1}\left( s\right)
}^{\infty }\frac{a\left( a^{-1}\left( \left\vert s\right\vert \right)
+\right) }{ta\left( t\right) }dt \\
&\leq &CM\int_{a^{-1}\left( s\right) }^{\infty }\gamma \left( \frac{t}{%
a^{-1}\left( \left\vert s\right\vert \right) }\right) ^{-1}\frac{dt}{t}\leq
CM\int_{1}^{\infty }\gamma \left( r\right) ^{-1}\frac{dr}{r}<\infty \text{.}
\end{eqnarray*}

Thus (\ref{fff2}) holds.
\end{proof}

\subsection{Existence and uniqueness for smooth input functions}

For $E=[0,T]\times \mathbf{R}^{d},p\geq 1$, we denote by $\tilde{C}^{\infty
}(E)$ the space of all measurable functions $f$ on $E$ such that for any
multiindex $\gamma \in \mathbf{N}_{0}^{d}$ and all $p\geq 1,$%
\begin{equation*}
\sup_{\left( t,x\right) \in E}\left\vert D^{\gamma }f\left( t,x\right)
\right\vert +\sup_{t\in \lbrack 0,T]}\left\vert D^{\gamma }f(t,\cdot
)\right\vert _{L_{p}\left( R^{d}\right) }<\infty .
\end{equation*}%
Similar space of functions on $\mathbf{R}^{d}$ is denoted $\tilde{C}^{\infty
}\left( \mathbf{R}^{d}\right) $.

Next we suppose that $f\in \tilde{C}^{\infty }(E),g\in \tilde{C}^{\infty
}\left( \mathbf{R}^{d}\right) $ and derive some estimates for the solution.

\begin{lemma}
\label{le0}Let $f\in \tilde{C}^{\infty }(E),g\in \tilde{C}^{\infty }\left( 
\mathbf{R}^{d}\right) $ then there is unique $u\in \tilde{C}^{\infty }\left(
E\right) $ solving (\ref{1'}). Moreover, 
\begin{equation}
u(t,x)=e^{-\lambda t}\mathbf{E}g\left( x+Z_{t}^{\pi }\right)
+\int_{0}^{t}e^{-\lambda (t-s)}\mathbf{E}f\left( s,x+Z_{t-s}^{\pi }\right)
ds,\left( t,x\right) \in E,  \label{h4}
\end{equation}%
and for $p\in \lbrack 1,\infty ]$ and any mutiindex $\gamma \in \mathbf{N}%
_{0}^{d},$%
\begin{eqnarray}
\left\vert D^{\gamma }u\right\vert _{L_{p}\left( E\right) } &\leq &\rho
_{\lambda }\left\vert D^{\gamma }f\right\vert _{L_{p}\left( E\right) }+\rho
_{\lambda }^{1/p}\left\vert D^{\gamma }g\right\vert _{L_{p}\left( \mathbf{R}%
^{d}\right) },  \label{h5} \\
\left\vert D^{\gamma }u\left( t\right) \right\vert _{L_{p}\left( \mathbf{R}%
^{d}\right) } &\leq &\left\vert D^{\gamma }g\right\vert _{L_{p}\left( 
\mathbf{R}^{d}\right) }+\int_{0}^{t}\left\vert D^{\gamma }f\left( s\right)
\right\vert _{L_{p}\left( \mathbf{R}^{d}\right) }ds,t\geq 0.  \label{h40}
\end{eqnarray}%
where $\rho _{\lambda }=\left( 1/\lambda \right) \wedge T$.
\end{lemma}

\begin{proof}
Let%
\begin{equation*}
h\left( t,x\right) =e^{-\lambda t}\mathbf{E}g\left( x+Z_{t}^{\pi }\right)
,\left( t,x\right) \in E.
\end{equation*}%
By Ito formula, $h$ solves (\ref{1'}) with $f=0.$ Obviously, for any
multiindex $\gamma ,$%
\begin{equation*}
D^{\gamma }h\left( t,x\right) =e^{-\lambda t}\mathbf{E}D^{\gamma }g\left(
x+Z_{t}^{\pi }\right) ,\left( t,x\right) \in E,
\end{equation*}%
and%
\begin{eqnarray*}
\sup_{t}\left\vert D^{\gamma }h\left( t\right) \right\vert _{L_{p}\left( 
\mathbf{R}^{d}\right) } &\leq &\left\vert D^{\gamma }g\right\vert
_{L_{p}\left( \mathbf{R}^{d}\right) }, \\
\left\vert D^{\gamma }h\right\vert _{L_{p}\left( E\right) } &\leq &\rho
_{\lambda }^{1/p}\left\vert D^{\gamma }g\right\vert _{L_{p}\left( \mathbf{R}%
^{d}\right) }.
\end{eqnarray*}%
\emph{\ }The claim follows by Lemma 8 in \cite{MPh}.
\end{proof}

\begin{remark}
\label{ren1}Using different notation, we rewrite (\ref{h4}) as%
\begin{equation}
u\left( t,x\right) =I_{\lambda }g\left( t,x\right) +R_{\lambda }f\left(
t,x\right) ,\left( t,x\right) \in E,  \label{h6}
\end{equation}%
where 
\begin{eqnarray}
I_{\lambda }g\left( t,x\right) &=&e^{-\lambda t}\mathbf{E}g\left(
x+Z_{t}^{\pi }\right) ,\left( t,x\right) \in E,  \label{h7} \\
R_{\lambda }f\left( t,x\right) &=&\int_{0}^{t}e^{-\lambda (t-s)}\mathbf{E}%
f\left( s,x+Z_{t-s}^{\pi }\right) ds,\left( t,x\right) \in E.  \notag
\end{eqnarray}
\end{remark}

\subsection{Estimate of \thinspace $R_{\protect\lambda }f$}

We prove that, under assumptions of Theorem \ref{t1}, there is $C=C\left(
n_{0},c_{1},N_{0},d,p\right) $ so that 
\begin{equation}
\left\vert L^{\mu }R_{\lambda }f\right\vert _{H_{p}^{\mu ;s}\left( E\right)
}\leq C\left\vert f\right\vert _{H_{p}^{\mu ;s}\left( E\right) }.
\label{h70}
\end{equation}%
Since $J_{\mu }^{t}:H_{p}^{\mu ;s}\left( \mathbf{R}^{d}\right) \rightarrow
H_{p}^{\mu ;s-t}\left( \mathbf{R}^{d}\right) $ is an isomorphism for any $%
s,t\in \mathbf{R}$, it is nought to derive the estimate for $s=0.$ We prove
that that there is $C=C\left( n_{0},c_{1},N_{0},d,p\right) $ so that%
\begin{equation}
\left\vert L^{\mu _{sym}}R_{\lambda }f\right\vert _{L_{p}\left( E\right)
}\leq C\left\vert f\right\vert _{L_{p}\left( E\right) }.  \label{h8}
\end{equation}%
First (\ref{h8}) holds for $p=2.$ Indeed,%
\begin{equation*}
\mathcal{F}\left( L^{\mu }R_{\lambda }f\right) =\psi ^{\mu }\int_{0}^{t}\exp
\left\{ \psi ^{\pi }\left( \xi \right) \left( t-s\right) -\lambda
(t-s)\right\} \hat{f}\left( s,\cdot \right) ds,
\end{equation*}%
and the estimate follows by Fubini theorem, Plancherel identity and
Corollary 4 in \cite{MPh}. According to Calderon-Zygmund theorem (see
Theorem 5 in \cite{MPh}), (\ref{h8}) reduces to the verification of H\"{o}%
rmander condition which follows from the following statement.

\begin{lemma}
\label{mainl}Let \textbf{D}$\left( \kappa ,l\right) $ and \textbf{B}$\left(
\kappa ,l\right) $ hold for $\mu \in \mathfrak{A}^{\sigma }$ with scaling
function $\kappa $ and scaling factor $l$. Let $\pi \in \mathfrak{A}%
_{sign}^{\sigma }$. Assume 
\begin{equation*}
\int_{\left\vert y\right\vert \leq 1}\left\vert y\right\vert ^{\alpha _{1}}d%
\widetilde{\left\vert \pi \right\vert }_{R}+\int_{\left\vert y\right\vert
>1}\left\vert y\right\vert ^{\alpha _{2}}d\widetilde{\left\vert \pi
\right\vert }_{R}\leq M,R>0,
\end{equation*}%
and%
\begin{equation*}
\int_{1}^{\infty }\frac{dr}{r\gamma \left( r\right) ^{1\wedge \alpha _{2}}}%
<\infty ,
\end{equation*}%
where $\alpha _{1},\alpha _{2}$ are exponents in \textbf{B}$\left( \kappa
,l\right) $. Let%
\begin{equation*}
K_{\lambda }^{\epsilon }\left( t,x\right) =e^{-\lambda t}L^{\pi }p^{\mu
^{\ast }}\left( t,x\right) \chi _{\left[ \epsilon ,\infty \right] }\left(
t\right) ,t>0,x\in \mathbf{R}^{d},
\end{equation*}%
where $\mu ^{\ast }\left( dy\right) =\mu \left( -dy\right) $. There exist $%
C_{0}>1$ and $C$ so that%
\begin{equation}
\mathcal{I}=\int \chi _{Q_{C_{0}\delta }\left( 0\right) ^{c}}\left(
t,x\right) \left\vert K_{\lambda }^{\epsilon }\left( t-\tilde{s},x-\tilde{y}%
\right) -K_{\lambda }^{\epsilon }\left( t,x\right) \right\vert dxdt\leq CM
\label{f23}
\end{equation}%
for all $\left\vert \tilde{s}\right\vert \leq \kappa \left( \delta \right)
,\left\vert \tilde{y}\right\vert \leq \delta ,\delta >0$, where $%
Q_{C_{0}\delta }\left( 0\right) =\left( -\kappa \left( C_{0}\delta \right)
,\kappa \left( C_{0}\delta \right) \right) \times \left\{ x:\left\vert
x\right\vert \leq C_{0}\delta \right\} .$
\end{lemma}

\begin{proof}
Let $C_{0}>3$ and $3l\left( 1\right) l\left( C_{0}^{-1}\right) <1$. We split%
\begin{equation*}
\mathcal{I}=\int_{-\infty }^{2\left\vert \tilde{s}\right\vert }\int
...+\int_{2\left\vert \tilde{s}\right\vert }^{\infty }\int ...=\mathcal{I}%
_{1}+\mathcal{I}_{2}
\end{equation*}

Since $\kappa \left( C_{0}\delta \right) >3\kappa \left( \delta \right)
,\delta >0,$ it follows by Lemma \ref{al2}, denoting $k_{0}=C_{0}-1,$%
\begin{eqnarray*}
\left\vert \mathcal{I}_{1}\right\vert &\leq &C\int_{0}^{3\left\vert \tilde{s}%
\right\vert }\int_{\left\vert x\right\vert >k_{0}a\left( \left\vert \tilde{s}%
\right\vert \right) }\left\vert L^{\pi }p^{\mu ^{\ast }}\left( t,x\right)
\right\vert dxdt \\
&\leq &CM\int_{0}^{3\left\vert \tilde{s}\right\vert }t^{-1}\frac{a\left(
t\right) ^{\alpha _{2}}}{a\left( \left\vert \tilde{s}\right\vert \right)
^{\alpha _{2}}}dt\leq CM\int_{0}^{3\left\vert \tilde{s}\right\vert
}t^{-1}\gamma \left( \frac{\left\vert \tilde{s}\right\vert }{t}\right)
^{-\alpha _{2}}dt \\
&=&CM\int_{1/3}^{\infty }\gamma \left( r\right) ^{-\alpha _{2}}\frac{dr}{r}.
\end{eqnarray*}

Now,

\begin{eqnarray*}
\mathcal{I}_{2} &\leq &\int_{2\left\vert \tilde{s}\right\vert }^{\infty
}\int \chi _{Q_{C_{0}\delta }^{c}\left( 0\right) }\left\vert L^{\pi }p^{\mu
^{\ast }}\left( t-\tilde{s},x-\tilde{y}\right) -L^{\pi }p^{\mu \ast }\left(
t-\tilde{s},x\right) \right\vert dxdt \\
&&+\int_{2\left\vert \tilde{s}\right\vert }^{\infty }\int \chi
_{Q_{C_{0}\delta }^{c}\left( 0\right) }\left\vert L^{\pi }p^{\mu ^{\ast
}}\left( t-\tilde{s},x\right) -L^{\pi }p^{\mu \ast }\left( t,x\right)
\right\vert dxdt \\
&=&\mathcal{\mathcal{I}}_{2,1}+\mathcal{I}_{2,2}.
\end{eqnarray*}

We split the estimate of $\mathcal{I}_{2,1}$ into two cases.

Case 1. Assume $\left\vert \tilde{y}\right\vert \leq a\left( 2\left\vert 
\tilde{s}\right\vert \right) .$ Then, by Lemma \ref{mvt},

\begin{eqnarray*}
\mathcal{I}_{2,1} &\leq &CM\left\vert \tilde{y}\right\vert \int_{2\left\vert 
\tilde{s}\right\vert }^{\infty }\left( t-\tilde{s}\right) ^{-1}a\left(
\left\vert t-\tilde{s}\right\vert \right) ^{-1}dt \\
&=&CM\left\vert \tilde{y}\right\vert a\left( 2\left\vert \tilde{s}%
\right\vert \right) ^{-1}\int_{2\left\vert \tilde{s}\right\vert }^{\infty
}\left( t-\tilde{s}\right) ^{-1}\frac{a\left( 2\left\vert \tilde{s}%
\right\vert \right) }{a\left( \left\vert t-\tilde{s}\right\vert \right) }dt
\\
&\leq &CM\int_{1/2}^{\infty }\gamma \left( r\right) ^{-1}\frac{dr}{r}.
\end{eqnarray*}

Case 2. Assume $\left\vert \tilde{y}\right\vert >a\left( 2\left\vert \tilde{s%
}\right\vert \right) $, i.e. $\delta \geq \left\vert \tilde{y}\right\vert
>a\left( 2\left\vert \tilde{s}\right\vert \right) $ and $a^{-1}\left( \delta
\right) \geq a^{-1}\left( \left\vert \tilde{y}\right\vert \right)
>2\left\vert \tilde{s}\right\vert .$ We split 
\begin{equation*}
\mathcal{I}_{2,1}=\int_{2\left\vert \tilde{s}\right\vert }^{2\left\vert 
\tilde{s}\right\vert +a^{-1}\left( \left\vert \tilde{y}\right\vert \right)
}\int ...+\int_{2\left\vert \tilde{s}\right\vert +a^{-1}\left( \left\vert 
\tilde{y}\right\vert \right) }^{\infty }\int ...=\mathcal{I}_{2,1,1}+%
\mathcal{I}_{2,1,2}.
\end{equation*}

If $2\left\vert \tilde{s}\right\vert \leq t\leq 2\left\vert \tilde{s}%
\right\vert +a^{-1}\left( \left\vert \tilde{y}\right\vert \right) $, then $%
0\leq t\leq 3a^{-1}\left( \delta \right) \leq 3l\left( 1\right) \kappa
\left( \delta \right) \leq \kappa \left( C_{0}\delta \right) $. Hence $%
\left\vert x\right\vert >C_{0}\delta \geq a\left( 2\left\vert \tilde{s}%
\right\vert \right) +\left\vert \tilde{y}\right\vert $ and%
\begin{eqnarray*}
\left\vert x-\tilde{y}\right\vert &\geq &\left( C_{0}-1\right) \delta
=k_{0}\delta \geq \frac{k_{0}}{2}[a\left( 2\left\vert \tilde{s}\right\vert
\right) +\left\vert \tilde{y}\right\vert ] \\
&\geq &a\left( 2\left\vert \tilde{s}\right\vert \right) +\left\vert \tilde{y}%
\right\vert \text{ if }\left( t,x\right) \notin Q_{C_{0}\delta }\left(
0\right) .
\end{eqnarray*}%
Also,%
\begin{equation}
2\geq \frac{2\left\vert \tilde{s}\right\vert +a^{-1}\left( \left\vert \tilde{%
y}\right\vert \right) }{2\left\vert \tilde{s}\right\vert +a^{-1}\left(
\left\vert \tilde{y}\right\vert \right) -\tilde{s}}\geq \frac{2}{3},
\label{af1}
\end{equation}%
and, by (\ref{pf1}),%
\begin{eqnarray}
&&\frac{a\left( 2\left\vert \tilde{s}\right\vert +a^{-1}(\left\vert \tilde{y}%
\right\vert )\right) }{a\left( 2\left\vert \tilde{s}\right\vert \right)
+\left\vert \tilde{y}\right\vert }  \label{af2} \\
&\leq &\frac{a\left( 2a^{-1}(\left\vert \tilde{y}\right\vert )\right) }{%
a\left( 2\left\vert \tilde{s}\right\vert \right) +\left\vert \tilde{y}%
\right\vert }\leq \gamma \left( 2^{-1}\right) ^{-1}\frac{a\left(
a^{-1}(\left\vert \tilde{y}\right\vert )\right) }{a\left( 2\left\vert \tilde{%
s}\right\vert \right) +\left\vert \tilde{y}\right\vert }\leq \gamma \left(
2^{-1}\right) ^{-1}.  \notag
\end{eqnarray}%
By Lemma \ref{al2}, (\ref{af1}), (\ref{af2}) and (\ref{pf1}),%
\begin{eqnarray*}
\mathcal{I}_{2,1,1} &\leq &C\int_{2\left\vert \tilde{s}\right\vert
}^{2\left\vert \tilde{s}\right\vert +a^{-1}\left( \left\vert \tilde{y}%
\right\vert \right) }\int_{\left\vert x\right\vert >a\left( 2\left\vert 
\tilde{s}\right\vert \right) +\left\vert \tilde{y}\right\vert }\left\vert
L^{\pi }p^{\mu ^{\ast }}\left( t-\tilde{s},x\right) \right\vert dtdx \\
&\leq &\frac{CMa\left( 2\left\vert \tilde{s}\right\vert +a^{-1}\left(
\left\vert \tilde{y}\right\vert \right) \right) ^{\alpha _{2}}}{[a\left(
2\left\vert \tilde{s}\right\vert \right) +\left\vert \tilde{y}\right\vert
]^{\alpha _{2}}}\int_{2\left\vert \tilde{s}\right\vert }^{2\left\vert \tilde{%
s}\right\vert +a^{-1}\left( \left\vert \tilde{y}\right\vert \right) }\left(
t-\tilde{s}\right) ^{-1}\frac{a\left( \left\vert t-\tilde{s}\right\vert
\right) ^{\alpha _{2}}}{a\left( 2\left\vert \tilde{s}\right\vert
+a^{-1}\left( \left\vert \tilde{y}\right\vert \right) \right) ^{\alpha _{2}}}%
dt \\
&\leq &CM\int_{2\left\vert \tilde{s}\right\vert }^{2\left\vert \tilde{s}%
\right\vert +a^{-1}\left( \left\vert \tilde{y}\right\vert \right) }\left( t-%
\tilde{s}\right) ^{-1}\gamma \left( \frac{2\left\vert \tilde{s}\right\vert
+a^{-1}\left( \left\vert \tilde{y}\right\vert \right) }{t-\tilde{s}}\right)
^{-\alpha _{2}}dt \\
&\leq &CM\int_{2/3}^{\infty }\gamma \left( r\right) ^{-\alpha _{2}}\frac{dr}{%
r}.
\end{eqnarray*}

Then, by Lemma \ref{mvt} and (\ref{af1})

\begin{eqnarray*}
&&\mathcal{I}_{2,1,2} \\
&\leq &\int_{2\left\vert \tilde{s}\right\vert +a^{-1}\left( \left\vert 
\tilde{y}\right\vert \right) }^{\infty }\left[ \int_{\mathbf{R}%
^{d}}\left\vert L^{\pi }p^{\mu ^{\ast }}\left( t-\tilde{s},x-\tilde{y}%
\right) -L^{\pi }p^{\mu \ast }\left( t-\tilde{s},x\right) \right\vert dx%
\right] dt \\
&=&CM\frac{\left\vert \tilde{y}\right\vert }{a\left( 2\left\vert \tilde{s}%
\right\vert +a^{-1}\left( \left\vert \tilde{y}\right\vert \right) \right) }%
\int_{2\left\vert \tilde{s}\right\vert +a^{-1}\left( \left\vert \tilde{y}%
\right\vert \right) }^{\infty }\left( t-\tilde{s}\right) ^{-1}\frac{a\left(
2\left\vert \tilde{s}\right\vert +a^{-1}\left( \left\vert \tilde{y}%
\right\vert \right) \right) }{a\left( t-\tilde{s}\right) }dr \\
&\leq &CM\int_{2\left\vert \tilde{s}\right\vert +a^{-1}\left( \left\vert 
\tilde{y}\right\vert \right) }^{\infty }\left( t-\tilde{s}\right)
^{-1}\gamma \left( \frac{t-\tilde{s}}{2\left\vert \tilde{s}\right\vert
+a^{-1}\left( \left\vert \tilde{y}\right\vert \right) }\right) ^{-1}dt\leq
CM\int_{1/2}^{\infty }\gamma \left( r\right) ^{-1}\frac{dr}{r},
\end{eqnarray*}%
because%
\begin{equation*}
\frac{\left\vert \tilde{y}\right\vert }{a\left( 2\left\vert \tilde{s}%
\right\vert +a^{-1}\left( \left\vert \tilde{y}\right\vert \right) \right) }%
\leq \frac{a\left( a^{-1}\left( \left\vert \tilde{y}\right\vert \right)
+\right) }{a\left( 2\left\vert \tilde{s}\right\vert +a^{-1}\left( \left\vert 
\tilde{y}\right\vert \right) \right) }\text{ }\leq 1.
\end{equation*}

Hence, $\mathcal{I}_{2,1}\leq C$.

Finally, by Lemma \ref{mvt}, 
\begin{eqnarray*}
\mathcal{I}_{2,2} &=&\int_{2\left\vert \tilde{s}\right\vert }^{\infty }\int
\chi _{Q_{C_{0}\delta }^{c}\left( 0\right) }\left\vert L^{\pi }p^{\mu ^{\ast
}}\left( t-\tilde{s},x\right) -L^{\pi }p^{\mu \ast }\left( t,x\right)
\right\vert dxdt \\
&\leq &CM.
\end{eqnarray*}

The proof is complete
\end{proof}

\subsection{Estimate of $I_{\protect\lambda }g$}

We prove that there is $C=C\left( n_{0},c_{1},N_{0},d,p\right) $ so that 
\begin{equation}
\left\vert L^{\mu }I_{\lambda }g\right\vert _{H_{p}^{\mu ;s}\left( E\right)
}\leq C\left\vert g\right\vert _{B_{pp}^{\mu ,N;s+1-1/p}\left( \mathbf{R}%
^{d}\right) }.  \label{h80}
\end{equation}%
Since $J_{\mu }^{t}:H_{p}^{\mu ;s}\left( \mathbf{R}^{d}\right) \rightarrow
H_{p}^{\mu ;s-t}\left( \mathbf{R}^{d}\right) $ and $J_{\mu }^{t}:B_{pp}^{\mu
,N;s}\left( \mathbf{R}^{d}\right) \rightarrow B_{pp}^{\mu ,N;s-t}\left( 
\mathbf{R}^{d}\right) $ is an isomorphism for any $s,t\in \mathbf{R}$, it is
enough to derive the estimate for $s=0.$ We prove that that there is $%
C=C\left( n_{0},c_{1},N_{0},d,p\right) $ so that%
\begin{equation}
\left\vert L^{\mu }I_{\lambda }g\right\vert _{L_{p}\left( E\right) }\leq
C\left\vert g\right\vert _{B_{pp}^{\mu ,N;1-1/p}\left( E\right) }.
\label{h9}
\end{equation}%
We will use an equivalent norm. Let $N>1$ be an integer, \thinspace $l\left(
N^{-1}\right) <1$. There exists a function $\phi \in C_{0}^{\infty }(\mathbf{%
R}^{d})$ (see Remark \ref{rem3}) such that $\mathrm{supp}\,\phi =\{\xi :%
\frac{1}{N}\leqslant |\xi |\leqslant N\}$, $\phi (\xi )>0$ if $N^{-1}<|\xi
|<N$ and 
\begin{equation*}
\sum_{j=-\infty }^{\infty }\phi (N^{-j}\xi )=1\quad \text{if }\xi \neq 0.
\end{equation*}%
Let%
\begin{equation*}
\tilde{\phi}\left( \xi \right) =\phi \left( N\xi \right) +\phi \left( \xi
\right) +\phi \left( N^{-1}\xi \right) ,\xi \in \mathbf{R}^{d}.
\end{equation*}%
Note that supp$~\tilde{\phi}\subseteq \left\{ N^{-2}\leq \left\vert \xi
\right\vert \leq N^{2}\right\} $ and $\tilde{\phi}\phi =\phi $. Let $\varphi
_{k}=\mathcal{F}^{-1}\phi \left( N^{-k}\cdot \right) ,k\geq 1,$ and $\varphi
_{0}\in \mathcal{S}\left( \mathbf{R}^{d}\right) $ is defined as%
\begin{equation*}
\varphi _{0}=\mathcal{F}^{-1}\left[ 1-\sum_{k=1}^{\infty }\phi \left(
N^{-k}\cdot \right) \right] .
\end{equation*}%
Let $\phi _{0}\left( \xi \right) =\mathcal{F}\varphi _{0}\left( \xi \right) ,%
\tilde{\phi}_{0}\left( \xi \right) =\mathcal{F}\varphi _{0}\left( \xi
\right) +\mathcal{F\varphi }_{1}\left( \xi \right) ,\xi \in \mathbf{R}^{d}%
\mathbf{,}\tilde{\varphi}=\mathcal{F}^{-1}\tilde{\phi},\varphi =\mathcal{F}%
^{-1}\phi $. Let 
\begin{equation*}
\tilde{\varphi}_{k}=\sum_{l=-1}^{1}\varphi _{k+l},k\geq 1,\tilde{\varphi}%
_{0}=\varphi _{0}+\varphi _{1}.
\end{equation*}%
Note that $\varphi _{k}=\tilde{\varphi}_{k}\ast \varphi _{k},k\geq 0$.
Obviously, $g=\sum_{k=0}^{\infty }g\ast \varphi _{k}$ in $\mathcal{S}%
^{\prime }\left( \mathbf{R}^{d}\right) $ for $g\in \mathcal{S}\left( \mathbf{%
R}^{d}\right) .$ For $j\geq 1,$ 
\begin{eqnarray*}
&&\mathcal{F}\left[ L^{\mu }I_{\lambda }g\left( t,\cdot \right) \ast \varphi
_{j}\right]  \\
&=&\kappa \left( N^{-j}\right) ^{-1}\psi ^{\tilde{\mu}_{N^{-j}}}\left(
N^{-j}\xi \right) \exp \left\{ \kappa \left( N^{-j}\right) ^{-1}\psi ^{%
\tilde{\pi}_{N^{-j}}}\left( N^{-j}\xi \right) t-\lambda t\right\}  \\
&&\times \tilde{\phi}\left( N^{-j}\xi \right) \hat{g}_{j}\left( \xi \right) ,
\end{eqnarray*}%
and%
\begin{equation*}
\mathcal{F}\left[ L^{\mu }I_{\lambda }g\left( t,\cdot \right) \ast \varphi
_{0}\right] =\psi ^{\mu }\left( \xi \right) \exp \left\{ \psi ^{\pi }\left(
\xi \right) t-\lambda t\right\} \tilde{\phi}_{0}\left( \xi \right) \hat{g}%
_{0}\left( \xi \right) ,
\end{equation*}%
where $g_{j}=g\ast \varphi _{j},j\geq 0.$

Let $Z^{j}=Z^{\tilde{\pi}_{N^{-j}}},j\geq 1.$ Let $\bar{\phi}\in
C_{0}^{\infty }\left( \mathbf{R}^{d}\right) ,0\notin $supp$\left( \bar{\phi}%
\right) $ and $\bar{\phi}\tilde{\phi}=\tilde{\phi}$, $\bar{\eta}=\mathcal{F}%
^{-1}\bar{\phi}$. Denoting $\eta =\mathcal{F}^{-1}\tilde{\phi}_{0},$ we have%
\begin{eqnarray}
L^{\mu }I_{\lambda }g\left( t,\cdot \right) \ast \varphi _{j} &=&\kappa
\left( N^{-j}\right) ^{-1}\bar{H}_{t}^{\lambda ,j}\ast g_{j},j\geq 1,
\label{for0} \\
L^{\mu }I_{\lambda }g\left( t,\cdot \right) \ast \varphi _{0} &=&\bar{H}%
_{t}^{\lambda ,0}\ast g_{0},t>0,  \notag
\end{eqnarray}%
where for $j\geq 1,$%
\begin{eqnarray*}
\bar{H}_{t}^{\lambda ,j}\left( x\right) &=&N^{jd}H_{\kappa \left(
N^{-j}\right) ^{-1}t}^{\lambda ,j}\left( N^{j}x\right) ,\left( t,x\right)
\in E, \\
H_{t}^{\lambda ,j} &=&e^{-\lambda \kappa \left( N^{-j}\right) t}(L^{\tilde{%
\mu}_{N^{-j}}}\bar{\eta})\ast \mathbf{E}\tilde{\varphi}\left( \cdot
+Z_{t}^{j}\right) ,t>0\mathbf{,}
\end{eqnarray*}%
and%
\begin{equation*}
\bar{H}_{t}^{\lambda ,0}\left( x\right) =e^{-\lambda t}L^{\mu }\mathbf{E}%
\eta \left( \cdot +Z_{t}^{\pi }\right) ,\left( t,x\right) \in E.
\end{equation*}

By Corollary 2 in \cite{MPh},%
\begin{equation*}
\sup_{j}\int \left\vert L^{\tilde{\mu}_{N^{-j}}}\bar{\eta}\right\vert
dx<\infty \text{.}
\end{equation*}%
Hence by Lemma \ref{auxl1},%
\begin{eqnarray*}
\int \left\vert H_{t}^{\lambda ,j}\right\vert dx &\leq &\int \left\vert L^{%
\tilde{\mu}_{N^{-j}}}\bar{\eta}\right\vert dx\int \left\vert \mathbf{E}%
\tilde{\varphi}\left( \cdot +Z_{t}^{j}\right) \right\vert dx \\
&\leq &Ce^{-ct},t>0,j\geq 1,
\end{eqnarray*}%
and%
\begin{equation}
\int \left\vert \bar{H}_{t}^{\lambda ,j}\right\vert dx\leq C\exp \left\{
-c\kappa \left( N^{-j}\right) ^{-1}t\right\} ,t>0,j\geq 1.  \label{for1}
\end{equation}%
and by Lemma \ref{al2},%
\begin{equation}
\int \left\vert \bar{H}_{t}^{\lambda ,0}\right\vert dx\leq C\left( \frac{1}{t%
}\wedge 1\right) ,t>0.  \label{h11}
\end{equation}

It follows by Proposition \ref{pro2} and (\ref{for0}) that%
\begin{eqnarray*}
\left\vert L^{\mu }I_{\lambda }g\left( t\right) \right\vert _{L_{p}\left( 
\mathbf{R}^{d}\right) }^{p} &\leq &C\left\vert \left( \sum_{j=1}^{\infty
}\left\vert \kappa \left( N^{-j}\right) ^{-1}\bar{H}_{t}^{\lambda ,j}\ast
g_{j}\right\vert ^{2}\right) ^{1/2}\right\vert _{L_{p}\left( \mathbf{R}%
^{d}\right) }^{p} \\
&&+C\int \left\vert \bar{H}_{t}^{\lambda ,0}\ast g_{0}\right\vert ^{p}dx.
\end{eqnarray*}

Hence%
\begin{equation*}
\left\vert L^{\mu }I_{\lambda }g\left( t\right) \right\vert _{L_{p}\left( 
\mathbf{R}^{d}\right) }^{p}\leq C\sum_{j=0}^{\infty }\left\vert \kappa
\left( N^{-j}\right) ^{-1}\bar{H}_{t}^{\lambda ,j}\ast g_{j}\right\vert
_{L_{p}\left( \mathbf{R}^{d}\right) }^{p}\text{ if }p\in (1,2],
\end{equation*}%
and, by Minkowski inequality,%
\begin{eqnarray*}
\left\vert L^{\mu }I_{\lambda }g\left( t\right) \right\vert _{L_{p}\left( 
\mathbf{R}^{d}\right) }^{p} &\leq &C\left( \sum_{j=1}^{\infty }\left( \int
\left\vert \kappa \left( N^{-j}\right) ^{-1}\bar{H}_{t}^{\lambda ,j}\ast
g_{j}\right\vert ^{p}dx\right) ^{2/p}\right) ^{p/2} \\
&&+C\int \left\vert \bar{H}_{t}^{\lambda ,0}\ast g_{0}\right\vert ^{p}dx
\end{eqnarray*}%
if $p>2$. Now, by (\ref{for1}),%
\begin{eqnarray}
&&\int \left\vert \kappa \left( N^{-j}\right) ^{-1}\bar{H}_{t}^{\lambda
,j}\ast g_{j}\right\vert ^{p}dx  \label{h12} \\
&\leq &\left( \int \left\vert \bar{H}_{t}^{\lambda ,j}\right\vert dx\right)
^{p}\int \left\vert \kappa \left( N^{-j}\right) ^{-1}g_{j}\right\vert ^{p}dx
\notag \\
&\leq &C\kappa \left( N^{-j}\right) ^{-p}\exp \left\{ -c\kappa \left(
N^{-j}\right) ^{-1}t\right\} \left\vert g_{j}\right\vert _{L_{p}}^{p}\text{
if }j\geq 1,  \notag
\end{eqnarray}%
and, by (\ref{h11}),%
\begin{equation}
\int \left\vert \bar{H}_{t}^{\lambda ,0}\ast g_{0}\right\vert ^{p}dx\leq
C\left( \frac{1}{t}\wedge 1\right) ^{p}\int \left\vert g_{0}\right\vert
^{p}dx.  \label{h13}
\end{equation}%
Therefore for $p\in (1,2],$%
\begin{equation*}
\int_{0}^{\infty }\left\vert L^{\mu }I_{\lambda }g\left( t\right)
\right\vert _{L_{p}\left( \mathbf{R}^{d}\right) }^{p}dt\leq
C\sum_{j=0}^{\infty }\left\vert \kappa \left( N^{-j}\right)
^{-(1-1/p)}\left\vert g_{j}\right\vert _{L_{p}\left( \mathbf{R}^{d}\right)
}\right\vert ^{p},
\end{equation*}%
and (\ref{h9}) follows by Proposition \ref{pro1}.

Let $p>2.$ In this case,%
\begin{equation*}
\int_{0}^{\infty }\left\vert L^{\mu }I_{\lambda }g\left( t\right)
\right\vert _{L_{p}\left( \mathbf{R}^{d}\right) }^{p}dt\leq C[G+\left\vert
g_{0}\right\vert _{L_{p}\left( \mathbf{R}^{d}\right) }^{p}],
\end{equation*}%
where%
\begin{equation*}
G=\int_{0}^{\infty }\left( \sum_{j=1}^{\infty }\exp \left\{ -c\kappa \left(
N^{-j}\right) ^{-1}t\right\} k_{j}^{2}\right) ^{p/2}dt
\end{equation*}%
with $c>0$ and 
\begin{equation*}
k_{j}=\kappa \left( N^{-j}\right) ^{-1}\left\vert g_{j}\right\vert
_{L_{p}\left( \mathbf{R}^{d}\right) },j\geq 1.
\end{equation*}%
Now, let $B=\left\{ j:\kappa \left( N^{-j}\right) ^{-1}t\leq 1\right\} $.
Then 
\begin{equation*}
\sum_{j=1}^{\infty }e^{-c\kappa \left( N^{-j}\right)
^{-1}t}k_{j}^{2}=\sum_{j\in B}...+\sum_{j\notin B}...=D\left( t\right)
+E\left( t\right) ,t>0.
\end{equation*}%
Let $a\left( t\right) =\inf \left\{ t:\kappa \left( r\right) \geq t\right\}
,t>0,0<\frac{\beta p}{2}\leq \beta _{1}$. By H\"{o}lder inequality,%
\begin{eqnarray*}
D\left( t\right)  &\leq &C\sum_{j=1}^{\infty }\chi _{\left\{ j:\kappa \left(
N^{-j}\right) ^{-1}t\leq 1\right\} }\gamma \left( \kappa \left(
N^{-j}\right) ^{-1}t\right) ^{\beta }\gamma \left( \kappa \left(
N^{-j}\right) ^{-1}t\right) ^{-\beta }k_{j}^{2} \\
&\leq &C\left( \sum_{j=1}^{\infty }\chi _{\left\{ j:a\left( t\right) \leq
N^{-j}\right\} }\gamma \left( l\left( \frac{a\left( t\right) }{N^{-j}}%
\right) \right) ^{\beta \frac{p}{p-2}}\right) ^{1-\frac{2}{p}} \\
&&\times \left( \sum_{j=1}^{\infty }\chi _{\left\{ j:j:\kappa \left(
N^{-j}\right) ^{-1}t\leq 1\right\} }\gamma \left( \kappa \left(
N^{-j}\right) ^{-1}t\right) ^{-\beta \frac{p}{2}}k_{j}^{p}\right) ^{\frac{2}{%
p}} \\
&=&CD_{1}^{\frac{p-2}{p}}D_{2}^{\frac{2}{p}}.
\end{eqnarray*}%
Denoting $\beta ^{\prime }=\beta p/\left( p-2\right) ,$ we have for $t>0,$ 
\begin{eqnarray*}
D_{1}\left( t\right)  &=&\sum_{j}\chi _{\left\{ j:\frac{a\left( t\right) }{%
N^{-j}}\leq 1\right\} }\left( \frac{a\left( t\right) }{N^{-j}}\right)
^{\beta ^{\prime }} \\
&\leq &C\int_{0}^{\infty }\chi _{\left\{ \frac{a\left( t\right) }{N^{-x}}%
\leq 1\right\} }\left( \frac{a\left( t\right) }{N^{-x}}\right) ^{\beta
^{\prime }}dx\leq C\int_{0}^{1}y^{\beta ^{\prime }}\frac{dy}{y}<\infty .
\end{eqnarray*}%
Hence%
\begin{eqnarray*}
&&\int_{0}^{\infty }D^{p/2}dt\leq C\sum_{j}\int_{0}^{\infty }\chi _{\left\{
j:j:\kappa \left( N^{-j}\right) ^{-1}t\leq 1\right\} }\gamma \left( \kappa
\left( N^{-j}\right) ^{-1}t\right) ^{-\beta \frac{p}{2}}k_{j}^{p}dt \\
&=&C\sum_{j}\int_{0}^{1}\gamma \left( t\right) ^{-\beta \frac{p}{2}}dt\kappa
\left( N^{-j}\right) k_{j}^{p}.
\end{eqnarray*}

Now, we estimate the second term \thinspace $E\left( t\right) ,t>0$. By H%
\"{o}lder inequality, for $t>0$,%
\begin{eqnarray*}
E\left( t\right)  &=&\sum_{\kappa \left( N^{-j}\right) ^{-1}t>1}e^{-c\kappa
\left( N^{-j}\right) ^{-1}t}\kappa \left( N^{-j}\right) ^{-2}\left\vert
g_{j}\right\vert _{L_{p}}^{2} \\
&\leq &\left( \sum_{\kappa \left( N^{-j}\right) ^{-1}t\geq 1}e^{-c\kappa
\left( N^{-j}\right) ^{-1}t}\right) ^{\frac{p-2}{p}}\left( \sum_{\kappa
\left( N^{-j}\right) ^{-1}t\geq 1}e^{-c\kappa \left( N^{-j}\right)
^{-1}t}k_{j}^{p}\right) ^{\frac{2}{p}}.
\end{eqnarray*}%
Since%
\begin{equation*}
l\left( N^{j}a\left( t\right) \right) \geq \kappa \left( N^{-j}\right)
^{-1}t\geq l\left( \frac{1}{N^{j}a\left( t\right) }\right) ^{-1},t>0,
\end{equation*}%
it follows by changing the variable of integration, $y=\frac{1}{a\left(
t\right) N^{x}},$%
\begin{eqnarray*}
&&\sum_{\kappa \left( N^{-j}\right) ^{-1}t\geq 1}e^{-c\kappa \left(
N^{-j}\right) ^{-1}t} \\
&\leq &\sum_{j=1}^{\infty }\chi _{\left\{ N^{j}a\left( t\right) \geq \gamma
\left( 1\right) \right\} }\exp \left\{ -cl\left( \frac{1}{N^{j}a\left(
t\right) }\right) ^{-1}\right\}  \\
&\leq &1+\int_{0}^{\infty }\chi _{\left\{ N^{x}a\left( t\right) \geq \gamma
\left( 1\right) \right\} }\exp \left\{ -cl\left( \frac{1}{N^{x}a\left(
t\right) }\right) ^{-1}\right\} dx \\
&\leq &1+C\int_{0}^{\gamma \left( 1\right) ^{-1}}l\left( y\right) ^{\beta
_{2}}\frac{dy}{y}.
\end{eqnarray*}%
Hence%
\begin{eqnarray*}
\int_{0}^{\infty }E\left( t\right) ^{p/2}dt &\leq &C\int_{0}^{\infty
}\sum_{\kappa \left( N^{-j}\right) ^{-1}t\geq 1}e^{-c\kappa \left(
N^{-j}\right) ^{-1}t}k_{j}^{p}dt \\
&\leq &C\sum_{j}\kappa \left( N^{-j}\right) k_{j}^{p}\text{.}
\end{eqnarray*}%
The estimate (\ref{h9}) is proved.

\subsection{Proof of Theorem \protect\ref{t1}}

We finish the proof of Theorem \ref{t1} in a standard way. Since $J_{\mu
}^{t}:H_{p}^{\mu ;s}\left( \mathbf{R}^{d}\right) \rightarrow H_{p}^{\mu
;s-t}\left( \mathbf{R}^{d}\right) $ and $J_{\mu }^{t}:B_{pp}^{\mu
,N;s}\left( \mathbf{R}^{d}\right) \rightarrow B_{pp}^{\mu ,N;s-t}\left( 
\mathbf{R}^{d}\right) $ is an isomorphism for any $s,t\in \mathbf{R}$, it is
enough to derive the statement for $s=0.$ Let $f\in L_{p}\left( E\right)
,g\in B_{pp}^{\mu ,N;1-1/p}\left( \mathbf{R}^{d}\right) $. There are
sequences $f_{n}\in \tilde{C}^{\infty }\left( E\right) ,g_{n}\in \tilde{C}%
^{\infty }\left( \mathbf{R}^{d}\right) $ such that 
\begin{equation*}
f_{n}\rightarrow f\text{ in }L_{p}\left( E\right) ,g_{n}\rightarrow g\text{
in }B_{pp}^{\mu ,N;1-1/p}\left( \mathbf{R}^{d}\right) .
\end{equation*}%
For each $n,$ there is unique $u_{n}\in \tilde{C}^{\infty }\left( E\right) $
solving (\ref{1'}). Hence%
\begin{eqnarray*}
\partial _{t}\left( u_{n}-u_{m}\right)  &=&\left( L^{\pi }-\lambda \right)
\left( u_{n}-u_{m}\right) +f_{n}-f_{m}, \\
u_{n}\left( 0,x\right) -u_{n}\left( 0,x\right)  &=&g_{n}\left( x\right)
-g_{m}\left( x\right) ,x\in \mathbf{R}^{d}.
\end{eqnarray*}%
By (\ref{h70}), (\ref{h80}) and Lemma \ref{le0}, 
\begin{eqnarray}
&&\left\vert L^{\mu }\left( u_{n}-u_{m}\right) \right\vert _{L_{p}\left(
E\right) }  \label{fo6} \\
&\leq &C\left( \left\vert f_{n}-f_{m}\right\vert _{L_{p}\left( E\right)
}+\left\vert g_{n}-g_{m}\right\vert _{B_{pp}^{\mu ,N;1-1/p}\left( \mathbf{R}%
^{d}\right) }\right) \rightarrow 0,  \notag \\
&&\left\vert u_{n}-u_{m}\right\vert _{L_{p}\left( E\right) }  \notag \\
&\leq &\rho _{\lambda }\left\vert f_{n}-f_{m}\right\vert _{L_{p}\left(
E\right) }+\rho _{\lambda }^{1/p}\left\vert g_{n}-g_{m}\right\vert
_{L_{p}\left( \mathbf{R}^{d}\right) }\rightarrow 0,  \notag
\end{eqnarray}%
as $n,m\rightarrow \infty $. Hence there is $u\in H_{p}^{\mu ;1}\left(
E\right) $ so that $u_{n}\rightarrow u$ in $H_{p}^{\mu ;1}\left( E\right) $.
Moreover, by Lemma \ref{le0},%
\begin{equation}
\sup_{t\leq T}\left\vert u_{n}\left( t\right) -u\left( t\right) \right\vert
_{L_{p}\left( \mathbf{R}^{d}\right) }\rightarrow 0,  \label{fo7}
\end{equation}%
and, according to Lemma \ref{prol}, 
\begin{equation}
\left\vert L^{\pi }f\right\vert _{L_{p}\left( E\right) }\leq C\left\vert
L^{\mu }f\right\vert _{L_{p}\left( E\right) },f\in \tilde{C}^{\infty }\left(
E\right) .  \label{fo8}
\end{equation}%
Hence (see (\ref{fo6})-(\ref{fo8})) we can pass to the limit in the equation%
\begin{equation}
u_{n}\left( t\right) =g_{n}+\int_{0}^{t}[L^{\pi }u_{n}\left( s\right)
-\lambda u_{n}\left( s\right) +f_{n}\left( s\right) ]ds,0\leq t\leq T.
\label{fo9}
\end{equation}%
Obviously, (\ref{fo9}) holds for $u,g$ and $f$. We proved the existence part
of Theorem \ref{t1}.

\emph{Uniqueness. }Assume $u_{1},u_{2}\in H_{p}^{\mu ;1}\left( E\right) $
solve (\ref{1'})$.$ Then $u=u_{1}-u_{2}\in H_{p}^{\mu ;1}\left( E\right) $
solves (\ref{1'}) with $f=0,g=0$. Now, let $\varphi \in \tilde{C}^{\infty
}\left( E\right) $, and $\tilde{\varphi}\left( t,x\right) =\varphi \left(
T-t,x\right) ,\left( t,x\right) \in E$. By Lemma \ref{le0}, there is unique $%
\tilde{v}\in \tilde{C}^{\infty }\left( E\right) $ solving (\ref{1'}) with $f=%
\tilde{\varphi},g=0$ and $\pi ^{\ast }$ instead of $\pi $. Let $v\left(
t,x\right) =\tilde{v}\left( T-t,x\right) ,\left( t,x\right) \in E$. Then $%
\partial _{t}v+L^{\pi ^{\ast }}v-\lambda v+\varphi =0$ in $E$ and $v\left(
T\right) =v\left( T,\cdot \right) =0$. Integrating by parts,%
\begin{eqnarray*}
\int_{E}\varphi u &=&\int_{E}u\left( -\partial _{t}v-L^{\pi ^{\ast
}}v+\lambda v\right) \\
&=&\int_{E}v\left( \partial _{t}u-L^{\pi }u+\lambda u\right) =0.
\end{eqnarray*}%
Hence $\int_{E}u\varphi ~dtdx=0$ $\forall \varphi \in \tilde{C}^{\infty
}\left( E\right) $. Hence $u=0$ a.e. Theorem \ref{t1} is proved.

\section{Appendix}

We will need the following Levy process moment estimate.

\begin{lemma}
\label{al00}Let $\pi \in \mathfrak{A}^{\sigma }.$ Assume 
\begin{equation}
\int_{\left\vert z\right\vert \leq 1}\left\vert z\right\vert ^{\alpha
_{1}}\pi (dz)+\int_{\left\vert z\right\vert >1}\left\vert z\right\vert
^{\alpha _{2}}\pi (dz)\leq M,  \label{fa5}
\end{equation}%
where $\alpha _{1},\alpha _{2}\in (0,1]\text{ if }\sigma \in (0,1)\text{; }%
\alpha _{1},\alpha _{2}\in (1,2]\text{ if }\sigma \in (1,2)$; $\alpha
_{1}\in (1,2]$\ and $\alpha _{2}\in \lbrack 0,1)$\ if $\sigma =1$. Let $%
\zeta _{t}$ be the Levy process associated to $\psi ^{\pi }$, that is%
\begin{equation*}
\mathbf{E}e^{i2\pi \xi \cdot \zeta _{t}}=\exp \{\psi \left( \xi \right)
t\},t\geq 0
\end{equation*}%
There is a constant $C=C\left( M\right) $ such that%
\begin{equation*}
\mathbf{E}\left[ \left\vert \zeta _{t}\right\vert ^{\alpha _{2}}\right] \leq
C\left( 1+t\right) ,t\geq 0.
\end{equation*}
\end{lemma}

\begin{proof}
Recall 
\begin{equation}
\zeta _{t}=\int_{0}^{t}\int \chi _{\sigma }(y)yq(ds,dy)+\int_{0}^{t}\int
(1-\chi _{\sigma }(y))yp(ds,dy),t\geq 0,  \label{fa3}
\end{equation}%
$p(ds,dy)$ is Poisson point measure with 
\begin{equation*}
\mathbf{E}p\left( ds,dy\right) =\pi \left( dy\right) ds,q\left( ds,dy\right)
=p\left( ds,dy\right) -\pi \left( dy\right) ds.
\end{equation*}%
Now, $\zeta _{t}=\bar{\zeta}_{t}+\tilde{\zeta}_{t}$ with%
\begin{eqnarray*}
\bar{\zeta}_{t} &=&\int_{0}^{t}\int_{\left\vert y\right\vert \leq 1}\chi
_{\sigma }(y)yq(ds,dy)+\int_{0}^{t}\int_{\left\vert y\right\vert \leq
1}(1-\chi _{\sigma }(y))yp(ds,dy), \\
\tilde{\zeta}_{t} &=&\int_{0}^{t}\int_{\left\vert y\right\vert >1}\chi
_{\sigma }(y)yq(ds,dy)+\int_{0}^{t}\int_{\left\vert y\right\vert >1}(1-\chi
_{\sigma }(y))yp(ds,dy),t\geq 0.
\end{eqnarray*}

\emph{Case 1: }$\sigma \in \left( 0,1\right) $. In this case (\ref{fa5})
holds with $\alpha _{1},\alpha _{2}\in (0,1]$. Then for any $t>0,$%
\begin{equation*}
\mathbf{E}\left\vert \bar{\zeta}_{t}\right\vert \leq t\int_{\left\vert
y\right\vert \leq 1}\left\vert y\right\vert ^{\alpha _{1}}\pi \left(
dy\right) \leq Ct,
\end{equation*}%
and%
\begin{equation*}
\left\vert \tilde{\zeta}_{t}\right\vert ^{\alpha _{2}}=\sum_{s\leq t}\left[
\left\vert \tilde{\zeta}_{s}\right\vert ^{\alpha _{2}}-\left\vert \tilde{%
\zeta}_{s-}\right\vert ^{\alpha _{2}}\right] \leq
\int_{0}^{t}\int_{\left\vert y\right\vert >1}\left\vert y\right\vert
^{\alpha _{2}}p\left( ds,dy\right)
\end{equation*}%
implies that $\mathbf{E}\left\vert \tilde{\zeta}_{t}\right\vert ^{\alpha
_{2}}\leq Ct.$

\emph{Case 2: }$\sigma \in \left( 1,2\right) $.\ In this case, $\alpha
_{1},\alpha _{2}\in (1,2]$. Then%
\begin{equation*}
\mathbf{E[}\left\vert \bar{\zeta}_{t}\right\vert ^{2}]=\int_{\left\vert
y\right\vert \leq 1}\left\vert y\right\vert ^{2}\pi \left( dy\right) t\leq
Ct,
\end{equation*}%
and, by BDG inequality, 
\begin{eqnarray*}
\mathbf{E[}\left\vert \tilde{\zeta}_{t}\right\vert ^{\alpha _{2}}] &\leq &C%
\mathbf{E}\left[ \left( \sum_{s\leq t}\left( \Delta \tilde{\zeta}_{s}\right)
^{2}\right) ^{\alpha _{2}/2}\right] \\
&\leq &C\mathbf{E}\left[ \sum_{s\leq t}\left( \Delta \tilde{\zeta}%
_{s}\right) ^{\alpha _{2}}\right] =Ct\int_{\left\vert y\right\vert
>1}\left\vert y\right\vert ^{\alpha _{2}}d\pi .
\end{eqnarray*}

\emph{Case 3: }$\sigma =1$. In this case, $\alpha _{1}\in (1,2]$\emph{\ and }%
$\alpha _{2}\in (0,1)$. Similarly as above, we find that%
\begin{eqnarray*}
\mathbf{E[}\left\vert \bar{\zeta}_{t}\right\vert ^{2}] &=&t\int_{\left\vert
y\right\vert \leq 1}\left\vert y\right\vert ^{2}\pi \left( dy\right) \leq Ct,
\\
\mathbf{E[}\left\vert \tilde{\zeta}_{t}\right\vert ^{\alpha _{2}}] &\leq &Ct.
\end{eqnarray*}%
The statement is proved.\newline
\end{proof}

We need the following Gaussian moments estimates as well.

\begin{lemma}
\label{gm}Let $a_{kj}\in \mathbf{R,}k,j\geq 0,$ and 
\begin{equation*}
\left\vert \left\vert a\right\vert \right\vert =\left( \sum_{k,j=0}^{\infty
}a_{kj}^{2}\right) ^{1/2}<\infty .
\end{equation*}%
and let $\zeta _{k},k\geq 0,$ be a sequence of independent standard normal
r.v. For $p\geq 1$ set 
\begin{equation*}
\xi =\left( \sum_{j=0}^{\infty }\left( \sum_{k=0}^{\infty }\zeta
_{k}a_{kj}\right) ^{2}\right) ^{p/2}.
\end{equation*}%
Then there are constants $0<c_{1}<c_{2}$ so that%
\begin{equation*}
c_{1}\left\vert \left\vert a\right\vert \right\vert ^{p}\leq \mathbf{E}\xi
\leq c_{2}\left\vert \left\vert a\right\vert \right\vert ^{p}.
\end{equation*}
\end{lemma}

\begin{proof}
Case1. Let $p\geq 2$. Since $\zeta _{k}$ are independent standard normal, by
Minkowski inequality,%
\begin{equation*}
\mathbf{E}\xi \leq \left( \sum_{j=0}^{\infty }\left[ \mathbf{E}\left(
\left\vert \sum_{k=0}^{\infty }\zeta _{k}a_{kj}\right\vert ^{p}\right) %
\right] ^{2/p}\right) ^{p/2}\leq C\left( \sum_{j=0}^{\infty
}\sum_{k=0}^{\infty }a_{kj}^{2}\right) ^{p/2}.
\end{equation*}%
On the other hand, by H\"{o}lder inequality,%
\begin{equation*}
\mathbf{E}\xi \geq \left( \mathbf{E}\sum_{j=0}^{\infty }\left(
\sum_{k=0}^{\infty }\zeta _{k}a_{kj}\right) ^{2}\right) ^{p/2}\geq c\left(
\sum_{j,k=0}^{\infty }a_{kj}^{2}\right) ^{p/2}.
\end{equation*}

Case 2. Let $p\in \lbrack 1,2)$. Then, by H\"{o}lder inequality,%
\begin{eqnarray*}
&&\mathbf{E}\left[ \left( \sum_{j=0}^{\infty }\left( \sum_{k=0}^{\infty
}\zeta _{k}a_{kj}\right) ^{2}\right) ^{p/2}\right]  \\
&\leq &\left( \mathbf{E}\sum_{j=0}^{\infty }\left( \sum_{k=0}^{\infty }\zeta
_{k}a_{kj}\right) ^{2}\right) ^{p/2}=\left( \sum_{j,k=0}^{\infty
}a_{kj}^{2}\right) ^{p/2}.
\end{eqnarray*}%
On the other hand, by H\"{o}lder and Minkowski inequality (recall $\zeta _{k}
$ are independent standard normal),%
\begin{eqnarray*}
&&\mathbf{E}\xi \geq \left[ \mathbf{E}\left( \sum_{j=0}^{\infty }\left\vert
\sum_{k=0}^{\infty }\zeta _{k}a_{kj}\right\vert ^{2}\right) ^{1/2}\right]
^{p} \\
&\geq &\left( \sum_{j,k=0}^{\infty }\left( \mathbf{E}\left\vert \zeta
_{k}a_{kj}\right\vert \right) ^{2}\right) ^{p/2}\geq c\left(
\sum_{j,k=0}^{\infty }a_{kj}^{2}\right) ^{p/2}.
\end{eqnarray*}
\end{proof}

\end{document}